\documentclass[9pt]{amsart}
\usepackage{amsmath}
\usepackage{amsxtra}
\usepackage{amscd}
\usepackage{amsthm}
\usepackage{amsfonts}
\usepackage{amssymb}
\usepackage{eucal}
\usepackage[all]{xy}
\usepackage{graphicx}
\usepackage[usenames]{color}
\usepackage[hidelinks]{hyperref}

\usepackage{tikz}
\usetikzlibrary{decorations.pathmorphing}
\usetikzlibrary{cd}

\newtheorem{cor}[subsubsection]{Corollary}

\newtheorem{lem}[subsubsection]{Lemma}
\newtheorem{prop}[subsubsection]{Proposition}

\newtheorem{conj}[subsubsection]{Conjecture}

\newtheorem{thm}[subsubsection]{Theorem}

\theoremstyle{remark}
\newtheorem{rem}[subsubsection]{Remark}
\newtheorem{example}[subsubsection]{Example}
\newtheorem{defin}[subsubsection]{Definition}

\theoremstyle{definition}

\theoremstyle{remark}

\newcommand{\thmref}[1]{Theorem~\ref{#1}}

\newcommand{\secref}[1]{Sect.~\ref{#1}}
\newcommand{\lemref}[1]{Lemma~\ref{#1}}
\newcommand{\propref}[1]{Proposition~\ref{#1}}
\newcommand{\corref}[1]{Corollary~\ref{#1}}
\newcommand{\conjref}[1]{Conjecture~\ref{#1}}
\newcommand{\remref}[1]{Remark~\ref{#1}}

\numberwithin{equation}{section}

\newcommand{\nc}{\newcommand}
\nc{\renc}{\renewcommand}
\nc{\ssec}{\subsection}
\nc{\sssec}{\subsubsection}
\nc{\on}{\operatorname}

\nc{\ips}{{\iota_P^{(S)}}}
\nc{\ipms}{{\iota_{P^-}^{(S)}}}
\nc{\sfpps}{{\sfp_P^{(S)}}}
\nc{\sfppms}{{\sfp_{P^-}^{(S)}}}

\nc\ol{\overline}
\nc\ul{\underline}
\nc\wt{\widetilde}
\nc\tboxtimes{\wt{\boxtimes}}
\nc\tstar{\wt{\star}}
\nc{\alp}{\alpha}

\nc{\ZZ}{{\mathbb Z}}
\nc{\NN}{{\mathbb N}}
\nc{\OO}{{\mathbb O}}
\renc{\SS}{{\mathbb S}}
\nc{\DD}{{\mathbb D}}
\nc{\GG}{{\mathbb G}}

\nc{\Fq}{{\mathbb F}_q}
\nc{\Fqb}{\ol{\mathbb F}_q}
\nc{\Ql}{{\mathbb Q}_\ell}
\nc{\Qlb}{{\ol{\mathbb Q}_\ell}}
\nc{\id}{\text{id}}
\nc\X{\mathcal X}

\nc{\red}{\on{red}}
\nc{\Ho}{\on{Ho}}
\nc{\Hom}{\on{Hom}}
\nc{\coHom}{\ul{\on{coHom}}}
\nc{\coMaps}{{\bf{coMaps}}}
\nc{\coef}{\on{coef}}
\nc{\Lie}{\on{Lie}}
\nc{\Loc}{\on{Loc}}
\nc{\coLoc}{\on{coLoc}}
\nc{\Pic}{\on{Pic}}
\nc{\Bun}{\on{Bun}}
\nc{\IC}{\on{IC}}
\nc{\Aut}{\on{Aut}}
\nc{\rk}{\on{rk}}
\nc{\Sh}{\on{Sh}}
\nc{\Perv}{\on{Perv}}
\nc{\pos}{{\on{pos}}}
\nc{\Conv}{\on{Conv}}
\nc{\Sph}{\on{Sph}}
\nc{\Sym}{\on{Sym}}
\nc{\BunBb}{\overline{\Bun}_B}
\nc{\BunNb}{\overline{\Bun}_N}
\nc{\BunTb}{\overline{\Bun}_T}
\nc{\BunBbm}{\overline{\Bun}_{B^-}}
\nc{\BunBbel}{\overline{\Bun}_{B,el}}
\nc{\BunBbmel}{\overline{\Bun}_{B^-,el}}
\nc{\Buno}{\overset{o}{\Bun}}
\nc{\BunPb}{{\overline{\Bun}_P}}
\nc{\BunBM}{\Bun_{B(M)}}
\nc{\BunBMb}{\overline{\Bun}_{B(M)}}
\nc{\BunPbw}{{\widetilde{\Bun}_P}}
\nc{\BunBP}{\widetilde{\Bun}_{B,P}}
\nc{\GUb}{\overline{G/U}}
\nc{\GUPb}{\overline{G/U(P)}}

\nc{\Hhom}{\underline{\on{Hom}}}
\nc\syminfty{\on{Sym}^{\infty}}
\nc\lal{\ol{\lambda}}
\nc\xl{\ol{x}}
\nc\thl{\ol{\theta}}
\nc\nul{\ol{\nu}}
\nc\mul{\ol{\mu}}
\nc\Sum\Sigma
\nc{\oX}{\overset{o}{X}{}}
\nc{\hl}{\overset{\leftarrow}h{}}
\nc{\hr}{\overset{\rightarrow}h{}}
\nc{\M}{{\mathcal M}}
\nc{\N}{{\mathcal N}}
\nc{\F}{{\mathcal F}}
\nc{\D}{{\mathcal D}}
\nc{\Q}{{\mathcal Q}}
\nc{\Y}{{\mathcal Y}}
\nc{\G}{{\mathcal G}}
\nc{\E}{{\mathcal E}}
\nc{\CalC}{{\mathcal C}}
\nc\Dh{\widehat{\D}}

\nc{\C}{{\mathcal C}}
\nc{\K}{{\mathcal K}}
\renewcommand{\H}{{\mathcal H}}

\nc{\T}{{\mathcal T}}
\nc{\V}{{\mathcal V}}
\renc{\P}{{\mathcal P}}
\nc{\A}{{\mathcal A}}
\nc{\B}{{\mathcal B}}
\nc{\U}{{\mathcal U}}

\nc{\Gr}{{\on{Gr}}}

\nc{\frn}{{\check{\mathfrak u}(P)}}

\nc{\fC}{\mathfrak C}
\nc{\fT}{\mathfrak T}
\nc{\p}{\mathfrak p}
\nc{\q}{\mathfrak q}
\nc\f{{\mathfrak f}}

\nc{\qo}{{\mathfrak q}}
\nc{\po}{{\mathfrak p}}
\nc{\s}{{\mathfrak s}}
\nc\w{\text{w}}

\renewcommand{\mod}{{\on{-mod}}}

\nc\mathi\iota
\nc\Spec{\on{Spec}}
\nc\Proj{\on{Proj}}
\nc\Mod{\on{Mod}}
\nc{\tw}{\widetilde{\mathfrak t}}
\nc{\pw}{\widetilde{\mathfrak p}}
\nc{\qw}{\widetilde{\mathfrak q}}
\nc{\jw}{\widetilde j}

\nc{\grb}{\overline{\Gr}}
\nc{\I}{\mathcal I}

\nc{\lambdach}{{\check\lambda}}
\nc{\Lambdach}{{\check\Lambda}{}}
\nc{\much}{{\check\mu}}
\nc{\omegach}{{\check\omega}}
\nc{\nuch}{{\check\nu}}
\nc{\etach}{{\check\eta}}
\nc{\alphach}{{\check\alpha}}
\nc{\oblvtach}{{\check\oblvta}}
\nc{\rhoch}{{\check\rho}}
\nc{\ch}{{\check h}}

\nc{\Hb}{\overline{\H}}

\nc{\BA}{{\mathbb{A}}}
\nc{\BC}{{\mathbb{C}}}
\nc{\BE}{{\mathbb{E}}}
\nc{\BF}{{\mathbb{F}}}
\nc{\BG}{{\mathbb{G}}}
\nc{\BL}{{\mathbb{L}}}
\nc{\BM}{{\mathbb{M}}}
\nc{\BO}{{\mathbb{O}}}
\nc{\BD}{{\mathbb{D}}}
\nc{\BN}{{\mathbb{N}}}
\nc{\BP}{{\mathbb{P}}}
\nc{\BQ}{{\mathbb{Q}}}
\nc{\BR}{{\mathbb{R}}}
\nc{\BV}{{\mathbb{V}}}
\nc{\BZ}{{\mathbb{Z}}}
\nc{\BS}{{\mathbb{S}}}
\nc{\Deep}{{\bf{deep}}}
\nc{\deep}{deep}

\nc{\CA}{{\mathcal{A}}}
\nc{\CB}{{\mathcal{B}}}

\nc{\CE}{{\mathcal{E}}}
\nc{\CF}{{\mathcal{F}}}
\nc{\CH}{{\mathcal{H}}}

\nc{\CL}{{\mathcal{L}}}
\nc{\CC}{{\mathcal{C}}}
\nc{\CG}{{\mathcal{G}}}
\nc{\CalD}{{\mathcal{D}}}
\nc{\CM}{{\mathcal{M}}}
\nc{\CN}{{\mathcal{N}}}
\nc{\CK}{{\mathcal{K}}}
\nc{\CO}{{\mathcal{O}}}
\nc{\CP}{{\mathcal{P}}}
\nc{\CQ}{{\mathcal{Q}}}
\nc{\CR}{{\mathcal{R}}}
\nc{\CS}{{\mathcal{S}}}
\nc{\CT}{{\mathcal{T}}}
\nc{\CU}{{\mathcal{U}}}
\nc{\CV}{{\mathcal{V}}}
\nc{\CW}{{\mathcal{W}}}
\nc{\CX}{{\mathcal{X}}}
\nc{\CY}{{\mathcal{Y}}}
\nc{\CZ}{{\mathcal{Z}}}
\nc{\CI}{{\mathcal{I}}}

\nc{\csM}{{\check{\mathcal A}}{}}
\nc{\oM}{{\overset{\circ}{\mathcal M}}{}}
\nc{\obM}{{\overset{\circ}{\mathbf M}}{}}
\nc{\oCA}{{\overset{\circ}{\mathcal A}}{}}
\nc{\obA}{{\overset{\circ}{\mathbf A}}{}}
\nc{\ooM}{{\overset{\circ}{M}}{}}
\nc{\osM}{{\overset{\circ}{\mathsf M}}{}}
\nc{\vM}{{\overset{\bullet}{\mathcal M}}{}}
\nc{\nM}{{\underset{\bullet}{\mathcal M}}{}}
\nc{\oD}{{\overset{\circ}{\mathcal D}}{}}
\nc{\obD}{{\overset{\circ}{\mathbf D}}{}}
\nc{\oA}{{\overset{\circ}{A}}{}}
\nc{\op}{{\overset{\bullet}{\mathbf p}}{}}
\nc{\cp}{{\overset{\circ}{\mathbf p}}{}}
\nc{\oU}{{\overset{\bullet}{\mathcal U}}{}}
\nc{\oZ}{{\overset{\circ}{\mathcal Z}}{}}
\nc{\ofZ}{{\overset{\circ}{\mathfrak Z}}{}}
\nc{\oF}{{\overset{\circ}{\fF}}}

\nc{\fa}{{\mathfrak{a}}}
\nc{\ofa}{\overset{\circ}{\mathfrak{a}}}
\nc{\fb}{{\mathfrak{b}}}
\nc{\fd}{{\mathfrak{d}}}
\nc{\ff}{{\mathfrak{f}}}
\nc{\fg}{{\mathfrak{g}}}
\nc{\fgl}{{\mathfrak{gl}}}
\nc{\fh}{{\mathfrak{h}}}
\nc{\fj}{{\mathfrak{j}}}
\nc{\fl}{{\mathfrak{l}}}
\nc{\fm}{{\mathfrak{m}}}
\nc{\ofm}{\overset{\circ}{\mathfrak{m}}}
\nc{\fn}{{\mathfrak{n}}}
\nc{\fu}{{\mathfrak{u}}}
\nc{\fp}{{\mathfrak{p}}}
\nc{\fr}{{\mathfrak{r}}}
\nc{\fs}{{\mathfrak{s}}}
\nc{\ft}{{\mathfrak{t}}}
\nc{\oft}{\overset{\circ}{\mathfrak{t}}}
\nc{\fz}{{\mathfrak{z}}}
\nc{\fsl}{{\mathfrak{sl}}}
\nc{\hsl}{{\widehat{\mathfrak{sl}}}}
\nc{\hgl}{{\widehat{\mathfrak{gl}}}}
\nc{\hg}{{\widehat{\mathfrak{g}}}}
\nc{\hm}{{\widehat{\mathfrak{m}}}}
\nc{\chg}{{\widehat{\mathfrak{g}}}{}^\vee}
\nc{\hn}{{\widehat{\mathfrak{n}}}}
\nc{\chn}{{\widehat{\mathfrak{n}}}{}^\vee}

\nc{\fA}{{\mathfrak{A}}}
\nc{\fB}{{\mathfrak{B}}}
\nc{\fD}{{\mathfrak{D}}}
\nc{\fE}{{\mathfrak{E}}}
\nc{\fF}{{\mathfrak{F}}}
\nc{\fG}{{\mathfrak{G}}}
\nc{\fK}{{\mathfrak{K}}}
\nc{\fL}{{\mathfrak{L}}}
\nc{\fM}{{\mathfrak{M}}}
\nc{\fN}{{\mathfrak{N}}}
\nc{\fP}{{\mathfrak{P}}}
\nc{\fU}{{\mathfrak{U}}}
\nc{\fV}{{\mathfrak{V}}}
\nc{\fZ}{{\mathfrak{Z}}}

\nc{\ba}{{\mathbf{a}}}
\nc{\bb}{{\mathbf{b}}}
\nc{\bc}{{\mathbf{c}}}
\nc{\bd}{{\mathbf{d}}}
\nc{\bbf}{{\mathbf{f}}}
\nc{\be}{{\mathbf{e}}}
\nc{\bi}{{\mathbf{i}}}
\nc{\bj}{{\mathbf{j}}}
\nc{\bh}{{\mathbf{h}}}
\nc{\bm}{{\mathbf{m}}}
\nc{\bn}{{\mathbf{n}}}
\nc{\bo}{{\mathbf{o}}}
\nc{\bp}{{\mathbf{p}}}
\nc{\bq}{{\mathbf{q}}}
\nc{\bu}{{\mathbf{u}}}
\nc{\bv}{{\mathbf{v}}}
\nc{\bx}{{\mathbf{x}}}
\nc{\bs}{{\mathbf{s}}}
\nc{\by}{{\mathbf{y}}}
\nc{\bw}{{\mathbf{w}}}
\nc{\bA}{{\mathbf{A}}}
\nc{\bK}{{\mathbf{K}}}
\nc{\bB}{{\mathbf{B}}}
\nc{\bC}{{\mathbf{C}}}
\nc{\bG}{{\mathbf{G}}}
\nc{\bD}{{\mathbf{D}}}
\nc{\bE}{{\mathbf{E}}}
\nc{\bH}{{{\mathbf{H}}}}
\nc{\bL}{{\mathbf{L}}}
\nc{\bM}{{\mathbf{M}}}
\nc{\bN}{{\mathbf{N}}}
\nc{\bO}{{\mathbf{O}}}
\nc{\bQ}{{\mathbf{Q}}}
\nc{\bV}{{\mathbf{V}}}
\nc{\bW}{{\mathbf{W}}}
\nc{\bX}{{\mathbf{X}}}
\nc{\bZ}{{\mathbf{Z}}}
\nc{\bS}{{\mathbf{S}}}

\nc{\sA}{{\mathsf{A}}}
\nc{\sB}{{\mathsf{B}}}
\nc{\sC}{{\mathsf{C}}}
\nc{\sD}{{\mathsf{D}}}
\nc{\sF}{{\mathsf{F}}}
\nc{\sG}{{\mathsf{G}}}
\nc{\sH}{{\mathsf{H}}}
\nc{\sK}{{\mathsf{K}}}
\nc{\sM}{{\mathsf{M}}}
\nc{\sN}{{\mathsf{N}}}
\nc{\sO}{{\mathsf{O}}}
\nc{\sW}{{\mathsf{W}}}
\nc{\sQ}{{\mathsf{Q}}}
\nc{\sP}{{\mathsf{P}}}
\nc{\sR}{{\mathsf{R}}}
\nc{\sT}{{\mathsf{T}}}
\nc{\sZ}{{\mathsf{Z}}}
\nc{\sfi}{{\mathsf{i}}}
\nc{\sfj}{{\mathsf{j}}}
\nc{\sfp}{{\mathsf{p}}}
\nc{\sfq}{{\mathsf{q}}}
\nc{\sfs}{{\mathsf{s}}}
\nc{\sft}{{\mathsf{t}}}
\nc{\sr}{{\mathsf{r}}}
\nc{\bk}{{\mathsf{k}}}
\nc{\sa}{{\mathsf{s}}}
\nc{\sg}{{\mathsf{g}}}
\nc{\sn}{{\mathsf{n}}}
\nc{\sh}{{\mathsf{h}}}
\nc{\sff}{{\mathsf{f}}}
\nc{\sk}{{\mathsf{k}}}
\nc{\sfb}{{\mathsf{b}}}
\nc{\sfc}{{\mathsf{c}}}
\nc{\sfe}{{\mathsf{e}}}
\nc{\sd}{{\mathsf{d}}}

\nc{\BK}{{\bar{K}}}

\nc{\tA}{{\widetilde{\mathbf{A}}}}
\nc{\tB}{{\widetilde{\mathcal{B}}}}
\nc{\tg}{{\widetilde{\mathfrak{g}}}}
\nc{\tG}{{\widetilde{G}}}
\nc{\TM}{{\widetilde{\mathbb{M}}}{}}
\nc{\tO}{{\widetilde{\mathsf{O}}}{}}
\nc{\tU}{{\widetilde{\mathfrak{U}}}{}}
\nc{\TZ}{{\tilde{Z}}}
\nc{\tx}{{\tilde{x}}}
\nc{\tbv}{{\tilde{\bv}}}
\nc{\tfP}{{\widetilde{\mathfrak{P}}}{}}
\nc{\tz}{{\tilde{\zeta}}}
\nc{\tmu}{{\tilde{\mu}}}

\nc{\urho}{\underline{\rho}}
\nc{\uB}{\underline{B}}
\nc{\uC}{{\underline{\mathbb{C}}}}
\nc{\ui}{\underline{i}}
\nc{\uj}{\underline{j}}
\nc{\ofP}{{\overline{\mathfrak{P}}}}
\nc{\oB}{{\overline{\mathcal{B}}}}
\nc{\og}{{\overline{\mathfrak{g}}}}
\nc{\oI}{{\overline{I}}}

\nc{\eps}{\varepsilon}
\nc{\hrho}{{\hat{\rho}}}

\nc{\one}{{\mathbf{1}}}
\nc{\two}{{\mathbf{t}}}

\nc{\Rep}{{\mathop{\operatorname{\rm Rep}}}}
\nc{\Tot}{{\mathop{\operatorname{\rm Tot}}}}
\nc{\Ker}{{\mathop{\operatorname{\rm Ker}}}}
\nc{\im}{{\mathop{\operatorname{\rm Im}}}}
\nc{\Hilb}{{\mathop{\operatorname{\rm Hilb}}}}
\nc{\End}{{\mathop{\operatorname{\rm End}}}}
\nc{\Ext}{{\mathop{\operatorname{\rm Ext}}}}
\nc{\CHom}{{\mathop{\operatorname{{\mathcal{H}}\it om}}}}
\nc{\CEnd}{{\mathop{\operatorname{{\mathcal{E}}\it nd}}}}
\nc{\GL}{{\mathop{\operatorname{\rm GL}}}}
\nc{\gr}{{\mathop{\operatorname{\rm gr}}}}
\nc{\HN}{{\mathop{\operatorname{\rm HN}}}}
\nc{\Id}{{\mathop{\operatorname{\rm Id}}}}
\nc{\de}{{\mathop{\operatorname{\rm def}}}}
\nc{\length}{{\mathop{\operatorname{\rm length}}}}
\nc{\supp}{{\mathop{\operatorname{\rm supp}}}}

\nc{\Cliff}{{\mathsf{Cliff}}}
\nc{\Fl}{\on{Fl}}
\nc{\Fib}{{\mathsf{Fib}}}
\nc{\Coh}{{\on{Coh}}}
\nc{\QCoh}{{\on{QCoh}}}
\nc{\IndCoh}{{\on{IndCoh}}}
\nc{\FCoh}{{\mathsf{FCoh}}}

\nc{\reg}{{\text{\rm reg}}}

\nc{\cplus}{{\mathbf{C}_+}}
\nc{\cminus}{{\mathbf{C}_-}}
\nc{\cthree}{{\mathbf{C}_\bullet}}
\nc{\Qbar}{{\bar{Q}}}
\nc\Eis{\on{Eis}}
\nc\Eisb{\ol\Eis{}}
\nc\Eisr{\on{Eis}^{rat}{}}
\nc\wh{\widehat}
\nc{\Def}{\on{Def_{\check{\fb}}(E)}}
\nc{\barZ}{\overline{Z}{}}
\nc{\barbarZ}{\overline{\barZ}{}}
\nc{\barpi}{\overline\pi}
\nc{\barbarpi}{\overline\barpi}
\nc{\barpip}{\overline\pi{}^+}
\nc{\barpim}{\overline\pi{}^-}

\nc{\fq}{\mathfrak q}

\nc{\fqb}{\ol{\sfq}{}}
\nc{\fpb}{\ol{\sfp}{}}
\nc{\fpr}{{\mathsf{pair}^{rat}}{}}
\nc{\fqr}{{\sfq^{rat}}{}}

\nc{\hattimes}{\wh\otimes}

\nc{\bOmega}{{\overline{\Omega(\check \fn)}}}

\nc{\seq}[1]{\stackrel{#1}{\sim}}

\nc{\cT}{{\check{T}}}
\nc{\cG}{{\check{G}}}
\nc{\cM}{{\check{M}}}
\nc{\cB}{{\check{B}}}
\nc{\cP}{{\check{P}}}

\nc{\ct}{{\check{\mathfrak t}}}
\nc{\cg}{{\check{\fg}}}
\nc{\cb}{{\check{\fb}}}
\nc{\cn}{{\check{\fn}}}

\nc{\cLambda}{{\check\Lambda}}

\nc{\cla}{{\check\lambda}}
\nc{\cmu}{{\check\mu}}
\nc{\cnu}{{\check\nu}}
\nc{\ceta}{{\check\eta}}

\nc{\DefbE}{{\on{Def}_{\cB}(E_\cT)}}

\nc{\imathb}{{\ol{\imath}}}
\nc{\rlr}{\overset{\longrightarrow}{\underset{\longrightarrow}\longleftarrow}}

\nc{\oBun}{\overset{\circ}\Bun}
\nc{\LS}{\on{LS}}
\nc{\BunBbb}{\ol{\ol{Bun}}_B}
\nc{\BunBr}{\Bun_B^{rat}}
\nc{\BunBrsg}{\Bun_B^{rat,\on{s.g.}}}
\nc{\BunBrp}{\Bun_B^{rat,polar}}
\nc{\BunBrpbg}{\Bun_B^{rat,polar,\on{b.g.}}}
\nc{\BunBrpsg}{\Bun_B^{rat,polar,\on{s.g.}}}
\nc{\BunTrp}{\Bun_T^{rat,polar}}
\nc{\BunTrpbg}{\Bun_T^{rat,polar,\on{b.g.}}}
\nc{\BunTrpsg}{\Bun_T^{rat,polar,\on{s.g.}}}
\nc{\BunNr}{\Bun_N^{rat}}
\nc{\BunNre}{\Bun_N^{enh,rat}}
\nc{\BunTr}{\Bun_T^{rat}}
\nc{\Vect}{\on{Vect}}
\nc{\Whit}{\on{Whit}}
\nc{\CTb}{\ol{\on{CT}}}
\nc{\Ran}{{\on{Ran}}}
\nc{\Ranu}{{\on{Ran}^{\on{untl}}}}
\nc{\Ranustr}{{\on{Ran}^{\on{untl}}_{\on{str}}}}
\nc{\Ranusubset}{{\Ran^{\on{untl}}_{\subseteq}}}
\nc{\Ranusubsetx}{{\Ran^{\on{untl}}_{x\subseteq}}}
\nc{\CTr}{\on{CT}^{rat}{}}
\nc\jmathr{\jmath^{rat}{}}
\nc{\ux}{\underline{x}}
\nc{\clambda}{{\check\lambda}}
\nc{\calpha}{{\check\alpha}}
\nc{\ind}{{\mathbf{ind}}}
\nc{\coinv}{{\mathbf{coinv}}}
\nc{\oblv}{{\mathbf{oblv}}}
\nc{\free}{{\mathbf{free}}}
\nc{\ox}{{\overline{x}}}
\nc{\cLa}{\check{\Lambda}}
\nc{\StinftyCat}{\on{DGCat}}
\nc{\inftyCat}{\infty\on{-Cat}}
\nc{\inftygroup}{\infty\on{-Grpd}}
\nc{\Dmod}{\on{D-mod}}
\nc{\CMaps}{{\mathcal Maps}}
\nc{\Maps}{\on{Maps}}
\nc{\affSch}{\on{Sch}^{\on{aff}}}
\nc{\dr}{{\on{dR}}}
\nc{\oCF}{\overset{\circ}\CF}
\nc{\oCY}{\overset{\circ}\CY}
\nc{\oCZ}{\overset{\circ}\CZ}
\nc{\opi}{\overset{\circ}\pi}
\nc{\leqG}{\underset{G}\leq}
\nc{\leqM}{\underset{M}\leq}
\nc{\leqGad}{\underset{G_{ad}}\leq}
\nc{\leqMad}{\underset{M_{ad}}\leq}
\nc{\Tr}{\on{Tr}}
\nc{\Frob}{{\on{Frob}}}
\nc{\DGCat}{\on{DGCat}}
\nc{\tDGCat}{2\on{-DGCat}_{\on{u.g.}}}
\nc{\ev}{\on{ev}}
\nc{\mmod}{{\on{-}\!\mathbf{mod}}}
\nc{\sotimes}{\overset{!}\otimes}
\nc{\Shv}{\on{Shv}}
\nc{\Spc}{\on{Spc}}
\nc{\Res}{\on{Res}}
\nc{\bDelta}{{\mathbf{\Delta}}}
\nc{\bMaps}{{\mathbf{Maps}}}
\nc{\cD}{\mathcal D}
\nc{\ocD}{\cD^\times}
\nc{\ppart}{(\!(t)\!)}
\nc{\qqart}{[\![t]\!]}
\nc{\oCU}{\overset{\circ}{\CU}}
\nc{\Exc}{{\mathcal{E}xc}}
\nc{\Sht}{\on{Sht}}
\nc{\Nilp}{{\on{Nilp}}}
\nc{\Drinf}{\on{Drinf}}
\nc{\Sing}{\on{Sing}}
\nc{\IndLisse}{\Lisse}
\nc{\Shvl}{\on{Shv}_{\on{lisse}}} 
\nc{\Lisse}{\on{Lisse}}
\nc{\Mir}{\on{Mir}}
\nc{\fSet}{\on{fSet}}
\nc{\qLisse}{\on{QLisse}}
\nc{\Ev}{\on{Ev}}
\nc{\Sat}{\on{Sat}}
\nc{\Se}{\on{Se}}
\nc{\coSht}{\on{co-Sht}}
\nc{\coCK}{\on{co-}\!\CK}
\nc{\FLE}{\on{FLE}}
\nc{\BRST}{\on{BRST}}
\nc{\KL}{\on{KL}}
\nc{\crit}{{\on{crit}}}
\nc{\Op}{{\on{Op}}}
\nc{\MOp}{\on{MOp}}
\nc{\Wak}{\on{Wak}}
\nc{\Av}{\on{Av}}
\nc{\semiinf}{{\frac{\infty}{2}}}
\nc{\DS}{\on{DS}}
\nc{\dR}{{\on{dR}}}
\nc{\Poinc}{{\on{Poinc}}}
\renc{\det}{\on{det}}
\nc{\oG}{\overset{\circ}{G}}
\nc{\Sectna}{\on{Sect}_\nabla}
\nc{\bGamma}{\mathbf{\Gamma}}
\nc{\ShvCat}{\on{ShvCat}}
\nc{\zero}{\mathbf{0}}

\begin{document}

%

\vskip1cm


\dedicatory{To V.~Drinfeld}

\title[Proof of the geometric Langlands conjecture V]{Proof of the geometric Langlands conjecture V: \\ the multiplicity one theorem}

\author{Dennis Gaitsgory and Sam Raskin}

\begin{quote}
``Fear will keep the local systems in line." \newline (Grand Moff Tarkin, \emph{Star Wars: Episode IV \textendash{} A New Hope}) \vskip0.5cm
\end{quote}

\date{\today}

\maketitle

\bigskip

\bigskip

\tableofcontents

\section*{Introduction}

This paper is a 
conclusion of the series of papers \cite{GLC1,GLC2,GLC3,GLC4}, written jointly with D.~Arinkin, D.~Beraldo, J.~Campbell, 
L.~Chen, J.~F\ae{}rgeman, K.~Lin, and N.~Rozenblyum.

\medskip

In the preceding papers, we have constructed the Langlands functor 
\begin{equation} \label{e:LG Intro}
\BL_G:\Dmod_{\frac{1}{2}}(\Bun_G) \to \IndCoh_\Nilp(\LS_{\cG}),
\end{equation}
and established many of its properties.

\medskip

In this paper we will prove the 
geometric Langlands conjecture (GLC) by showing 
that $\BL_G$ is an equivalence. 

\medskip 

This theorem confirms the original vision of 
Beilinson-Drinfeld, which
was circulated by them in discussions of their work \cite{BD}. 
A precise version of the conjecture was originally
given in \cite[Conjecture 11.2.3]{AG}. The functor $\BL_G$
was constructed in \cite{GLC1}, and a version of GLC with
this precise functor included 
appears as \cite[Conjecture 1.6.7]{GLC1}. As discussed
in \cite{GLC1}, the functor
$\BL_G$ is the unique possible equivalence that is
compatible with the \emph{Whittaker} and \emph{Eisenstein}
compatibilities of \cite{Ga4}.

%
%
%

\ssec{What was done in previous papers?} \label{ss:lan party}

We begin with a short summary of the key points of \cite{GLC1,GLC2,GLC3,GLC4}.
This material is discussed in greater detail in \secref{s:review}.

\sssec{}

Recall from \cite[Sect. 1.1]{GLC1} that $\Dmod_{\frac{1}{2}}(\Bun_G)$ denotes the category
of \emph{half-twisted} D-modules on $\Bun_G$. This is the category that appears on the
automorphic side of the (de Rham) version of the geometric Langlands conjecture (GLC). 

\medskip 

In \cite{GLC1}, we used the functor of Whittaker coefficient to construct 
the \emph{geometric Langlands functor}
\begin{equation} \label{e:LG Intro again}
\BL_G:\Dmod_{\frac{1}{2}}(\Bun_G) \to \IndCoh_\Nilp(\LS_{\cG}),
\end{equation}
where the right hand side was defined in \cite{AG}. 
This functor is linear with respect to the \emph{spectral action} of
$\QCoh(\LS_{\cG})$ on $\Dmod_{\frac{1}{2}}(\Bun_G)$ (see \cite[Sect. 1.2]{GLC1}).
The geometric Langlands conjecture states that $\BL_G$ is an equivalence of categories.

\sssec{}

We proved in \cite[Theorem 1.6.2]{GLC4} (building on \cite{FR} and \cite{GLC3}) that the functor 
$\BL_G$ is conservative. We also proved that $\BL_G$ admits a left adjoint, to be denoted $\BL^L_G$. 
Moreover, we proved that the composition of the functor $\BL_G$ and its left adjoint $\BL^L_G$,
which is an endofunctor of $\IndCoh_\Nilp(\LS_{\cG})$, is given by tensoring by an object
$$\CA_G\in \QCoh(\LS_\cG).$$
This object $\CA_G$ is naturally an associative algebra object in $\QCoh(\LS_\cG)$. 

\medskip 

By Barr-Beck, we obtain an equivalence
\[
\begin{tikzcd}
\Dmod_{\frac{1}{2}}(\Bun_G) 
\arrow[r,dotted,"\simeq"] 
\arrow[dr,swap,"\BL_G"]
&
\CA_G\mod(\IndCoh_\Nilp(\LS_{\cG}))
\arrow[d]
\\
&
\IndCoh_\Nilp(\LS_{\cG})
\end{tikzcd}
\]

\noindent between $\Dmod_{\frac{1}{2}}(\Bun_G)$ and the category
of $\CA_G$-modules in $\IndCoh_\Nilp(\LS_{\cG})$. 

\medskip 

Therefore, the geometric Langlands conjecture 
amounts to the assertion that the
unit map
\begin{equation} \label{e:AG initial Intro}
\CO_{\LS_\cG}\to \CA_G,
\end{equation}
is an isomorphism in $\QCoh(\LS_\cG)$.

\sssec{}

The main theorem of \cite{GLC3} asserts
that the restriction of the map \eqref{e:AG initial Intro} 
to the locus of \emph{reducible} local systems is an isomorphism. 

\medskip

Therefore, it remains to show that the map
\begin{equation} \label{e:AG irred Intro}
\CO_{\LS_\cG^{\on{irred}}}\to \CA_G|_{\LS_\cG^{\on{irred}}}, 
\end{equation}
induced by \eqref{e:AG initial Intro}, is an isomorphism. 

\sssec{} \label{sss:A red Intro}

The paper \cite{GLC4} developed some structural features of $\CA_{G,\on{irred}}$.

\medskip 

Namely, that work proves that  
$$\CA_{G,\on{irred}}:=\CA_G|_{\LS_\cG^{\on{irred}}}$$
is a \emph{classical} vector bundle\footnote{Here we are for simplicity assuming that $G$ is semi-simple, i.e.,
that $Z^0_G=\{1\}$.} that carries a flat connection. Moreover, this connection 
has finite monodromy. 

\sssec{}

The above observations mark the starting point of the present paper.

\ssec{What is done in this paper?}\label{ss:proof outline}

\sssec{}\label{sss:what is to be done}
First, let us observe what needs to be done given the preliminaries above.

\medskip

Let $\sigma$ be an \emph{irreducible} $\cG$-local system. We need to 
show that the fiber $\CA_{G,\sigma}$ of $\CA_G$ at $\sigma$ is the unit
associative algebra. 

\medskip 

As $\CA_{G,\sigma}\mod$ is the category $\on{Hecke}_{\sigma}(\Dmod(\Bun_G))$ 
of Hecke eigensheaves with 
eigenvalue $\sigma$, the above amounts to showing that there
is a \emph{unique} (up to tensoring with a vector space) 
Hecke eigensheaf for each such $\sigma$.

\medskip 

By the construction of the Langlands functor $\BL_G$, this amounts to 
two assertions, (\emph{i}) the existence of a Hecke eigensheaf for an
irreducible spectral parameter $\sigma$, and (\emph{ii}) 
a multiplicity one theorem: we need to show
that cuspidal objects of $\Dmod(\Bun_G)$ 
can be (uniquely) reconstructed from their Whittaker coefficients.

\sssec{}\label{sss:features}

We now outline the argument presented in the paper.

\medskip 

For simplicity, we assume that $G$ is simple (in particular, of adjoint type) and 
that the genus $g$ is at least $2$; 
we also exclude the case when $g=2$ and $G=PGL_2$. 
The proof is based on the three
observations concerning the topology of $\LS_{\cG}^{\on{irred}}$, 
the algebraic geometry of $\LS_{\cG}$, and the sheaf theory of
$\Bun_G$: 

\medskip

\begin{itemize} 

\item $\LS_\cG^{\on{irred}}$ is simply-connected (\thmref{t:pi 1 LS});

\smallskip

\item $\LS_\cG$ is Cohen-Macaulay 
and the complement of $\LS_\cG^{\on{irred}}$ has codimension $\geq 2$
(\corref{c:LS CM}, \propref{p:compl 2});

\smallskip

\item Endomorphisms of the vacuum Poincar\'e sheaf
$\on{Poinc}^{\on{Vac,glob}}_{G,!}\in \Dmod_{\frac{1}{2}}(\Bun_G)$
are just scalars (\thmref{t:end Poinc vac a}).

\end{itemize}

\medskip

We remind here that the vacuum Poincar\'e sheaf appears in characterizing
the Langlands functor $\BL_G$, cf. \cite[Sect. 1.4]{GLC1}.  

\medskip 

The last of these three results is the simplest to prove. 
But as we will see below, it is ultimately the 
point most responsible for addressing the \emph{multiplicity one} problem
mentioned in \secref{sss:what is to be done}. 
 
\sssec{}

The above three observations combine as follows.

\medskip

The first observation combined with the features of
$\CA_{G,\on{irred}}$ mentioned in \secref{sss:A red Intro} implies that 
$\CA_{G,\on{irred}}$ is isomorphic to the direct sum of several copies of the structure sheaf, i.e.,
$$\CA_{G,\on{irred}} \simeq \CO^{\oplus n}_{\LS_\cG^{\on{irred}}}.$$

As \eqref{e:AG irred Intro} is a morphism of algebras, it 
suffices to show that $n=1$. 

\sssec{} \label{sss:dim count}

The second observation, combined with the fact that \eqref{e:AG initial Intro} is an isomorphism on the 
\emph{reducible} locus, implies that
$$H^0\left(\Gamma(\LS_\cG,\CA_G)\right)\to H^0\left(\Gamma(\LS^{\on{irred}}_\cG,\CA_{G,\on{irred}})\right)$$
is an isomorphism (see \secref{sss:step 1} for more details).

\medskip

So, we obtain that 
\begin{multline} \label{e:dim estim}
\dim\left(H^0\left(\Gamma(\LS_\cG,\CA_G)\right)\right)=\dim\left(H^0\left(\Gamma(\LS^{\on{irred}}_\cG,\CA_{G,\on{irred}})\right)\right)=\\
=\dim\left(H^0\left(\Gamma(\LS^{\on{irred}}_\cG,\CO^{\oplus n}_{\LS^{\on{irred}}})\right)\right)\geq n.
\end{multline}

\sssec{} \label{sss:main step}

Finally, by the construction of $\BL_G$, we have
$$\Gamma(\LS_\cG,\CA_G)\simeq \CEnd_{\Dmod_{\frac{1}{2}}(\Bun_G)}(\on{Poinc}^{\on{Vac}}_{G,!}).$$

The third observation implies that $H^0$ of the right-hand side is one-dimensional. Combined with \eqref{e:dim estim}, this implies that $n=1$,
as desired. 

\begin{rem}

The outline above uses a special 
feature of the de Rham setting. By contrast, Betti or
\'etale moduli stacks of local systems have infinite dimensional algebras
of global functions, while the de Rham moduli stack has very
few global functions.
For this reason, our strategy does not adapt to either Betti or \'etale settings.
In particular, the dimension count above carries no meaning in those settings, and
our overall strategy for controlling multiplicities does not adapt.

\medskip

With that said, we remind that it was shown in \cite[Theorem 3.5.2]{GLC1} that the de Rham version of GLC
proved in this paper implies the Betti version.

\end{rem}

\sssec{Deficiencies} \label{sss:imperfect1}

An awkward aspect of this paper is that the argument outlined above does not apply in low genus. Namely, the three features 
mentioned in \secref{sss:features} break down when $g\leq 1$ (and the first two also for $g=2$ when $G$ is of type $A_1$).  

\medskip

Thus, the proof of GLC we give is \emph{not} uniform across all reductive groups and genera. 

\sssec{}

We treat the outlying cases as follows (still assuming that $G$ is simple): 

\medskip

First, when $g=0$, there is nothing to prove, as $\LS^{\on{irred}}_\cG$ is empty in this case.

\medskip

Second, we recall from \cite[Theorem 1.8.2]{GLC4} 
that multiplicity one problems can be settled for arbitrary
genus and groups of type $A_n$ using features of the 
geometry of opers, which is much easier to understand 
in this case.

\medskip

Thus, it remains to treat the case of $g=1$ and a simple group of type different than
$A_n$. However, in this case a theorem of \cite{KS,BFM} asserts that $\LS^{\on{irred}}_\cG$
is again empty. 

\ssec{Complications stemming from a non-trivial center}

This subsection can be skipped on the first pass. That said, here we will point to a 
fun part of this paper: the 2-categorical Fourier-Mukai transform. 
 
\sssec{}

In the outline in \secref{ss:proof outline}, we have assumed that $G$ is of adjoint type. However, this case is \emph{insufficient}
in order to deduce GLC for any reductive group $G$. 

\medskip

In fact, we prove (see \corref{c:sc}) that GLC reduces
to \emph{almost simple simply-connected} groups.

\medskip

In this subsection we indicate how the outline in \secref{ss:proof outline} 
needs to me modified in order to treat this case.

\sssec{}

Assume that $G$ is semi-simple. Then if $\cG$ is not simply-connected, 
$\LS^{\on{irred}}_\cG$ is not connected, nor are 
its connected components simply-connected.
In fact
$$\pi_0(\LS^{\on{irred}}_\cG)\simeq \pi_0(\Bun_\cG)\simeq (Z_G)^\vee.$$

We denote this bijection by
$$\alpha\in (Z_G)^\vee\,\, \rightsquigarrow \,\,\LS^{\on{irred}}_{\cG,\alpha},$$
where $\LS^{\on{irred}}_{\cG,\alpha}$ is a connected component of $\LS^{\on{irred}}_\cG$.

\medskip

Moreover, each 
connected component of $\LS^{\on{irred}}_\cG$ has abelian fundamental group, 
and its characters are in
bijection with $\Bun_{Z_G}$. We denote this bijection by
$$(\CP_{Z_G}\in \Bun_{Z_G})\rightsquigarrow \CL_{\CP_{Z_G}},$$
where $\CL_{\CP_{Z_G}}$ is a 1-dimensional local system.

\sssec{} \label{sss:center Intro}

With this understood, it is not difficult to adapt the proof explained in \secref{ss:proof outline}, modulo the following issue:

\medskip

We need to know that:

\begin{itemize}

\item The functor $\BL_G$ sends the direct summand of $\Dmod_{\frac{1}{2}}(\Bun_G)$ on which the action of $Z_G$
by automorphisms of the identity functor is given by $\alpha$ to the direct summand of $\IndCoh_\Nilp(\LS_{\cG})$ consisting of sheaves
supported on $\LS^{\on{irred}}_{\cG,\alpha}$.

\smallskip

\item The functor $\BL_G$ intertwines the automorphism of $\Dmod_{\frac{1}{2}}(\Bun_G)$ given by translation
by $\CP_{Z_G}$ with the automorphism of $\IndCoh_\Nilp(\LS_{\cG})$ given by 
tensor product with $\CL_{\CP_{Z_G}}$.

\end{itemize} 

\sssec{}

The above two properties can be formulated purely on the geometric side, in terms of the spectral
action of $\QCoh(\LS_\cG)$ on $\Dmod_{\frac{1}{2}}(\Bun_G)$, see Theorems \ref{t:Hecke Z 0} and
\ref{t:Hecke Z 1}, respectively.

\medskip

The proofs of these two theorems involve a fun manipulation (that appears to be new) with the \emph{2-categorical
Fourier-Mukai transform}. 

\sssec{}

The 2-categorical Fourier-Mukai transform in the case at hand is an equivalence of 2-categories 
between \emph{sheaves of categories}
over the 2-stacks $\on{Ge}_{Z_G}(X)$ and $\on{Ge}_{\pi_1(\cG)}(X)$, which classify gerbes on $X$ with respect to
$Z_G$ and $\pi_1(\cG)$, respectively. 

\medskip

In turns out that one can upgrade $\Dmod_{\frac{1}{2}}(\Bun_G)$ to a sheaf of categories over $\on{Ge}_{Z_G}(X)$
and over $\on{Ge}_{\pi_1(\cG)}(X)$, and the assertion is that the resulting two sheaves of categories map to one
another under the 2-categorical Fourier-Mukai transform. 
Per \remref{r:not 1 aff}, the prestacks 
$\on{Ge}_{Z_G}(X)$ and $\on{Ge}_{\pi_1(\cG)}(X)$ are
\emph{not} 1-affine, so it is essential to work
with sheaves of categories here. 

\medskip

The above assertion implies the two properties mentioned in \secref{sss:center Intro}. But in fact it carries much
more information. For example, it contains the answer to the following question: 

\medskip

What the spectral side of GLC,
when on the automorphic side instead of the usual $\Bun_G$, we consider the stack of bundles with respect to
a non-pure inner form of $G$ (i.e., a twist by a $G_{\on{ad}}$-torsor, which does not come from a $G$-torsor)?

\medskip

It turns out that the answer is the twist of (the usual) $\IndCoh_\Nilp(\LS_\cG)$ by a gerbe on $\LS_\cG$ that is
attached to the $G_{\on{ad}}$-torsor in question.

\ssec{What is not done in this paper?}\label{ss:not done}

\sssec{}\label{sss:neighbors}

First, there are a number of nearby problems that we do not consider
and do not know how to solve. Here are some:

\begin{itemize}

\item Geometric Langlands with Iwahori ramification.

\item Quantum geometric Langlands.

\item Local geometric Langlands with wild ramification. 

\item Global geometric Langlands with wild ramification.

\item Restricted geometric Langlands for $\ell$-adic sheaves 
(for curves in positive characteristic).

\item Geometric Langlands for the Fargues-Fontaine curve.

\end{itemize}

\sssec{} Moreover, \emph{even in the (global, untwisted, unramified, de Rham) setting
of the present paper}, we feel there is more to understand.
The tricks indicated in \secref{ss:proof outline} analyze 
automorphic sheaves via their shadows (specifically, the core part of the argument uses a dimension count).
In the remainder of this subsection, 
we indicate questions and projects 
that might let us observe them more directly.

\sssec{} \label{sss:Whit ff}

As in \secref{sss:what is to be done}, we need to be able to recover
cuspidal D-modules -- or equivalently: eigensheaves with irreducible
spectral parameters -- from their Whittaker coefficients. This can be stated
more formally either as showing that the composition
\begin{equation}\label{eq:coeff-cusp}
\Dmod_{\frac{1}{2}}(\Bun_G)_{\on{cusp}} \subset 
\Dmod_{\frac{1}{2}}(\Bun_G) \xrightarrow{\on{coeff}^{\on{Vac,glob}}_G} 
\Whit(\Gr_{G,\Ran})
\end{equation}

\noindent is fully faithful, or showing that the functor of vacuum 
Whittaker coefficient
\[
\on{Hecke}_{\sigma}\big(\Dmod_{\frac{1}{2}}(\Bun_G)\big) \to 
\Dmod_{\frac{1}{2}}(\Bun_G) \xrightarrow{\on{coeff}_G^{\on{Vac},\on{glob}}} \Vect
\]

\noindent is an equivalence for any irreducible local system $\sigma$.

\medskip 

For $G = GL_n$ (or a quotient thereof), the \emph{mirabolic trick}
suffices to prove this result; see \cite{Be1} for a strong version of this
assertion. We remind that a similar assertion holds at the level of automorphic
functions, although in the arithmetic setup, such assertions are \emph{not}
valid for more general reductive groups.

\medskip 

A posteriori, the same statement for general $G$ follows from the results
of the present paper. But one can imagine tackling this statement 
more directly. 

\medskip 

Below we record some possible strategies that we know have been discussed previously
in the geometric Langlands community. We allow ourselves more informality in this discussion than elsewhere in this
text. 

\sssec{Proof via contractibility of opers?}\label{sss:opers?}

Recall that $\CA_G \in \QCoh(\LS_\cG)$ was defined using sheaves on $\Bun_G$,
i.e., it crosses the barrier between $G$ and $\cG$ -- as should be no surprise.

\medskip

However, one key point of \cite{GLC4} is that there is an alternative construction 
of $\CA_{G,\on{irred}}$ purely
in $\cG$ (i.e., spectral) terms: this sheaf is the relative homology of the space
of rational opers (with irreducible underlying local system). Equivalently, for irreducible $\sigma$, the fiber 
$\CA_{G,\sigma}$ can be identified with homology of the space of generic oper structures on 
$\sigma$. 

\medskip

Therefore, as recorded in 
\cite[Conjecture 4.5.7]{GLC4}, it suffices to show that this space of generic oper structures is 
\emph{contractible}. Note that one piece of this assertion -- 
namely, that this space is non-empty -- is a theorem of D.~Arinkin,
\cite{Ari}. 

\medskip

The contractibility assertion is easy for 
$G=GL_n$, and has recently been established in \cite{BKS} for all classical groups. So 
this provides a uniform proof of GLC for classical groups, see \cite[Theorem 4.5.11]{GLC4}. 

\medskip

The contractibility assertion for general $G$ follows a posteriori from the results of this paper. 
But conceivably, one could find an a priori proof of this assertion for a general $G$ that would 
yield a more satisfying conclusion to the geometric Langlands conjecture. 

\medskip

We remark that in \cite[Corollary 4.5.5]{GLC4}, we showed that the space of generic oper structures on 
$\sigma$ has vanishing homology in positive degrees, so it only remains to show that it  
is \emph{connected}. In fact, as $\CA_{G,\on{irred}}$ is a vector bundle, it suffices
to check this for a single irreducible local system $\sigma$ (per connected component of
$\LS_{\cG}^{\on{irred}}$).

\sssec{Proof via microlocal sheaf theory?}

To date, the most successful geometric approach to studying 
Whittaker coefficients of automorphic sheaves for general reductive groups is \cite{FR},
where it was shown that the functor \eqref{eq:coeff-cusp} is \emph{conservative}. 
One could try extending the techniques of \cite{FR} to obtain a geometric proof of
fully faithfulness. Here we briefly fantasize about one
possible form this idea might take.

\medskip 

First, \cite{FR} works primarily in the setting of sheaves \emph{with nilpotent singular support} 
as in \cite{AGKRRV}. One key point of \cite{FR} is the picture that the vacuum Whittaker
coefficient is the microstalk at the basepoint of the global nilpotent cone $\Nilp$, a picture
that was then realized in a better and more precise form in \cite{NT}. 

\medskip 

On the one hand, by \cite{Wa}, there is a category $\mu\Perv(T^*(\Bun_G))$
of \emph{microsheaves 
supported on $T^*(\Bun_G)$} and 
a subcategory $\mu\Perv_\Nilp(T^*(\Bun_G))$ of microsheaves
\emph{supported on the nilpotent cone}. Moreover, there is a 
fully faithful \emph{microlocalization} functor 
$\Perv_{\Nilp}(\Bun_G) \to \mu\Perv_{\Nilp}(T^*(\Bun_G))$. 

\medskip 

For $\overset{\circ}{\Nilp}\subset \Nilp$ the open of \emph{generically regular} 
nilpotent Higgs fields, \cite[Theorem G]{FR} morally says that
the composition 
$$\Perv_{\Nilp}(\Bun_G)_{\on{temp}} \to \Perv_{\Nilp}(\Bun_G) \to \mu\Perv_{\Nilp}(T^*(\Bun_G))
\to \mu\Perv_{\overset{\circ}{\Nilp}}(T^*(\Bun_G))$$
remains fully faithful. 

\medskip 

On the other hand, one can dream of a strengthening of \cite{NT} that expresses
\emph{all} Whittaker coefficients from the microlocalization to 
$\mu\Perv_{\overset{\circ}{\Nilp}}(T^*(\Bun_G))$, with this microlocal functor 
being fully faithful for some natural geometric reasons. 

\medskip

This would also yield another resolution to the geometric 
Langlands conjecture, and one that might teach us more than
the present paper does. 

\sssec{Proof via Verlinde gluing?}

By \cite[Theorem 3.5.6]{GLC1}, we can deduce the de Rham geometric Langlands conjecture from 
its Betti counterpart \cite{BZN}.

\medskip 

As in \cite{BZN} Sect. 4.6, a better understanding of the automorphic side of 
Betti Langlands could allow one to glue $\Shv_{\frac{1}{2},\Nilp}^{\on{Betti}}(\Bun_G)$
via degeneration to a nodal curve, using Bezrukavnikov-style equivalences as local input. 
As we understand, there is active and on-going work of D.~Nadler and Z.~Yun in this direction.

\sssec{Proof via arithmetic?}

Suppose for simplicity that $G$ is adjoint and the genus is at least $2$,
so that $\LS_{\cG}^{\on{irred}}$ is connected. 

\medskip 

Recall that $\CA_G|_{\LS_\cG^{\on{irred}}}$ is a vector bundle. 
As mentioned in \secref{sss:opers?}, it follows 
that $\CA_G|_{\LS_\cG^{\on{irred}}}$ has one-dimensional fibers if it has
a one-dimensional fiber at \emph{any} point.
This in turn should reduce by specialization to a statement in characteristic $p$. Using
the intimate relationship between geometric and arithmetic Langlands established
in \cite{AGKRRV3}, one should be able to reduce de Rham geometric Langlands to 
a suitable multiplicity one statement for 
essentially \emph{any} class of unramified cusp forms
(equipped with the action of V.~Lafforgue's excursion operators).

\medskip 

With that said, we are not aware of any such arithmetic results for general reductive groups.

\sssec{Proof via... something else?}

We do not intend to represent the above as 
an exhaustive summary of discussions that have occurred. More seriously, we anticipate
future innovations whose form we do not yet know. 

\ssec{Contents}

\sssec{}

Let us briefly review the actual contents of this paper.

\medskip

In \secref{s:review} we summarize the results of \cite{GLC1,GLC2,GLC3,GLC4} that will be used in this paper.

\medskip

In \secref{s:sc} we show that it is sufficient to prove GLC in the special case when the group $G$ is 
almost simple and simply-connected.

\medskip

In \secref{s:low genus} we prove GLC when the genus $g$ of our curve is either $0$ or $1$.

\medskip

In \secref{s:pi 1} we express the fundamental groupoid of $\LS_\cG^{\on{irred}}$ in terms of $Z_G$.
(In particular, we show that if $G$ is simple, then $\LS_\cG^{\on{irred}}$ is connected and simply-connected,
under the assumption that $g\geq 2$, and excluding the case $g=2$ and $G=PGL_2$.) 

\medskip

In \secref{s:proof} we prove GLC for curves of genus $\geq 2$, essentially elaborating on the outline in \secref{ss:proof outline}. Here we also state some further results whose
proofs are given in later sections.

\medskip

In \secref{s:Poinc} we calculate (in the case $g\geq 2$) endomorphisms of the vacuum Poincar\'e object,
and show that they consist essentially only of scalars. 

\medskip

In \secref{s:LS} we prove the properties concerning algebraic geometry and topology of the stacks $\Bun_\cG$ and $\LS_\cG$.

\medskip

In \secref{s:gerbes} we introduce the 2-categorical Fourier-Mukai transform and use it 
establish the compatibility between the action of $\Bun_{Z_G}$ on $\Dmod_{\frac{1}{2}}(\Bun_G)$ and
the spectral action. This material is used to adapt the discussion of 
\secref{ss:proof outline} to the case when $G$ is not of adjoint type. As a byproduct, we obtain
a version of geometric Langlands for non-pure inner twists of $G$. 

\sssec{Conventions and notation}

This paper follows the conventions and notation of the \cite{GLC1,GLC2,GLC3,GLC4} series. 

\ssec{Acknowledgements}

We would like to thank our collaborators on this project: D. ~Arinkin, J.~Campbell, D.~Beraldo, L.~Chen, J.~Faergeman, K.~Lin and N.~Rozenblyum
for sharing their insight and many fruitful discussions. 

\medskip

The first author wishes to thank Peter Scholze for a crucial discussion during
which which the idea to use the translation functors 
by points of $\Bun_{Z_G}$ was conceived (otherwise, this paper would only treat the case of $G$ with a connected center). 

\medskip

We thank Z.~Yun for informing us of the paper \cite{BFM}.

\medskip

We reiterate our gratitude to the mathematicians acknowledged in \cite{GLC1} for 
sharing their ideas on the geometric Langlands correspondence:
D.~Ben-Zvi, J.~Bernstein, R.~Bezrukavnikov, A.~Braverman, J.~Campbell, P.~Deligne, 
G.~Dhillon, R.~Donagi, L.~Fargues, B.~Feigin, M.~Finkelberg, E.~Frenkel, D.~Gaiotto, A.~Genestier, V.~Ginzburg, 
S.~Gukov, J.~Heinloth, A.~Kapustin, D.~Kazhdan, 
V.~Lafforgue, G.~Laumon, G.~Lusztig, S.~Lysenko, I.~Mirkovi\'c, D.~Nadler, T.~Pantev, P.~Scholze, 
C.~Teleman, Y.~Varshavsky, K.~Vilonen, E.~Witten, C.~Xue, Z.~Yun and X.~Zhu.

\medskip

We particularly recognize A.~Beilinson and V.~Drinfeld for conceiving the whole subject.

\medskip 

The work of D.G. was supported by NSF grant DMS-2005475. 
The work of S.R. was supported by NSF grant DMS-2101984 and a Sloan Research Fellowship 
while this work was in preparation.

\section{Summary of the preceding results} \label{s:review}

In this section we summarize the results from \cite{GLC1,GLC2,GLC3,GLC4} that will be used in 
the present paper for the proof of GLC. 

\ssec{The Langlands functor and its categorical properties}

\sssec{}

In \cite[Sect. 1]{GLC1}, we constructed a functor
$$\BL_G:\Dmod_{\frac{1}{2}}(\Bun_G) \to \IndCoh_\Nilp(\LS_{\cG}).$$
It satisfies the following properties:

\begin{itemize}

\item The functor $\BL_G$ admits a left adjoint $\BL^L_G$, 
see \cite[Theorem 16.1.2]{GLC3}; 

\item The functor $\BL_G$ is conservative, see \cite[Theorem 1.6.2]{GLC4}. 

\end{itemize}

\sssec{}

The geometric Langlands conjecture (GLC) says:

\begin{conj} \label{c:GLC}
The functor $\BL_G$ is an equivalence. 
\end{conj}

\sssec{}

By the above, \conjref{c:GLC} is equivalent to the fact that the unit of the adjunction
\begin{equation} \label{e:unit initial}
\on{Id}_{\IndCoh_\Nilp(\LS_{\cG})}\to \BL_G\circ \BL^L_G
\end{equation}
is an isomorphism.

\ssec{Spectral properties}

\sssec{}

Recall that the Hecke action gives rise to an action of the monoidal category $\QCoh(\LS_\cG)$
on $\Dmod_{\frac{1}{2}}(\Bun_G)$, see \cite[Sect. 1.2]{GLC1}.

\medskip

The functor $\BL_G$ has the following features with respect to this action:

\begin{itemize}

\item The functor $\BL_G$ is $\QCoh(\LS_\cG)$-linear, i.e., it intertwines the actions of $\QCoh(\LS_\cG)$ on the two sides, see
\cite[Sect. 1.7]{GLC1}. As the monoidal category $\QCoh(\LS_\cG)$ is rigid,
this implies that the functor $\BL_G^L$ is also $\QCoh(\LS_\cG)$-linear; 

\medskip

\item The $\QCoh(\LS_\cG)$-linear monad $\BL_G\circ \BL^L_G$ on $\IndCoh_\Nilp(\LS_{\cG})$
is given by a (uniquely defined) associative algebra object 
$$\CA_G\in \QCoh(\LS_\cG),$$
see \cite[Theorem 16.4.2]{GLC3}.

\end{itemize}

\sssec{}

The unit map
\begin{equation} \label{e:AG initial}
\CO_{\LS_\cG}\to \CA_G
\end{equation}
for the associative algebra $\CA_G$ in $\QCoh(\LS_\cG)$ 
gives rise to the map \eqref{e:unit initial}. 
Therefore, to see that \eqref{e:unit initial} is an isomorphism of functors, it suffices to show that 
\eqref{e:AG initial} is an isomorphism of quasi-coherent sheaves.

\ssec{Restriction to the irreducible locus}

\sssec{}

Let
$$\LS^{\on{red}}_\cG \subset \LS_\cG$$
be the closed substack consisting of \emph{reducible} local systems, i.e., the union of the images of the
(proper) maps
$$\LS_\cP\to \LS_\cG,$$
where $\cP \subsetneq \cG$ are proper parabolics of $\cG$.

\medskip

Let
$$\LS_\cG^{\on{irred}}\overset{\jmath}\hookrightarrow \LS_\cG$$
denote the embedding of the complement to 
$\LS^{\on{red}}_\cG$, i.e., $\LS_\cG^{\on{irred}}$ is 
the open substack of \emph{irreducible} local systems. 

\sssec{}

Let
\begin{equation} \label{e:embed red}
\QCoh(\LS_\cG)_{\on{red}}\subset \QCoh(\LS_\cG)
\end{equation}
be the full subcategory consisting of objects set-theoretically supported on $\LS^{\on{red}}_\cG$, i.e.,
$$\QCoh(\LS_\cG)_{\on{red}}=\on{ker}\left(\jmath^*:\QCoh(\LS_\cG)\to \QCoh(\LS^{\on{irred}}_\cG)\right).$$

\medskip

The above faithful embedding \eqref{e:embed red} admits a right adjoint, denoted $\wh\iota^!$. Explicitly,
$$\wh\iota^!\simeq \on{Fib}(\on{Id}\to \jmath_*\circ \jmath^*(-)).$$

\sssec{}  \label{sss:on Eis}

The following is the main result of \cite{GLC3} (it is equivalent to Theorem 17.1.2 in {\it loc. cit.}, see Sect. 17.3.3):

\medskip

\begin{itemize}

\item The map $\wh\iota^!(\CO_{\LS_\cG})\to \wh\iota^!(\CA_G)$ induced by \eqref{e:AG initial} is an isomorphism.

\end{itemize}

\sssec{} \label{sss:red to irred}

Set
$$\CA_{G,\on{irred}}:=\jmath^*(\CA_G).$$

Given the isomorphism in \secref{sss:on Eis}, we obtain that the fact that \eqref{e:AG initial} is an isomorphism 
(and hence, the statement of GLC) is equivalent to the fact that 
the map 
\begin{equation} \label{e:AG irred}
\CO_{\LS^{\on{irred}}_\cG}\to \CA_{G,\on{irred}}, 
\end{equation}
induced by \eqref{e:AG initial} is an isomorphism.

\ssec{Properties of \texorpdfstring{$\CA_{G,\on{irred}}$}{AG}} \label{ss:AG nice}

\sssec{}

In this subsection we will assume that $G$ is semi-simple. In this case $\LS_\cG^{\on{irred}}$ is a classical smooth algebraic stack.

\sssec{} \label{sss:AG nice}

The following property of $\CA_{G,\on{irred}}$ is one of the two main results of
the paper \cite{GLC4}, see Theorem 3.1.8 in {\it loc. cit.}:

\medskip

\begin{itemize}

\item The object $\CA_{G,\on{irred}}\in \QCoh(\LS_\cG^{\on{irred}})$ is a vector bundle (in particular, it is concentrated in cohomological degree $0$);

\end{itemize} 

\sssec{}  \label{sss:AG conn}

In addition, we have: 

\medskip

\begin{itemize}

\item The object $\CA_{G,\on{irred}}$ is canonically of the form
$\oblv^l(\CF)$, where $\CF\in \Dmod(\LS_\cG^{\on{irred}})$ and 
$$\oblv^l:\Dmod(\LS_\cG^{\on{irred}})\to \QCoh(\LS_\cG^{\on{irred}})$$
denotes the ``left" forgetful functor. This is \cite[Corollary 4.2.5]{GLC4}. 

\medskip

\item The above object $\CF$ is a local system
with a finite monodromy (in particular, it is concentrated in cohomological degree $0$). This is \cite[Proposition 4.2.8]{GLC4}. 

\end{itemize}

\section{Reduction to the case when \texorpdfstring{$G$}{G} is almost simple and simply-connected} \label{s:sc}

For a number of technical reasons\footnote{One reason is not very serious: if $G$ has a center of positive dimension, the stack 
$\LS^{\on{irred}}_\cG$ is not smooth, which is a silly annoyance. The more serious reason is that when $g=2$ we will need to assume
that $G$ has no factors of $A_1$ in its Dynkin diagram, see the preamble to \secref{ss:proof of GLC}.}, in the main argument in the proof of GLC,
we will need to assume that the group $G$ is almost simple (by which we mean simple modulo a finite center)
and simply-connected. In this section we will perform a reduction to this case. 

\ssec{Compatibility between Langlands functors}

\sssec{}

Let $G_1$ and $G_2$ be a pair of reductive groups (over our ground field $k$), and let 
$$\phi:G_1\to G_2$$
be an \emph{almost isogeny}, i.e., a map that induces an isogeny of their derived groups.

\medskip

By a slight abuse of notation we will denote by the same symbol $\phi$ the induced map
$$\Bun_{G_1}\to \Bun_{G_2}.$$

\sssec{}

The map $\phi$ induces a map
$$\phi^\vee:\cG_2\to \cG_1.$$

By a slight abuse of notation, we will denote by the same symbol $\phi^\vee$ the induced map
$$\LS_{\cG_2}\to \LS_{\cG_1}.$$

\sssec{} \label{sss:isog compat}

Note that the functor
\begin{equation} \label{e:phi!}
\phi^!:\Dmod_{\frac{1}{2}}(\Bun_{G_2}) \to \Dmod_{\frac{1}{2}}(\Bun_{G_1})
\end{equation} 
is linear with respect to the action of $\Rep(\cG_1)_\Ran$, where the action on 
$\Dmod_{\frac{1}{2}}(\Bun_{G_2})$ is via the functor
$$\Rep(\cG_1)_\Ran\to \Rep(\cG_2)_\Ran,$$
given by restriction along $\phi^\vee$, see \secref{sss:Sat restr}. 

\medskip

Hence, \eqref{e:phi!} is $\QCoh(\LS_{\cG_1})$-linear, where the action on 
$\Dmod_{\frac{1}{2}}(\Bun_{G_2})$ is via 
$$(\phi^\vee)^*:\QCoh(\LS_{\cG_1})\to \QCoh(\LS_{\cG_2}).$$

\sssec{}

Note also that the functor
$$(\phi^\vee)^\IndCoh_*:\IndCoh(\LS_{\cG_2})\to \IndCoh(\LS_{\cG_1})$$
sends
$$\IndCoh_\Nilp(\LS_{\cG_2})\to \IndCoh_\Nilp(\LS_{\cG_1}).$$

\medskip

This follows, e.g., from \cite[Proposition 7.1.3(b)]{AG}. 

\sssec{}

We will prove:

\begin{prop} \label{p:alm isogeny} 
The following diagram of functors commutes
\begin{equation} \label{e:alm isogeny}
\CD
\Dmod_{\frac{1}{2}}(\Bun_{G_1}) @>{\BL_{G_1}}>> \IndCoh_\Nilp(\LS_{\cG_1}) \\
@A{\phi^!}AA  @A{(\phi^\vee)^\IndCoh_*}AA  \\
\Dmod_{\frac{1}{2}}(\Bun_{G_2}) @>{\BL_{G_2}}>> \IndCoh_\Nilp(\LS_{\cG_2}).
\endCD
\end{equation} 
Moreover, this datum of commutativity is compatible with the action of $\QCoh(\LS_{\cG_1})$. 
\end{prop} 

\ssec{Proof of \propref{p:alm isogeny}}

%
%
%
%
%

\sssec{}

First, we claim that it is sufficient show that the outer square in the following diagram commutes:
\begin{equation} \label{e:alm isogeny'}
\CD
\Dmod_{\frac{1}{2}}(\Bun_{G_1}) @>{\BL_{G_1}}>> \IndCoh_\Nilp(\LS_{\cG_1}) @>>> \QCoh(\LS_{\cG_1}) \\
@A{\phi^!}AA  @A{(\phi^\vee)_*}AA  @A{(\phi^\vee)_*}AA \\
\Dmod_{\frac{1}{2}}(\Bun_{G_2}) @>{\BL_{G_2}}>> \IndCoh_\Nilp(\LS_{\cG_2}) @>>> \QCoh(\LS_{\cG_2}),
\endCD
\end{equation} 
where the composite horizontal functors are 
$$\BL_{G_1,\on{coarse}} \text{ and } \BL_{G_2,\on{coarse}},$$
respectively (see \cite[Sect. 1.4]{GLC1}). 

\medskip

Indeed, suppose that \eqref{e:alm isogeny'} commutes, and let us show that this uniquely upgrades to the commutation 
of \eqref{e:alm isogeny}. 

\sssec{} \label{sss:phi coconn}

It is enough to show that the two circuits in \eqref{e:alm isogeny} are isomorphic when restricted to the subcategory 
of compact objects in $\Dmod_{\frac{1}{2}}(\Bun_{G_2})$. Since the functor
$$\IndCoh_\Nilp(\LS_{\cG_1}) \to \QCoh(\LS_{\cG_1})$$
is fully faithful on 
$$\IndCoh_\Nilp(\LS_{\cG_1})^{>-\infty}\subset \IndCoh_\Nilp(\LS_{\cG_1}),$$
it suffices to show that both circuits in \eqref{e:alm isogeny}, when restricted to
$$\Dmod_{\frac{1}{2}}(\Bun_{G_2})^c\subset \Dmod_{\frac{1}{2}}(\Bun_{G_2})$$
map to $\IndCoh_\Nilp(\LS_{\cG_1})^{>-\infty}$.

\sssec{}

By the construction in \cite[Corollary 1.6.5]{GLC1}, the functors $\BL_{G_i}$ send
$$\Dmod_{\frac{1}{2}}(\Bun_{G_i})^c\to \IndCoh_\Nilp(\LS_{\cG_i})^{>-\infty},$$
$i=1,2$.

\medskip

The statement about the counterclockwise circuit follows now because the functor 
$(\phi^\vee)_*$ has finite cohomological amplitude.

\medskip

We now prove the statement for the clockwise circuit. Note that the map  
$$\phi:\Bun_{G_1}\to \Bun_{G_2}$$
is not schematic, so the functor 
\begin{equation} \label{e:phi! Bun}
\phi^!:\Dmod_{\frac{1}{2}}(\Bun_{G_2})\to \Dmod_{\frac{1}{2}}(\Bun_{G_2}).
\end{equation}
does \emph{not} preserve compactness. To circumvent this, we argue as follows. 

\medskip

First, we note that \eqref{e:phi! Bun}
does preserve compactness if the initial homomorphism $\phi$ has a finite kernel. 
So in this case, the clockwise circuit does send $\Dmod_{\frac{1}{2}}(\Bun_{G_2})^c$ to
$\IndCoh_\Nilp(\LS_{\cG_1})^{>-\infty}$.

\medskip

Thus, let us temporarily assume having proved \propref{p:alm isogeny} for such homomorphisms. 

\sssec{}

In general, we decompose $\phi$ as 
$$G_1\overset{\phi'}\to G'_2\overset{\psi}\to G_2,$$
where 
$$G'_2:=G_2\times G_{1,\on{ab}},$$
and $\psi$ is the projection. The assertion of \propref{p:alm isogeny} trivially holds for $\psi$. 

\medskip

Note now that the homomorphism $\phi'$ does have a finite kernel. Hence, by assumption,
the functor $\BL_{G_1}\circ (\phi')^!$ is isomorphic to $((\phi')^\vee)_*\circ \BL_{G'_2}$, and hence
$$\BL_{G_1}\circ \phi^!\simeq \BL_{G_1}\circ (\phi')^!\circ \psi^!\simeq
((\phi')^\vee)_*\circ \BL_{G'_2} \circ \psi^! \simeq ((\phi')^\vee)_*\circ (\psi^\vee)_*\circ \BL_{G_2}\simeq
(\phi^\vee)_*\circ \BL_{G_2}.$$

In particular, the clockwise circuit for $\phi$ also sends 
$$\Dmod_{\frac{1}{2}}(\Bun_{G_2})^c\to \IndCoh_\Nilp(\LS_{\cG_1})^{>-\infty}.$$

\sssec{}

The same argument implies the following:

\medskip

Once we prove that the outer circuit in 
\eqref{e:alm isogeny'} commutes as a diagram of
$\QCoh(\LS_{\cG_1})$-module categories, it would follow that its
restriction to compact objects commutes as a diagram
of $\on{Perf}(\LS_{\cG_1})$-module categories, so the
resulting left commuting square in \eqref{e:alm isogeny'}
is one of $\on{Perf}(\LS_{\cG_1})$-module categories,
and hence of 
$\on{Ind}(\on{Perf}(\LS_{\cG_1}))\simeq\QCoh(\LS_{\cG_1})$-module categories\footnote{
The latter equivalence follows from the fact that $\LS_\cG$ admits a (global) embedding 
into a smooth stack, see \cite[Sect. 10.6.2 and Proposition 10.6.6(a)]{AG}.}.

\sssec{}

Let us now prove the commutativity of 
\begin{equation} \label{e:alm isogeny prime}
\CD
\Dmod_{\frac{1}{2}}(\Bun_{G_1}) @>{\BL_{G_1,\on{coarse}}}>> \QCoh(\LS_{\cG_1}) \\
@A{\phi^!}AA  @A{(\phi^\vee)_*}AA  \\
\Dmod_{\frac{1}{2}}(\Bun_{G_2}) @>{\BL_{G_2,\on{coarse}}}>> \QCoh(\LS_{\cG_2}).
\endCD
\end{equation} 

It will follow from the construction that the data of commutativity is compatible with the Hecke action of $\Rep(\cG_1)_\Ran$,
and hence also of $\QCoh(\LS_{\cG_1})$. 

\sssec{}

Since the functor
$$\Gamma^{\on{spec}}_{\cG_1}:\QCoh(\LS_{\cG_1}) \to \Rep(\cG_1)_{\Ran}$$
is fully faithful (see \cite[Proposition 1.1.4]{GLC4}), 
it suffices to show that the outer square in the concatenation of \eqref{e:alm isogeny prime} with the commutative diagram
\begin{equation} \label{e:alm isogeny''}
\CD
\QCoh(\LS_{\cG_1}) @>>> \Rep(\cG_1)_{\Ran} \\
@A{(\phi^\vee)_*}AA @AA{\on{coInd}^{\phi^\vee}}A \\
\QCoh(\LS_{\cG_2}) @>>> \Rep(\cG_2)_{\Ran}
\endCD
\end{equation} 
commutes, where 
$$\on{coInd}^{\phi^\vee}:\Rep(\cG_2)\to \Rep(\cG_1)$$
is the functor of co-induction\footnote{It may be more common to call this operation just ``induction".}, i.e., the right adjoint to the functor
$$\Res^{\phi^\vee}:\Rep(\cG_1)\to \Rep(\cG_2).$$

I.e., we need to establish the commutativity of 
\begin{equation} \label{e:alm isogeny prime bis}
\CD
\Dmod_{\frac{1}{2}}(\Bun_{G_1}) @>{\BL_{G_1,\on{coarse}}}>> \QCoh(\LS_{\cG_1})  @>>> \Rep(\cG_1)_{\Ran}  \\
@A{\phi^!}AA  & & @AA{\on{coInd}^{\phi^\vee}}A  \\
\Dmod_{\frac{1}{2}}(\Bun_{G_2}) @>{\BL_{G_2,\on{coarse}}}>> \QCoh(\LS_{\cG_2}) @>>> \Rep(\cG_2)_{\Ran}  .
\endCD
\end{equation} 

\sssec{}

The map (of factorization ind-schemes) 
$$\phi:\Gr_{G_1}\to \Gr_{G_2}$$
induces a pair of adjoint functors
\begin{equation} \label{e:pullpush Gr}
\phi_!:\Dmod_{\frac{1}{2}}(\Gr_{G_1}) \rightleftarrows \Dmod_{\frac{1}{2}}(\Gr_{G_2}):\phi^!
\end{equation}
(as factorization categories), which in turn induce adjoint functors
$$\phi_!:\Whit(G_1) \rightleftarrows \Whit(G_2):\phi^!$$
(also as factorization categories). 

\sssec{}

In addition, we claim that the functors \eqref{e:pullpush Gr} also induce (factorization) functors
$$\Sph_{G_1} \simeq \Dmod_{\frac{1}{2}}(\Gr_{G_1})^{\fL^+(G_1)}   \overset{'\!\phi^!}\leftarrow \Dmod_{\frac{1}{2}}(\Gr_{G_2})^{\fL^+(G_2)}\simeq \Sph_{G_2}$$
and 
$$\Sph_{G_1} \simeq \Dmod_{\frac{1}{2}}(\Gr_{G_1})^{\fL^+(G_1)}   \overset{'\!\phi_!}\to \Dmod_{\frac{1}{2}}(\Gr_{G_2})^{\fL^+(G_2)}\simeq \Sph_{G_2}$$
compatible with the forgetful functors $\Sph_{G_i}\to \Dmod_{\frac{1}{2}}(\Gr_{G_i})$. 

\medskip

Namely, the functor $'\!\phi^!$ is
$$\Dmod_{\frac{1}{2}}(\Gr_{G_2})^{\fL^+(G_2)} \to \Dmod_{\frac{1}{2}}(\Gr_{G_2})^{\fL^+(G_1)} 
\overset{\phi^!}\to \Dmod_{\frac{1}{2}}(\Gr_{G_1})^{\fL^+(G_1)},$$
where the first arrow is the functor of forgetting $\fL^+(G_2)$-equivariance to $\fL^+(G_1)$-equivariance.

\medskip

The functor $'\!\phi_!$ is
$$\Dmod_{\frac{1}{2}}(\Gr_{G_1})^{\fL^+(G_1)} \overset{\phi_!}\to \Dmod_{\frac{1}{2}}(\Gr_{G_2})^{\fL^+(G_1)} \to 
\Dmod_{\frac{1}{2}}(\Gr_{G_2})^{\fL^+(G_2)},$$ 
where the second arrow is pullback with respect to the projection 
$$\fL^+(G_2)\backslash \Gr_{G_2} \to \fL^+(G_1)\backslash \Gr_{G_2}$$
equal to
\begin{multline*} 
\fL^+(G_2)\backslash \Gr_{G_2}=\fL^+(G_2)\backslash \fL(G_2)/\fL^+(G_2)\simeq 
\on{pt}/\fL^+(Z_{G_2})\overset{\on{pt}/\fL^+(Z_{G_1})}\times (\fL^+(G_1)\backslash \fL(G_2)/\fL^+(G_2)) \simeq \\
\simeq  
\on{pt}/\fL^+(Z_{G_2})\overset{\on{pt}/\fL^+(Z_{G_1})}\times \on{pt}/\fL^+(Z_{G_2})
\overset{\on{pt}/\fL^+(Z_{G_2})}\times (\fL^+(G_1)\backslash \fL(G_2)/\fL^+(G_2)) \to  \\
\to \on{pt}/\fL^+(Z_{G_2})\overset{\on{pt}/\fL^+(Z_{G_2})}\times (\fL^+(G_1)\backslash \fL(G_2)/\fL^+(G_2)) \simeq
\fL^+(G_1)\backslash \fL(G_2)/\fL^+(G_2)=\fL^+(G_1)\backslash \Gr_{G_2},
\end{multline*}
where:

\begin{itemize}

\item The 2nd isomorphism is obtained from the identification $\fL^+(G_2)\simeq \fL^+(Z_{G_2})\overset{\fL^+(Z_{G_1})}\times \fL^+(G_1)$;

\medskip

\item The 3rd isomorphism is obtained from noticing that the action of $\on{pt}/\fL^+(Z_{G_1})$ on the double quotient $\fL^+(G_1)\backslash \fL(G_2)/\fL^+(G_2)$
on the left equals the action obtained from the projection $\on{pt}/\fL^+(Z_{G_1})\to \on{pt}/\fL^+(Z_{G_2})$, and the action of $\on{pt}/\fL^+(Z_{G_2})$
on $\fL^+(G_1)\backslash \fL(G_2)/\fL^+(G_2)$ on the right;

\medskip

\item The 4th arrow is induced by the multiplication map
$$\on{pt}/\fL^+(Z_{G_2})\overset{\on{pt}/\fL^+(Z_{G_1})}\times \on{pt}/\fL^+(Z_{G_2})\to \on{pt}/\fL^+(Z_{G_2}).$$

\end{itemize} 

\begin{rem}
Note that the functor $'\!\phi_!$ is \emph{not} the left adjoint of $'\!\phi^!$: the latter would involve an additional
step of !-averaging from $\fL^+(G_1)$-equivariance to $\fL^+(G_2)$-equivariance.

\medskip

Note also that the functor $'\!\phi_!$ is \emph{not} monoidal. 

\end{rem}

\sssec{} \label{sss:Sat restr}

It follows from the construction of the \emph{naive} geometric Satake functor
$\on{Sat}_G^{\on{nv}}$ (see \cite[Sect. 6.28-6.35]{Ras}) that the diagram (of factorization functors)
\begin{equation} \label{e:Sph restr}
\CD
\Sph_{G_1} @<{\Sat^{\on{nv}}_{G_1}}<<  \Rep(\cG_1) \\
@V{'\!\phi_!}VV @VV{\Res^{\phi^\vee}}V  \\
\Sph_{G_2} @<{\Sat^{\on{nv}}_{G_2}}<<  \Rep(\cG_2)
\endCD
\end{equation} 
commutes. 

\medskip

Unwinding, we obtain that in the global situation, the functor
$$\phi^!: \Dmod_{\frac{1}{2}}(\Bun_{G_2})  \to \Dmod_{\frac{1}{2}}(\Bun_{G_1})$$
is compatible with the actions of $\Rep(\cG_1)_\Ran$, where the action on $\Dmod_{\frac{1}{2}}(\Bun_{G_2})$
is via $\Res^{\phi^\vee}$. 

\sssec{}

By the definition of the Casselman-Shalika equivalence $\on{CS}_G$ (see \cite[Sect. 1.4]{GLC2}), the commutativity of
\eqref{e:Sph restr} implies that the diagram (of factorization functors)
$$
\CD
\Whit(G_1) @>{\on{CS}_{G_1}}>>  \Rep(\cG_1) \\
@V{\phi_!}VV @VV{\Res^{\phi^\vee}}V  \\
\Whit(G_2) @>{\on{CS}_{G_2}}>>  \Rep(\cG_2)
\endCD
$$
also commutes. 

\medskip

Since the horizontal arrows in the latter diagram are equivalences, we obtain that the diagram
obtained by passing to right adjoints along the vertical arrows also commutes:

$$
\CD
\Whit(G_1) @>{\on{CS}_{G_1}}>>  \Rep(\cG_1) \\
@A{\phi^!}AA  @AA{\on{coInd}^{\phi^\vee}}A  \\
\Whit(G_2) @>{\on{CS}_{G_2}}>>  \Rep(\cG_2).
\endCD
$$

\sssec{}

Combining this with\footnote{We recall that the subscript ``Ran" in the next few formulas 
indicates that we are taking global sections of the corresponding crystal of categories over the Ran space, see \cite[Sect. B.11.1]{GLC2}.}
the commutative diagrams \cite[Equation (18.4)]{GLC2}
$$
\CD
\Whit(G_i)_\Ran @>>> \Rep(\cG_i)_{\Ran}  \\
@A{\on{coeff}_{G_i}[2\delta_{N_{\rho(\omega_X)}}]}AA @AA{\Gamma^{\on{spec}}_{\cG_i}}A \\
\Dmod_{\frac{1}{2}}(\Bun_{G_i}) @>{\BL_{G_i,\on{course}}}>> \QCoh(\LS_{\cG_i})
\endCD
$$
for $i=1,2$, we obtain that the commutativity of \eqref{e:alm isogeny prime bis} is equivalent to the
commutativity of
\begin{equation} \label{e:alm isogeny'''}
\CD
\Dmod_{\frac{1}{2}}(\Bun_{G_1})  @>{\on{coeff}_{G_1}}>> \Whit(G_1)_\Ran \\
@A{\phi^!}AA @AA{\phi^!}A \\
\Dmod_{\frac{1}{2}}(\Bun_{G_2})  @>{\on{coeff}_{G_2}}>> \Whit(G_2)_\Ran. \\
\endCD
\end{equation} 

\sssec{}

Passing to left adjoints, the commutativity of \eqref{e:alm isogeny'''} is equivalent to the 
commutativity of 
$$
\CD
\Dmod_{\frac{1}{2}}(\Bun_{G_1})  @<{\on{Poinc}_{G_1,!}}<< \Whit(G_1)_\Ran \\
@V{\phi_!}VV @VV{\phi_!}V \\
\Dmod_{\frac{1}{2}}(\Bun_{G_2})  @<{\on{Poinc}_{G_2,!}}<< \Whit(G_2)_\Ran, \\
\endCD
$$
while the latter is tautological. 

\qed[\propref{p:alm isogeny}]

\ssec{Changing the group}

\sssec{}

Passing to left adjoint functors in \eqref{e:alm isogeny}, we obtain a commutative diagram
$$
\CD
\Dmod_{\frac{1}{2}}(\Bun_{G_1}) @<{\BL^L_{G_1}}<< \IndCoh_\Nilp(\LS_{\cG_1}) \\
@V{\phi_!}VV  @VV{(\phi^\vee)^*}V \\
\Dmod_{\frac{1}{2}}(\Bun_{G_2}) @<{\BL^L_{G_2}}<< \IndCoh_\Nilp(\LS_{\cG_2}),
\endCD
$$
compatible with the action of $\QCoh(\LS_{\cG_1})$. Tensoring up, we obtain a commutative diagram
\begin{equation} \label{e:alm isogeny ten}
\CD
\QCoh(\LS_{\cG_2})\underset{\QCoh(\LS_{\cG_1})}\otimes 
\Dmod_{\frac{1}{2}}(\Bun_{G_1}) @<{\BL^L_{G_1}}<< \QCoh(\LS_{\cG_2})\underset{\QCoh(\LS_{\cG_1})}\otimes  \IndCoh_\Nilp(\LS_{\cG_1}) \\
@VVV  @VVV \\
\Dmod_{\frac{1}{2}}(\Bun_{G_2}) @<{\BL^L_{G_2}}<< \IndCoh_\Nilp(\LS_{\cG_2}).
\endCD
\end{equation}

\sssec{}

We claim: 
\begin{prop} \label{p:change group 1}
The functor
$$\QCoh(\LS_{\cG_2})\underset{\QCoh(\LS_{\cG_1})}\otimes  \IndCoh_\Nilp(\LS_{\cG_1})\to \IndCoh_\Nilp(\LS_{\cG_2})$$
is an equivalence.
\end{prop} 

Indeed, this is particular case of \cite[Corollary 7.6.2]{AG}
(extended to stacks in a straightforward way\footnote{One uses the fact that for any affine scheme $S$ mapping to $\LS_{\cG_i}$,
the operation $\QCoh(S)\underset{\QCoh(\LS_{\cG_i})}\otimes -$ commutes with limits on $\QCoh(\LS_{\cG_i})$-module categories.}). 

\sssec{}

In \secref{ss:change group 2} below, we will prove:
\begin{thm} \label{t:change group 2}
The functor
\begin{equation} \label{e:change group 2}
\QCoh(\LS_{\cG_2})\underset{\QCoh(\LS_{\cG_1})}\otimes  \Dmod_{\frac{1}{2}}(\Bun_{G_1})  \to \Dmod_{\frac{1}{2}}(\Bun_{G_2})
\end{equation}
is an equivalence.
\end{thm} 

\sssec{}

Combining \propref{p:change group 1} and \thmref{t:change group 2}, we obtain:

\begin{cor} \label{c:change group}
Suppose that the functor $\BL_{G_1}$ is an equivalence. Then so is $\BL_{G_2}$. 
\end{cor}

\sssec{}

For a given reductive group $G$, take $G_2=G$, and let $G_1:=G_{\on{sc}}$ be the simply-connected cover of its derived group.
Let $\phi$ be the canonical map 
$$G_{\on{sc}}\to G.$$

As a particular case of \corref{c:change group}, we obtain:

\begin{cor} \label{c:prel sc}
If GLC holds for $G_{\on{sc}}$, then it also holds for $G$.
\end{cor} 

\sssec{}

Let $G=G_1\times G_2$. We have
$$\Bun_G\simeq \Bun_{G_1}\times \Bun_{G_2},$$
and hence
$$\Dmod_{\frac{1}{2}}(\Bun_G) \simeq \Dmod_{\frac{1}{2}}(\Bun_{G_1}) \otimes \Dmod_{\frac{1}{2}}(\Bun_{G_2}).$$

We also have:
$$\cG\simeq \cG_1\times \cG_2,$$
and hence
$$\LS_\cG\simeq \LS_{\cG_1}\times \LS_{\cG_2},$$
so that
$$\IndCoh_\Nilp(\LS_\cG)\simeq \IndCoh_\Nilp(\LS_{\cG_1})\otimes \IndCoh_\Nilp(\LS_{\cG_2}).$$

It is clear that under the above identifications,
$$\BL_G\simeq \BL_{G_1}\otimes \BL_{G_2}.$$

Combining with \corref{c:prel sc}, we obtain:

\begin{cor} \label{c:sc}
If GLC holds for all $G$ that are almost simple and simply-connected, then it holds for any
reductive $G$.
\end{cor} 

\ssec{Proof of \thmref{t:change group 2}} \label{ss:change group 2}

\sssec{}

We will distinguish two special types of homomorphisms $\phi$:

\medskip

\noindent Type A: $\phi$ is injective;

\smallskip

\noindent Type B: $\phi$ is surjective with a connected kernel.

\medskip

Note that any $\phi$ can be factored as a composition
$$G_1\to G'_1\to G'_2\to G_2,$$
where:

\begin{itemize}

\item The homomorphisms $G_1\to G'_1$ and $G'_2\to G_2$ are of type $B$;

\item $G'_1\to G'_2$ is of type $A$. 

\end{itemize}

\medskip

So, it is enough to prove \thmref{t:change group 2} for $\phi$ of each of the above two types separately. 


\sssec{Proof for type B} \label{sss:type B}

Set 
$$T:=\on{ker}(\phi).$$ 

(In this subsection $T$ is just a torus, i.e., it is \emph{not} the Cartan subgroup of either $G_1$ or $G_2$.)

\medskip

We have an action of $\Bun_T$ on $\Bun_{G_1}$, and 
\begin{equation} \label{e:Bun as quot}
\Bun_{G_2}\simeq \Bun_{G_1}/\Bun_T.
\end{equation} 

We also have a projection
$$\LS_{\cG_1}\to \LS_\cT,$$
and 
\begin{equation} \label{e:LS as fiber}
\LS_{\cG_2}\simeq \on{pt}\underset{\LS_\cT}\times \LS_{\cG_1}
\end{equation} 
where $\on{pt}\to \LS_\cT$ is the unit point. 

\medskip

From \eqref{e:Bun as quot} we obtain that the naturally defined functor 
\begin{equation} \label{e:Bun as quot bis}
\Vect \underset{\Dmod(\Bun_T)}\otimes \Dmod_{\frac{1}{2}}(\Bun_{G_1})\to \Dmod_{\frac{1}{2}}(\Bun_{G_2})
\end{equation} 
is an equivalence, where:

\begin{itemize}

\item $\Dmod(\Bun_T)$ acts on $\Dmod_{\frac{1}{2}}(\Bun_{G_1})$ by !-convolution;

\item The functor $\Dmod(\Bun_T)\to \Vect$ is cohomology with compact supports. 

\end{itemize}

\medskip

From \eqref{e:LS as fiber} we obtain:
$$\QCoh(\LS_{\cG_2})\simeq \Vect\underset{\QCoh(\LS_\cT)}\otimes \QCoh(\LS_{\cG_1}),$$
and hence
$$\QCoh(\LS_{\cG_2})\underset{\QCoh(\LS_{\cG_1})}\otimes  \Dmod_{\frac{1}{2}}(\Bun_{G_1})$$
can be rewritten as 
$$\Vect\underset{\QCoh(\LS_\cT)}\otimes \Dmod_{\frac{1}{2}}(\Bun_{G_1}).$$

However, by Fourier-Mukai (i.e., GLC for tori)
$$\QCoh(\LS_\cT) \overset{\on{FM}}\simeq \Dmod(\Bun_T).$$

Hence, we obtain 
\begin{multline*}
\QCoh(\LS_{\cG_2})\underset{\QCoh(\LS_{\cG_1})}\otimes  \Dmod_{\frac{1}{2}}(\Bun_{G_1}) \simeq 
\Vect\underset{\QCoh(\LS_\cT)}\otimes \Dmod_{\frac{1}{2}}(\Bun_{G_1}) \overset{\on{FM}}\simeq \\
\simeq \Vect\underset{\Dmod(\Bun_T)}\otimes \Dmod_{\frac{1}{2}}(\Bun_{G_1})
\overset{\text{\eqref{e:Bun as quot bis}}}\simeq  \Dmod_{\frac{1}{2}}(\Bun_{G_2}).
\end{multline*}

This is the desired equivalence \eqref{e:change group 2}. 

\sssec{Proof for type A}

First, replacing $G_1$ by its derived group, we obtain that it is enough to consider the case when $G_1$
is semi-simple, which we will from now on assume.

\medskip

Denote by $T$ the cokernel of $\phi$, and denote by $\psi$ the projection
$$\Bun_{G_2}\to \Bun_T.$$

Consider $\Dmod(\Bun_T)$ as a (symmetric) monoidal category with respect to the pointwise $\sotimes$ tensor product,
and let it act on $\Dmod_{\frac{1}{2}}(\Bun_{G_2})$ via $\psi^!(-)\sotimes (-)$. 

\medskip

Denote
\begin{equation} \label{e:neutral subcateg}
\Dmod_{\frac{1}{2}}(\Bun_{G_1})':=\Vect\underset{\Dmod(\Bun_T)}\otimes \Dmod_{\frac{1}{2}}(\Bun_{G_2}),
\end{equation} 
where $\Dmod(\Bun_T)\to \Vect$ is the functor of !-fiber at the unit point.

\medskip

The functor $\phi^!$ naturally factors as 
$$\Dmod_{\frac{1}{2}}(\Bun_{G_2})\to \Dmod_{\frac{1}{2}}(\Bun_{G_1})' \overset{(\phi^!)'}\to \Dmod_{\frac{1}{2}}(\Bun_{G_1}),$$
and it is easy to see that the functor $(\phi^!)'$ is fully faithful. In fact, its essential image is a \emph{direct summand} in 
$\Dmod_{\frac{1}{2}}(\Bun_{G_1})$, described as follows.


\sssec{}

Note that the group $Z_{G_1}$ (which is finite, due to the assumption that $G_1$ is semi-simple) 
acts by automorphisms of the identity functor of $\Dmod_{\frac{1}{2}}(\Bun_{G_1})$. Hence,
the category $\Dmod_{\frac{1}{2}}(\Bun_{G_1})$ splits as a direct sum according to characters of $Z_{G_1}$: 
$$\Dmod_{\frac{1}{2}}(\Bun_{G_1})=\underset{\alpha\in (Z_{G_1})^\vee}\oplus\, \Dmod_{\frac{1}{2}}(\Bun_{G_1})_\alpha.$$

\medskip

We can identify $G_1$ with the derived group of $G_2$; moreover
$$G_2\simeq G_1\overset{\Gamma}\times Z^0_{G_2},$$
where $\Gamma$ is a finite group equipped with embeddings
\begin{equation} \label{e:grp Gamma}
Z_{G_1}\hookleftarrow \Gamma \hookrightarrow  Z^0_{G_2}.
\end{equation} 

\medskip

We claim:

\begin{lem} \label{l:comps that arise}
$$\Dmod_{\frac{1}{2}}(\Bun_{G_1})'= \underset{\alpha}\oplus\, \Dmod_{\frac{1}{2}}(\Bun_{G_1})_\alpha,$$
where $\alpha$ runs over the subset consisting of those characters that vanish on $\Gamma$. 
\end{lem}

\begin{proof}

The subcategory $\Dmod_{\frac{1}{2}}(\Bun_{G_1})'$ is the full subcategory in 
$\Dmod_{\frac{1}{2}}(\Bun_{G_1})$, generated by the essential image of $\phi^!$.  

\medskip

Denote
$$\Bun'_{G_1}:=\Bun_{G_1}\overset{\on{pt}/\Gamma}\times (\on{pt}/Z^0_{G_2}),$$ so that we
can factor $\phi$ as
$$\Bun_{G_1}\overset{\phi'}\to \Bun_{G'_1}\to \Bun_{G_2},$$
where the second arrow is a closed embedding. Hence, it suffices to show that the essential image of $(\phi')^!$
generates the subcategory described in the statement of the lemma. 

\medskip

Consider the projection
$$\Bun'_{G_1} \overset{\phi''}\to \Bun_{G_1}/(\on{pt}/\Gamma).$$
It is a gerbe with typical fiber $\on{pt}/Z^0_{G_2}$. Hence, the essential image of $(\phi'')^!$ generates 
$\Dmod_{\frac{1}{2}}(\Bun'_{G_1})$. 

\medskip

Hence, the subcategory generated by the essential image of $(\phi')^!$ is the same as the one generated
by the essential image of $(\phi''\circ \phi')^!$. However, the latter is exactly the category singled out
by the condition on the characters, since $\phi''\circ \phi'$ is the projection
$$\Bun_{G_1}\to \Bun_{G_1}/(\on{pt}/\Gamma).$$

\end{proof}

\medskip


\sssec{}

Consider the map
$$\phi^\vee:\LS_{\cG_2}\to \LS_{\cG_1}.$$

Note that we have a commutative diagram
$$
\CD
\pi_0(\LS_{\cG_2}) @>>> \pi_0(\LS_{\cG_1}) \\
@VVV @VVV \\
(Z_{G_2/Z^0_{G_2}})^\vee @>>> (Z_{G_1})^\vee,
\endCD
$$
see \secref{sss:alg fund group}, and note that
$$G_2/Z^0_{G_2}\simeq G_1/\Gamma \,\Rightarrow\, Z_{G_2/Z^0_{G_2}}\simeq Z_{G_1}/\Gamma,$$
where $\Gamma$ is as in \eqref{e:grp Gamma}.

\medskip

Let 
\begin{equation} \label{e:im LS}
\LS'_{\cG_1}\subset \LS_{\cG_1}
\end{equation}
be the union of connected components that lie in the essential image of $\phi^\vee$. The following assertion is a particular case 
of \thmref{t:Hecke Z 0} below: 

\begin{prop} \label{p:neutral}
The full subcategory
$$\Dmod_{\frac{1}{2}}(\Bun_{G_1})' \subset \Dmod_{\frac{1}{2}}(\Bun_{G_1})$$
equals
$$\QCoh(\LS'_{\cG_1})\underset{\QCoh(\LS_{\cG_1})}\otimes \Dmod_{\frac{1}{2}}(\Bun_{G_1}).$$
\end{prop}

\sssec{}

Let us assume this proposition and proceed with the proof of Case A of \thmref{t:change group 2}. Note that 
the functor $\phi_!$ factors as 
$$\Dmod_{\frac{1}{2}}(\Bun_{G_1})\to  \Dmod_{\frac{1}{2}}(\Bun_{G_1})' \overset{(\phi_!)'}\to \Dmod_{\frac{1}{2}}(\Bun_{G_2}),$$
where the first arrow is the corresponding orthogonal projection. 

\medskip

We obtain that the functor 
\eqref{e:change group 2} is an equivalence if and only if the functor
\begin{equation} \label{e:change group 2'}
\QCoh(\LS_{\cG_2})\underset{\QCoh(\LS'_{\cG_1})}\otimes  \Dmod_{\frac{1}{2}}(\Bun_{G_1})'  \to \Dmod_{\frac{1}{2}}(\Bun_{G_2}),
\end{equation}
induced by $(\phi_!)'$, is an equivalence. 

\sssec{}

Consider $\QCoh(\LS_\cT)$ as a monoidal category with respect to \emph{convolution}. As such,
it acts on $\QCoh(\LS_{\cG_2})$ by convolution, corresponding to the action of $\LS_\cT$ on $\LS_{\cG_2}$
given by the map 
$$\psi^\vee:\cT\to Z_{\cG_2}.$$

\medskip

Note that using the Fourier-Mukai equivalence (i.e., GLC for tori)
$$\QCoh(\LS_\cT) \overset{\on{FM}}\simeq \Dmod(\Bun_T),$$
we can rewrite $\Dmod_{\frac{1}{2}}(\Bun_{G_1})'$ as
$$\Vect\underset{\QCoh(\LS_\cT)}\otimes \Dmod_{\frac{1}{2}}(\Bun_{G_2}),$$
where the functor $\QCoh(\LS_\cT)\to \Vect$ is $\Gamma(\LS_\cT,-)$.

\sssec{}

Note the action of $\QCoh(\LS_{\cG_2})$ on $\Dmod_{\frac{1}{2}}(\Bun_{G_2})$ is compatible with 
the actions of $\QCoh(\LS_\cT)$ on both: indeed, this is a particular case of the compatibility in \secref{sss:isog compat}
for the homomorphism $G_2\to G_2\times T$. 

\medskip

Thus, \eqref{e:change group 2'} can be rewritten as the special
case (for $\bC:=\Dmod_{\frac{1}{2}}(\Bun_{G_2})$) of the functor 
\begin{equation} \label{e:change group 2''}
\QCoh(\LS_{\cG_2})\underset{\QCoh(\LS'_{\cG_1})}\otimes  
\left(\Vect\underset{\QCoh(\LS_\cT)}\otimes \bC\right) \to \bC,
\end{equation}
defined for a DG category $\bC$, equipped with an action of $\QCoh(\LS_{\cG_2})$ and a compatible action of
$\QCoh(\LS_\cT)$.

\sssec{}

We claim that \eqref{e:change group 2''} is an equivalence for any such $\bC$. Here is the general paradigm: 

\medskip

Let
$\CY$ be an algebraic stack with an affine diagonal, and let 
$$\wt\CY\to \CY$$
be a torsor with respect to a group-stack $\CT$, also with an affine diagonal. 

\medskip

Assume that both $\CY$ and $\CT$ are quasi-compact, locally almost of finite type and eventually coconnective
(so that \cite[Theorem 2.2.6]{Ga3}\footnote{Some gaps in the proof of this theorem have been found after its publication,
which have subsequently been fixed. However, for our purposes here, we are only using the case of algebraic stacks
that can be written as global quotients, in which case the result follows from the (easier) Theorem 2.2.4 in {\it loc.cit.}}
is applicable). 

\medskip

Then the 2-category of DG categories tensored over $\QCoh(\wt\CY)$ and equipped with a compatible
action of $\QCoh(\CT)$ is equivalent to the 2-category of DG categories tensored over $\QCoh(\CY)$,
with the mutually inverse equivalences being
$$\bC \mapsto \Vect\underset{\QCoh(\CT)}\otimes \bC$$
and
$$\bD\mapsto \QCoh(\wt\CY) \underset{\QCoh(\CY)}\otimes \bD.$$

\sssec{}

We apply this to 
$$\CY:=\LS'_{\cG_1},\,\, \wt\CY:=\LS_{\cG_2},\,\, \CT:=\LS_{\cT}.$$

\qed[\thmref{t:change group 2}]

\section{Low genus cases} \label{s:low genus}

The device that we use to prove the GLC breaks down when 
$X$ has genus $0$ or $1$.
In this section, we treat these cases separately. 
We highlight the key role played by the main 
results of \cite{GLC3} and \cite{GLC4} in this material. 

\ssec{What do we need to prove?}

According to \secref{sss:red to irred}, in order to prove GLC, we 
need to prove that the map \eqref{e:AG irred}
is an isomorphism.

\medskip

We will show that this is automatic when $X$ has low genus. 

\ssec{The case of \texorpdfstring{$g=0$}{g0}}

Note that for a curve of genus $0$, we have $\LS^{\on{irred}}=\emptyset$, so that \eqref{e:AG irred} holds trivially.

\ssec{The case \texorpdfstring{$g=1$}{g1}}

\sssec{}

According to \corref{c:sc}, we can assume that $G$ is almost simple and simply-connected.
We will separate two cases:

\medskip

\noindent(a) $G=SL_n$;

\medskip

\noindent(b) $G\neq SL_n$. 

\sssec{}

In case (a), the dual group $\cG$ is isomorphic to $PGL_n$. In this case, 
\cite[Conjecture 4.5.7]{GLC4} is known (in fact, it is a trivial particular case of
\cite{BKS})\footnote{Fix an irreducible $PGL_n$ local system $\sigma$, and choose its generic lift to an $SL_n$-local
system; denote the underlying vector bundle by $\CE$. Then the space of generic oper structures on 
$\sigma$ is isomorphic to the space of generically defined line subbundles in $\CE$, and this
space is known to be homologically contractible by \cite{Ga5}.}.

\medskip

This implies GLC by \cite[Corollary 4.5.5]{GLC4}.

\sssec{} \label{sss:g 2}

Note that this proof covers the case of $G=SL_n$ for \emph{any} genus. 

\sssec{}

We now consider case (b). 

\begin{prop}[\cite{KS}, \cite{BFM}] \label{p:genus 1}
Let $\sG$ be an adjoint group different from $PGL_n$. Then for a curve of genus $1$, we have
$\LS_\sG^{\on{irred}}=\emptyset$.
\end{prop} 

From the proposition, we obtain that \eqref{e:AG irred} holds trivially in this case.  

\sssec{Proof of \propref{p:genus 1}}

Appealing to Riemann-Hilbert, it is enough to show that a Riemann surface $X$ of genus $1$ does not
admit irreducible (Betti) $\sG_\BC$-local systems. 

\medskip

A $\sG_\BC$-local system $\sigma$ on $X$ is given by a pair of commuting elements
$$g_1,g_2\in \sG_\BC.$$

Consider the subgroup 
$$Z_\sG(g_1,g_2)\simeq \on{Aut}_\sigma.$$

A standard argument shows that if $\sigma$ is irreducible then (in any genus) 
the Lie algebra of $\on{Aut}_\sigma$ is zero. Hence, $Z_\sG(g_1,g_2)$ is finite. 

\medskip

Since $g_1$ and $g_2$ commute, the subgroup that they generate is contained 
in $Z_\sG(g_1,g_2)$, and hence is itself finite. Hence, the pair $(g_1,g_2)$ 
is contained in a compact form $K$ of $\sG_\BC$.

\medskip

However, now \cite[Proposition 4.1.1]{BFM}\footnote{A related result is established also in \cite{KS}.}  
implies that $K\simeq PSU_n$
for some $n$; hence $\sG\simeq PGL_n$.

\qed[\propref{p:genus 1}]

\section{Calculation of the fundamental group} \label{s:pi 1}

In this section we let $G$ be a semi-simple group. 

\medskip

One of the key parts of the argument in the proof of GLC is that\footnote{Under the assumption that $g\geq 2$ (and if $g=2$, the Dynkin diagram
of $G$ has no $A_1$ factors).} the fundamental group of the stack $\LS_\cG^{\on{irred}}$ is small
(outside a few exceptional cases).

\medskip

For example, if $G$ is adjoint (in which case $\cG$ is simply-connected), the stack $\LS_\cG^{\on{irred}}$ is also simply-connected. The reader may
choose to focus on this case on the first pass.

\medskip

For a general $G$, we will show that the fundamental group of $\LS_\cG^{\on{irred}}$ is controlled by the finite group $Z_G$. 

\medskip 

We remark that the arguments in this section are of de Rham nature.
It would be nice to also have a direct topological proof of
\thmref{t:pi 1 LS} in its Betti incarnation.

\ssec{The fundamental groupoid of \texorpdfstring{$\Bun_{\cG}$}{BunGc}}

\sssec{}

Let $\CS$ be a connective spectrum. We can regard it as a constant prestack, and we let
$\CS_{\on{et}}$ be its \'etale sheafification. 

\medskip

For example, if $\CS=B(\Gamma)$, where $\Gamma$ is a finite abelian group, then $B(\Gamma)_{\on{et}}$
is the \'etale stack $\on{pt}/\Gamma$.



\sssec{}

Consider the tautological map 
$$\cG\to \on{pt}/\pi_1(\cG),$$
where $\pi_1(\cG)$ denotes the \'etale fundamental group of $\cG$.

\medskip

The above map induces a map
$$\on{pt}/\cG\to B^2(\pi_1(\cG))_{\on{et}},$$
and hence to a map 
\begin{equation} \label{e:to gerbes}
\Bun_\cG=\bMaps(X,\on{pt}/\cG)\to \bMaps(X,B^2(\pi_1(\cG))_{\on{et}})=:\on{Ge}_{\pi_1(\cG)}(X),
\end{equation}
where $\bMaps(-,-)$ denoted the prestack of maps. 

\begin{rem}
The map \eqref{e:to gerbes} means that to a $\cG$-bundle we can canonically associate an \'etale $\pi_1(\cG)$-gerbe. Namely,
this is the gerbe of \'etale-local lifts of our bundle to the simply-connected cover of $\cG$.
\end{rem} 

\sssec{}

Note that we can think of $\on{Ge}_{\pi_1(\cG)}(X)$ also as
$$B^2(\on{C}^\cdot(X,\pi_1(\cG)))_{\on{et}},$$
where $\on{C}^\cdot(X,\pi_1(\cG))$ is the spectrum of \'etale cochains on $X$ with coefficients in $\pi_1(\cG)$.

\medskip

Accordingly, the (2)-stack $\on{Ge}_{\pi_1(\cG)}(X)$ splits into connected components indexed by $H^2(X,\pi_1(\cG))$. 
The neutral connected component is canonically isomorphic to
$$B(\Bun_{\pi_1(\cG)})_{\on{et}}.$$

\sssec{}

We will prove:

\begin{prop} \label{p:pi 1 Bun}
The map \eqref{e:to gerbes} defines an isomorphism of $\tau_{\leq 1}$ truncations
of \'etale homotopy types.
\end{prop}

The concrete meaning of this proposition is that the map \eqref{e:to gerbes} defines a bijection on the sets
of connected components, and on each connected component an isomorphism of 
\'etale fundamental groups.

\medskip

The proof will be given in \secref{ss:calc pi 1 Bun}.

%
%
%
%
%

\ssec{Line bundles on \texorpdfstring{$\Bun_\cG$}{lnbndle}}

We will now use the map \eqref{e:to gerbes} to construct line bundles on $\Bun_\cG$ starting from $Z_G$-torsors. 
This will be part of a more general construction, which will be extensively used in \secref{s:gerbes}. 

\sssec{} 

Note that we have a canonical identification
\begin{equation} \label{e:pi 1 Z duality}
\pi_1(\cG)\simeq (Z_G)^\vee(1),
\end{equation} 
where $(-)^\vee$ denotes Cartier duality and $(1)$ denotes the Tate twist. 

\sssec{}

Combining \eqref{e:pi 1 Z duality} with Verdier duality 
$$\on{C}^\cdot(X,Z_G)^\vee\simeq B^2\left(\on{C}^\cdot(X,(Z_G)^\vee(1))\right),$$
we obtain an identification
\begin{equation} \label{e:Verdier duality}
\on{C}^\cdot(X,Z_G)^\vee \simeq B^2(\on{C}^\cdot(X,\pi_1(\cG))),
\end{equation} 
and in particular a bilinear pairing
\begin{equation} \label{e:Verdier pairing initial}
B^2(\on{C}^\cdot(X,Z_G))\times B^2(\on{C}^\cdot(X,\pi_1(\cG)))\to B^2(\mu_\infty).
\end{equation}

\sssec{}

After \'etale sheafification, from \eqref{e:Verdier pairing initial} we obtain a bilinear pairing
\begin{equation} \label{e:gerbe pairing}
\on{Ge}_{Z_G}(X)\times \on{Ge}_{\pi_1(\cG)}(X)\to \on{Ge}_{\mu_\infty}(\on{pt}),
\end{equation}
where
$$\on{Ge}_{\mu_\infty}(\on{pt}):=B^2(\mu_\infty)_{\on{et}}.$$

\sssec{} \label{sss:loop gerbes}

Looping \eqref{e:gerbe pairing} along the first factor, we obtain a pairing\footnote{We will return to the untruncated pairing \eqref{e:gerbe pairing}
in \secref{s:gerbes}, where it will play a fundamental role.}
\begin{equation} \label{e:truncated pairing}
\Bun_{Z_G} \times \on{Ge}_{\pi_1(\cG)}(X) \to B(\mu_\infty)_{\on{et}}\to \on{pt}/\BG_m. 
\end{equation}

In particular, we obtain that a point 
$$\CP_{Z_G}\in \Bun_{Z_G}$$
gives rise to a canonically defined $\mu_\infty$-torsor, to be denoted
$$\CL_{\CP_{Z_G}},$$  
on $\on{Ge}_{\pi_1(\cG)}(X)$.

\medskip

The fact that \eqref{e:Verdier pairing initial} is a perfect pairing implies:

\begin{lem} \label{l:every torsor}
Every $\mu_\infty$-torsor on every connected component of $\on{Ge}_{\pi_1(\cG)}(X)$ is the restriction of $\CL_{\CP_{Z_G}}$
for some $\CP_{Z_G}$. 
\end{lem} 

\sssec{}  \label{sss:every torsor}

We will denote by the same symbol $\CL_{\CP_{Z_G}}$ the pullback of the above $\mu_\infty$-torsor along the map \eqref{e:to gerbes}.
By a slight abuse of notation, we will continue to use the same symbol $\CL_{\CP_{Z_G}}$ the corresponding \'etale local
system of $k$-vector spaces on $\Bun_\cG$.

\medskip

By \propref{p:pi 1 Bun}, every \'etale local system of $k$-vector spaces on a given component of $\Bun_\cG$
splits as a direct sum of 1-dimensional ones, and by \lemref{l:every torsor}, each of the latter is isomorphic to the restriction of 
$\CL_{\CP_{Z_G}}$ for some $\CP_{Z_G}$. 

\sssec{} \label{sss:L P Z G}

Since $\CL_{\CP_{Z_G}}$, viewed as an \'etale local system of $k$-vector spaces comes from a $\mu_\infty$-torsor, 
we can canonically associate to it a de Rham local system, which we will still denote by the
same character $\CL_{\CP_{Z_G}}$. 

\medskip

We will also use the same symbol $\CL_{\CP_{Z_G}}$ to denote the corresponding line bundle (viewed either
as the line bundle underlying the corresponding de Rham local system, or equivalently, as a $\CO^\times$-torsor
induced by the map $\mu_\infty\to \CO^\times$). 

\medskip

By a further abuse of notation, we will use the same symbol $\CL_{\CP_{Z_G}}$ to denote its pullback
(in all three incarnations: \'etale, de Rham, coherent) along the map
$$\LS_\cG\to \Bun_\cG.$$

\ssec{The fundamental group of \texorpdfstring{$\LS_\cG^{\on{irred}}$}{fnd}}

\sssec{}

Consider the map
\begin{equation} \label{e:LS irred to Bun}
\LS_\cG^{\on{irred}}\to \LS_\cG\to \Bun_\cG.
\end{equation}

The main result of this subsection is the following assertion:

\begin{thm} \label{t:pi 1 LS}
Assume that $g\geq 2$, and if $g=2$, then its root system does not have $A_1$ factors. 
Then the map \eqref{e:LS irred to Bun} induces an isomorphism on the 
$\tau_{\leq 1}$ truncations of \'etale homotopy types.
\end{thm} 

Before we proceed to the proof, we record the following corollary, obtained by combining
\thmref{t:pi 1 LS} and \propref{p:pi 1 Bun} (see \secref{sss:every torsor}):

\begin{cor} \label{c:pi 1 LS}
For every connected component of $\LS_\cG^{\on{irred}}$, every \'etale
local system of $k$-vector spaces on it splits as a direct sum of 1-dimensional ones. 
Each of the latter is isomorphic to the restriction of $\CL_{\CP_{Z_G}}$ for
some $\CP_{Z_G}\in \Bun_{Z_G}$.
\end{cor}

The rest of this subsection is devoted to the proof of \thmref{t:pi 1 LS}. 

\sssec{}

Let 
$$\Bun_\cG^{\on{stbl}}\subset \Bun_\cG$$
be the stable locus. We will prove:

\begin{prop} \label{p:compl 1}
Under the assumptions on $G$ and $g$ specified in \thmref{t:pi 1 LS}, the complement
of $\Bun_\cG^{\on{stbl}}$ in $\Bun_\cG$ has codimension $\geq 2$.
\end{prop}

\begin{rem}

Statements of this type are classical in the 
literature of $\Bun_G$; they begin with \cite[Sect. 9]{NR}.
The literature we found concerned coarse moduli spaces instead
of moduli stacks, so we include the argument for
\propref{p:compl 1} in \secref{ss:proof of compl 1}.
There are no significant differences between our argument
and those in the existing literature.

\end{rem}

%
%

\sssec{}

Denote
$$\LS_\cG^{\on{stbl}}:=\LS_\cG\underset{\Bun_\cG}\times \Bun_\cG^{\on{stbl}}.$$

The following is well-known: 

\begin{prop} \label{p:conn on stable bundles}
The map 
$$\LS_\cG^{\on{stbl}}\to \Bun_\cG^{\on{stbl}}$$
is smooth and surjective. The fibers 
of this map are affine spaces.
\end{prop}

For the sake of completeness, we will supply a proof in \secref{ss:conn on stable bundles}. 
%

\sssec{Proof of \thmref{t:pi 1 LS}}

First, we claim that there is an inclusion 
$$\LS_\cG^{\on{stbl}}\subset \LS_\cG^{\on{irred}}$$
with both spaces being smooth Deligne-Mumford stacks, and both are open inside $\LS_\cG$.

\medskip

It is enough to establish the inclusion at the level of $k$-points. 
Suppose $\sigma \in \LS_\cG^{\on{stbl}}$; we need to show it
does not admit a reduction $\sigma_\cP$ to any proper parabolic $\cP \subset \cG$. 
Suppose we had such a reduction. Then, by the definition of stability (see Appendix \ref{s:stable}), 
we would have $\langle 2\rhoch_\cP,\deg(\sigma_P)\rangle<0$. However, the above integer is the degree
of the line bundle induced from $\sigma_\cP$ using $2\rhoch_\cP$, viewed as a homomorphism
$\cP \to \cM \to \BG_m$. Since the line bundle is endowed with a connection, its degree must be zero,
which is a contradiction. 

\medskip

Next, we claim that $\LS_\cG^{\on{stbl}}$ is dense in $\LS_\cG^{\on{irred}}$. This is equivalent to saying that 
$\LS_\cG^{\on{stbl}}$ has a non-empty intersection with every irreducible component 
of $\LS_\cG^{\on{irred}}$. Since $\LS_\cG^{\on{irred}}$ is smooth, its irreducible components
are the same as connected components. 

\medskip

According to \corref{c:compl 2} below, the embedding  
$$\LS^{\on{irred}}_\cG\hookrightarrow \LS_\cG$$
induces a bijection on the sets of connected components. 

\medskip

Hence, we obtain that it is sufficient to show that the embedding
$$\LS_\cG^{\on{stbl}}\hookrightarrow \LS_\cG$$
induces a bijection on the sets of connected components. 

\medskip

We have a commutative square
$$
\CD
\LS^{\on{stbl}}_\cG @>>>  \LS_\cG \\
@VVV @VVV \\
\Bun^{\on{stbl}}_\cG @>>>  \Bun_\cG.
\endCD
$$

In it, the lower horizontal arrow and the left vertical arrow induce 
bijections on the sets of connected components, by Propositions
\ref{p:compl 1} and \ref{p:conn on stable bundles}, respectively.

\medskip

Hence, the desired assertion follows from the fact that the map
$$\LS_\cG \to \Bun_\cG$$
induces a bijection on $\pi_0$ (see \cite[Proposition 2.11.4]{BD}). 

\medskip

Thus, we can work at one connected (=irreducible) component at a time (denote it by the subscript $\alpha$). 
We will show that the maps
\begin{equation} \label{e:pi 1 maps}
\pi_1(\LS^{\on{stbl}}_{\cG,\alpha}) \to \pi_1(\LS^{\on{irred}}_{\cG,\alpha})\to \pi_1(\Bun_{\cG,\alpha})
\end{equation} 
(for some choice of a base-point on $\LS^{\on{stbl}}_{\cG,\alpha}$) are isomorphisms. 

\medskip

The first map in \eqref{e:pi 1 maps} is surjective, thanks to the density. Hence, it is sufficient to show
that the composite map is an isomorphism. 

\medskip

However, the above composite map is
$$\pi_1(\LS^{\on{stbl}}_{\cG,\alpha}) \to \pi_1(\Bun^{\on{stbl}}_{\cG,\alpha})\to \pi_1(\Bun_{\cG,\alpha}).$$

In the latter composition, the first map is an isomorphism thanks to \propref{p:conn on stable bundles}. 
The corresponding fact for the second map follows
from \propref{p:compl 1}.

\qed[\thmref{t:pi 1 LS}]
  
\section{The core of the proof} \label{s:proof}

Throughout this section, we will assume that $G$ is semi-simple. 

\medskip

After some preparations, in this section we will give a proof of GLC for curves of genus $\geq 2$. 

\medskip

We note that the proof is particularly simple when $G$ is adjoint, so that $\cG$ is simply-connected
(in this case, one needs neither \secref{ss:act of cen} nor \thmref{t:end Poinc vac b}). The reader 
may choose to focus on this case on the first pass. 

\ssec{Action of the center} \label{ss:act of cen}

\sssec{}

Note that the (abelian) group-stack $\Bun_{Z_G}$ acts on $\Bun_G$, and this action lifts to
an action of $\Bun_{Z_G}$ on $\Dmod_{\frac{1}{2}}(\Bun_G)$. 

\medskip

For $\CP_{Z_G}\in \Bun_{Z_G}$, we will denote the corresponding automorphism of $\Dmod_{\frac{1}{2}}(\Bun_G)$ by 
\begin{equation} \label{e:central shift}
\CP_{Z_G}\cdot -.
\end{equation} 

\sssec{}

Note that the group $Z_G$ acts on the unit point $\CP^0_{Z_G}\in \Bun_{Z_G}$. Since 
$$\CP^0_{Z_G} \cdot (-)$$ 
is the identity functor, we obtain that $Z_G$ acts by automorphisms of the identity endofunctor
of $\Dmod_{\frac{1}{2}}(\Bun_G)$.

\medskip

Let
$$\Dmod_{\frac{1}{2}}(\Bun_G) \simeq \underset{\alpha}\oplus\, \Dmod_{\frac{1}{2}}(\Bun_G)_\alpha, \quad \alpha\in (Z_G)^\vee$$
denote the corresponding decomposition. Denote by $\sP_\alpha$ the idempotent on $\Dmod_{\frac{1}{2}}(\Bun_G)$, corresponding 
to $\Dmod_{\frac{1}{2}}(\Bun_G)_\alpha$.

\sssec{} \label{sss:alg fund group}

Let $\pi_{1,\on{alg}}(\cG)$ denote the \emph{algebraic fundamental group} of $\cG$, i.e., the quotient of the coweight lattice by the coroot lattice. 
We have
$$\pi_{1,\on{alg}}(\cG)\simeq \pi_0(\Bun_\cG)\hookleftarrow \pi_0(\LS_\cG),$$
and the latter inclusion is an isomorphism if $g\geq 2$. 

\medskip

Note that we have 
$$\pi_{1,\on{alg}}(\cG)\simeq \pi_1(\cG)(-1),$$
so that we have a canonical identification
$$(Z_G)^\vee\simeq \pi_{1,\on{alg}}(\cG).$$

\sssec{}

For a given $\alpha\in (Z_G)^\vee$, let $\LS_{\cG,\alpha}$ denote the corresponding connected component of
$\LS_\cG$. Consider the corresponding idempotent
$$\CO_{\LS_{\cG,\alpha}}\in \QCoh(\LS_\cG).$$

We will prove (see \secref{sss:proof central Hecke}):

\begin{thm} \label{t:Hecke Z 0}
For $\alpha\in (Z_G)^\vee$ as above, the idempotent 
$$\CO_{\LS_{\cG,\alpha}}\otimes (-):\Dmod_{\frac{1}{2}}(\Bun_G)\to \Dmod_{\frac{1}{2}}(\Bun_G),$$
where $\otimes$ denotes the spectral action of $\QCoh(\LS_\cG)$ on $\Dmod_{\frac{1}{2}}(\Bun_G)$,
identifies canonically with $\sP_{\alpha}$.
\end{thm}

\sssec{}

We will also prove (see \secref{sss:proof central Hecke}):

\begin{thm} \label{t:Hecke Z 1}
For $\CP_{Z_G}\in \Bun_{Z_G}$ and the corresponding $\CL_{\CP_{Z_G}}\in \QCoh(\LS_\cG)$,
the functor 
$$\CL^{\otimes -1}_{\CP_{Z_G}}\otimes(-):\Dmod_{\frac{1}{2}}(\Bun_G)\to \Dmod_{\frac{1}{2}}(\Bun_G),$$
where $\otimes$ denotes the spectral action of $\QCoh(\LS_\cG)$ on $\Dmod_{\frac{1}{2}}(\Bun_G)$,
identifies canonically with \eqref{e:central shift}.
\end{thm} 
%
%

\ssec{Endomorphisms of the vacuum Poincar\'e object}

In this subsection, we will assume that $g\geq 2$. 

\sssec{}

Recall the object
$$\on{Poinc}^{\on{Vac,glob}}_{G,!}\in \Dmod_{\frac{1}{2}}(\Bun_G),$$
see \cite[Sect. 9.6.4]{GLC2}. 
For $\alpha\in (Z_G)^\vee$, let $\on{Poinc}^{\on{Vac}}_{G,!,\alpha}$ denote the corresponding direct summand. 

\sssec{}

We will prove (see \secref{ss:proof of end Poinc}):

\begin{thm} \label{t:end Poinc vac a}
For every $\alpha$, the map
$$k\to H^0(\CEnd(\on{Poinc}^{\on{Vac}}_{G,!,\alpha}))$$ 
is an isomorphism.
\end{thm}

In fact, we will prove a more precise result, but only \thmref{t:end Poinc vac a} will be needed for the proof
of GLC:

\begin{thm} \label{t:end Poinc vac a'}
$H^i(\CEnd(\on{Poinc}^{\on{Vac,glob}}_{G,!}))=0$ for $i\neq 0$.  
\end{thm}

As a corollary of \thmref{t:end Poinc vac a}, we obtain:

\begin{cor} \label{c:end Poinc vac a} 
$\dim\left(H^0(\CEnd(\on{Poinc}^{\on{Vac,glob}}_{G,!}))\right)=|Z_G|$.
\end{cor}

\sssec{}

We will also prove (see \secref{ss:proof of end Poinc transl}):

\begin{thm} \label{t:end Poinc vac b}
For a non-trivial $\CP_{Z_G}\in \Dmod_{\frac{1}{2}}(\Bun_G)$,
$$H^0\left(\CHom(\on{Poinc}^{\on{Vac,glob}}_{G,!},\CP_{Z_G}\cdot \on{Poinc}^{\on{Vac,glob}}_{G,!})\right)=0.$$
\end{thm} 

As in the case of \thmref{t:end Poinc vac a}, we will actually prove a more precise result
(but only \thmref{t:end Poinc vac b} will be needed for the proof of GLC):

\begin{thm} \label{t:end Poinc vac b'}
For a non-trivial $\CP_{Z_G}\in \Dmod_{\frac{1}{2}}(\Bun_G)$,
$$\CHom(\on{Poinc}^{\on{Vac,glob}}_{G,!},\CP_{Z_G}\cdot \on{Poinc}^{\on{Vac,glob}}_{G,!})=0.$$
\end{thm} 

\ssec{Algebraic geometry of \texorpdfstring{$\LS_\cG$}{LS}}

In this subsection, we continue to assume that $g\geq 2$. 

\sssec{}

First, we recall (see \cite[Proposition 2.11.2]{BD}):

\begin{thm} \label{t:LS CM}
The stack $\LS_\cG$ is a classical locally complete intersection of dimension 
$$\dim(\fg)\cdot 2(g-1).$$
\end{thm} 

\begin{cor}  \label{c:LS CM}
The stack $\LS_\cG$ is Cohen-Macaulay of dimension $\dim(\fg)\cdot 2(g-1)$.
\end{cor} 

\sssec{}

Next we claim:

\begin{prop} \label{p:compl 2}
Excluding the case of $g=2$ with the root system of $G$ containing an $A_1$ factor, 
the complement to $\LS_\cG^{\on{irred}}$ in $\LS_\cG$ has codimension $\geq 2$.
\end{prop}

This proposition will be proved in \secref{ss:proof of compl 2}.

\sssec{}

Note that from \corref{c:LS CM} and \propref{p:compl 2}, we obtain:

\begin{cor} \label{c:compl 2} 
The embedding $\LS^{\on{irred}}_\cG\hookrightarrow \LS_\cG$
induces a bijection between the sets of connected components.
\end{cor} 

\ssec{Structure of \texorpdfstring{$\CA_{G,\on{irred}}$}{strAG}} \label{ss:A G irred}

In this subsection we continue to assume that $g\geq 2$, and we will exclude the case that $g=2$ and the root system of 
$G$ contains a factor of $A_1$ (as in \propref{p:compl 2}). 

\sssec{}

Recall that $\CA_{G,\on{irred}}$ is a vector bundle on $\LS_\cG^{\on{irred}}$ (see \secref{sss:AG nice}). 
The next proposition provides an explicit description of the shape that $\CA_{G,\on{irred}}$ can have. This
description will play a crucial role in the proof of GLC given below. 

\begin{prop} \label{p:A irred}
The restriction of $\CA_{G,\on{irred}}$ to every connected component is isomorphic to a
direct sum of lines bundles, each of which is a restriction of some $\CL_{\CP_{Z_G}}$ 
(see \secref{sss:L P Z G}). 
\end{prop}

\sssec{Proof of \propref{p:A irred}}

According to \secref{ss:AG nice}, the object $\CA_{G,\on{irred}}$ is the vector bundle
underlying a local system $\CF$ with finite monodromy (in particular, 
it has regular singularities).

\medskip

Thanks to the finite monodromy property, we can think of $\CF$ as an \'etale local system of $k$-vector spaces. 
The assertion of the proposition follows now from \corref{c:pi 1 LS}. 

\qed[\propref{p:A irred}]

\ssec{Proof of GLC} \label{ss:proof of GLC}

Let $g$ and $G$ be as in \secref{ss:A G irred}. Note that it is sufficient to prove GLC under these assumptions:

\medskip

Indeed, by \secref{s:low genus}, we can assume that $g\geq 2$. By \corref{c:sc}, we may assume that $G$ is almost simple and 
simply-connected. By \secref{sss:g 2} we can assume that it is not isomorphic to $SL_2$. 

\sssec{Step 0}

Fix a connected component $\LS_{\cG,\alpha}$ of $\LS_\cG$. Denote
$$\LS_{\cG,\alpha}^{\on{irred}}:=\LS_{\cG,\alpha}\cap \LS_\cG^{\on{irred}}.$$

\medskip

Using \propref{p:A irred}, we can write
\begin{equation} \label{e:decomp A}
\CA_{G,\on{irred}}|_{\LS_{\cG,\alpha}^{\on{irred}}}\simeq \underset{\CP_{Z_G}\in \Bun_{Z_G}}\oplus\, 
\CL^{\oplus\, n_{\CP_{Z_G},\alpha}}_{\CP_{Z_G}}|_{\LS_{\cG,\alpha}^{\on{irred}}}
\end{equation} 
for some integers $n_{\CP_{Z_G},\alpha}$. 

\medskip

It is sufficient to show that for every $\alpha$
$$
n_{\CP_{Z_G},\alpha}=
\begin{cases}
&1 \text{ if } \CP_{Z_G} \text{ is trivial};\\
&0 \text{ if } \CP_{Z_G} \text{ is non-trivial}.
\end{cases}
$$

Indeed, since each $\CA_{G,\on{irred}}|_{\LS_{\cG,\alpha}^{\on{irred}}}$ is a unital associative algebra in $\QCoh(\LS_{\cG,\alpha}^{\on{irred}})$,
the latter will automatically imply (see \cite[Lemma 17.3.7]{GLC3}) that the unit map
\begin{equation} \label{e:unit alpha irred}
\CO_{\LS_{\cG,\alpha}^{\on{irred}}}\to \CA_{G,\on{irred}}|_{\LS_{\cG,\alpha}^{\on{irred}}}
\end{equation}
is an isomorphism, i.e., \eqref{e:AG irred} is an isomorphism. 

\sssec{Step 1}\label{sss:step 1}

Fix a particular $\CP_{Z_G}$. We are going to prove that the map
$$\CA_G\otimes \CL_{\CP_{Z_G}}\to \jmath_*\circ \jmath^*(\CA_G\otimes \CL_{\CP_{Z_G}}) = 
\jmath_*(\CA_{G,\on{irred}}) \otimes \CL_{\CP_{Z_G}}$$
induces an isomorphism at the level of $H^0(\Gamma(\LS_\cG,-))$.

\medskip

Recall that $\wh\iota^!$ denotes the right adjoint to the embedding
$$\QCoh(\LS_\cG)_{\on{red}}\hookrightarrow \QCoh(\LS_\cG),$$
so that for every $\CF\in \QCoh(\LS_\cG)$ we have a fiber sequence
$$\wh\iota^!(\CF)\to \CF\to \jmath_*\circ \jmath^*(\CF).$$

Thus, it is sufficient to prove that 
$$\wh\iota^!(\CA_G\otimes \CL_{\CP_{Z_G}})$$
is concentrated in cohomological degrees $\geq 2$.

\medskip

We have
$$\wh\iota^!(\CA_G\otimes \CL_{\CP_{Z_G}})\simeq \wh\iota^!(\CA_G)\otimes \CL_{\CP_{Z_G}},$$
where $\CL_{\CP_{Z_G}}$ is a line bundle. So it is enough to show that 
$$\wh\iota^!(\CA_G)$$
is concentrated in cohomological degres $\geq 2$.

\medskip

Recall (see \secref{sss:on Eis}) that the unit map 
$$\CO_{\LS_\cG}\to \CA_G$$ 
induces an isomorphism 
$$\wh\iota^!(\CO_{\LS_\cG})\to \wh\iota^!(\CA_G).$$

Thus, it remains to show that 
$$\wh\iota^!(\CO_{\LS_\cG})$$
is concentrated in cohomological degres $\geq 2$.

\medskip

However, this follows from \propref{p:compl 2} and \corref{c:LS CM}: 

\medskip

If $Y$ is a 
Cohen-Macaulay scheme/stack and $i:Z\subset Y$ is a closed subscheme/substack of 
codimension $\geq k$, then, $\wh{i}^!(\CO_Y)$ lives in cohomological degrees
$\geq k$. 

\begin{rem}
Note that the last step in the above argument shows that the map
$$\CO_{\LS_\cG} \to \jmath_*(\CO_{\LS^{\on{irred}}_\cG})$$
also induces an isomorphism at the level of $H^0(\Gamma(\LS_\cG,-))$.
\end{rem} 

\sssec{Step 2} \label{sss:step 2}

We will now show that $n_{\CP_{Z_G},\alpha}=0$ if $\CP_{Z_G}$ is non-trivial. 

\medskip 

We begin by showing that
\begin{equation}\label{eq:no Hom with L}
\Gamma(\LS_\cG,\CL^{\otimes -1}_{\CP_{Z_G}}\otimes \CA_G) = 0.
\end{equation}

\medskip 

Recall that
$$\CA_G = \BL_G\circ \BL_G^L(\CO_{\LS_\cG}).$$
Therefore, we have
\begin{multline} \label{e:Hom with L}
\Gamma(\LS_\cG,\CL^{\otimes -1}_{\CP_{Z_G}}\otimes \CA_G)=
\Gamma(\LS_\cG,\CL^{\otimes -1}_{\CP_{Z_G}}\otimes \BL_G\circ \BL_G^L(\CO_{\LS_\cG}))\,\,\,\,\overset{\QCoh(\LS_\cG)\text{-linearity of\,}\BL^L_G \text{ and } \BL_G}\simeq \\
\simeq \Gamma(\LS_\cG,\BL_G\circ \BL_G^L(\CL^{\otimes -1}_{\CP_{Z_G}}\otimes \CO_{\LS_\cG}))=
\CHom_{\QCoh(\LS_\cG)}(\CO_{\LS_\cG},\BL_G\circ \BL_G^L(\CL^{\otimes -1}_{\CP_{Z_G}}\otimes \CO_{\LS_\cG}))\simeq \\
\overset{\text{adjunction}}\simeq \CHom_{\Dmod_{\frac{1}{2}}(\Bun_G)}(\BL_G^L(\CO_{\LS_\cG}),\BL_G^L(\CL^{\otimes -1}_{\CP_{Z_G}}\otimes \CO_{\LS_\cG})).
\end{multline} 

Recall that 
$$\BL_G^L(\CO_{\LS_\cG})\simeq \on{Poinc}^{\on{Vac,glob}}_{G,!}.$$

Hence, since the functor $\BL^L_\cG$ is $\QCoh(\LS_\cG)$-linear, and taking into account
\thmref{t:Hecke Z 1}, we can rewrite the right-hand side in \eqref{e:Hom with L} as 
$$\CHom(\on{Poinc}^{\on{Vac,glob}}_{G,!},\CP_{Z_G}\cdot \on{Poinc}^{\on{Vac,glob}}_{G,!}).$$

By \thmref{t:end Poinc vac b}, this expression vanishes, so the same is true of
\eqref{eq:no Hom with L}.

\medskip

Applying Step 1 (with $\CL^{\otimes -1}_{\CP_{Z_G}}$ instead of $\CL_{\CP_{Z_G}}$), we find that
\[
H^0\left(\Gamma(\LS^{\on{irred}}_\cG,\CL^{\otimes -1}_{\CP_{Z_G}}\otimes \CA_{G,\on{irred}})\right)
= H^0\left(\Gamma(\LS_\cG,\CL^{\otimes -1}_{\CP_{Z_G}}\otimes \CA_G)\right) = 0.
\]

\medskip

Hence, we obtain 
$$H^0\left(\Gamma(\LS^{\on{irred}}_\cG,\CL^{\otimes -1}_{\CP_{Z_G}}\otimes \CA_{G,\on{irred}})\right)=0.$$

\medskip

However, by definition, $\CL^{\otimes -1}_{\CP_{Z_G}}\otimes \CA_{G,\on{irred}}$ 
carries $n_{\CP_{Z_G},\alpha}$ direct summands isomorphic to $\CO_{\LS_{\cG,\alpha}^{\on{irred}}}$.
Therefore, we obtain that
\[
0 = \dim H^0\left(\Gamma(\LS^{\on{irred}}_\cG,\CL^{\otimes -1}_{\CP_{Z_G}}\otimes \CA_{G,\on{irred}})\right) \geq 
H^0\left(\Gamma(\LS^{\on{irred}}_\cG,\CO_{\LS_{\cG,\alpha}^{\on{irred}}}^{\oplus n_{\CP_{Z_G},\alpha}})\right)
 \geq n_{\CP_{Z_G},\alpha}
\]

\noindent meaning that $n_{\CP_{Z_G},\alpha} = 0$, as was desired.



\sssec{Step 3} \label{sss:step 3}

Thus, we obtain that the decomposition \eqref{e:decomp A} is in fact of the form
\begin{equation} \label{e:decomp A bis}
\CA_{G,\on{irred}}|_{\LS_{\cG,\alpha}^{\on{irred}}}\simeq \CO^{\oplus\, n_\alpha}_{\LS_{\cG,\alpha}^{\on{irred}}},
\end{equation} 
for some integers $n_\alpha$, and we wish to show that they are all equal to $1$.

\medskip

By Step 1, for every $\alpha$, we have
$$H^0\left(\Gamma(\LS_{\cG,\alpha},\CA_G|_{\LS_{\cG,\alpha}})\right)\simeq
H^0\left(\Gamma(\LS^{\on{irred}}_{\cG,\alpha},\CA_{G,\on{irred}}|_{\LS^{\on{irred}}_{\cG,\alpha}})\right).$$

Hence,
$$\dim H^0\left(\Gamma(\LS_{\cG,\alpha},\CA_G|_{\LS_{\cG,\alpha}})\right) =
\dim H^0\left(\Gamma(\LS^{\on{irred}}_{\cG,\alpha},\CO^{\oplus\, n_\alpha}_{\LS_{\cG,\alpha}^{\on{irred}}})\right)
\geq n_\alpha.$$

Therefore,
$$\dim H^0\left(\Gamma(\LS_\cG,\CA_G)\right) \geq \underset{\alpha}\Sigma\, n_\alpha.$$

Now,
\begin{multline} \label{e:global sect}
\Gamma(\LS_\cG,\CA_G)\simeq \CHom_{\Dmod_{\frac{1}{2}}(\Bun_G)}(\BL_G^L(\CO_{\LS_\cG}),\BL_G^L(\CO_{\LS_\cG})) \simeq \\
\simeq \CHom_{\Dmod_{\frac{1}{2}}(\Bun_G)}(\on{Poinc}^{\on{Vac,glob}}_{G,!},\on{Poinc}^{\on{Vac,glob}}_{G,!}).
\end{multline}

Applying \corref{c:end Poinc vac a}, we obtain 
$$\underset{\alpha}\Sigma\, 1=|Z_G|=|\pi_0(\LS_\cG)|\geq \underset{\alpha}\Sigma\, n_\alpha.$$

Hence, in order to prove the desired equality, it suffices to show that $n_\alpha\neq 0$ for all $\alpha$. I.e., we have to show
that $\CA_{G,\on{irred}}$ does not vanish on any connected component $\LS^{\on{irred}}_{\cG,\alpha}$.


\sssec{Step 4}

%

By \thmref{t:Hecke Z 0}, we have:
$$\Gamma(\LS_{\cG,\alpha},\CA_G|_{\LS_{\cG,\alpha}})\simeq 
\CEnd_{\Dmod_{\frac{1}{2}}(\Bun_G)}(\on{Poinc}^{\on{Vac}}_{G,!,\alpha}).$$
As $\on{Poinc}^{\on{Vac}}_{G,!,\alpha} \neq 0$
by \thmref{t:end Poinc vac a}, we must have:
\[
0 \neq \on{id} \in 
H^0\left(\CEnd(\on{Poinc}^{\on{Vac}}_{G,!,\alpha}) \right)\simeq
H^0\left(\Gamma(\LS_{\cG,\alpha},\CA_G|_{\LS_{\cG,\alpha}}) \right)\simeq
H^0\left(\Gamma(\LS^{\on{irred}}_{\cG,\alpha},\CA_{G,\on{irred}}|_{\LS^{\on{irred}}_{\cG,\alpha}})\right).
\]

\noindent 

\qed[GLC]

\ssec{Additional remarks}

\sssec{}

For $G$ simply-connected, the assertion that all $\CA_{G,\on{irred}}|_{\LS^{\on{irred}}_{\cG,\alpha}}$ are non-zero can be also deduced
from the main result of \cite{Ari}:

\medskip

Let $\sigma$ be a point of $\LS^{\on{irred}}_{\cG,\alpha}$. Recall that according to \cite[Theorem 3.1.5]{GLC4}, the fiber
of $\CA_G$ at $\sigma$ is isomorphic to the homology of the space of generic oper structures on $\sigma$.

\medskip

Thus, it is sufficient to know that the latter space is non-empty. However, this is precisely the 
result\footnote{Note that the result of \cite{Ari} is about $\cg$-opers, which are different from $\cG$-opers, unless $\cG$ is adjoint.
However, as we have seen earlier, it is sufficient to prove GLC in the latter case.}
of \cite{Ari}.

\sssec{}

Note that taking into account \thmref{t:end Poinc vac a'}, and knowing that $\CO_{\LS_\cG}\simeq \CA_G$, from 
\eqref{e:global sect} we obtain:

\begin{cor} \label{c:O on LS}
For $g\geq 2$, for every connected component $\alpha$ of $\LS_\cG$, the map
$$k\to \Gamma(\LS_{\cG,\alpha},\CO_{\LS_{\cG,\alpha}})$$
is an isomorphism.
\end{cor} 

\begin{rem}
One can prove \corref{c:O on LS} directly by a deformation argument using 
\cite[Theorem 4.2]{FT}. 
\end{rem}

\begin{rem}
We remark that Corollary \ref{c:O on LS} is a special feature of the de Rham setting; 
there are many more global functions on the Betti moduli stack. 
\end{rem}

\sssec{}

Similarly, from \thmref{t:end Poinc vac b'}, we obtain:

\begin{cor}
For $g\geq 2$, for a non-trivial $\CP_{Z_G}$, we have $\Gamma(\LS_\cG,\CL_{\CP_{Z_G}})=0$. 
\end{cor} 

\section{The vacuum Poincar\'e object} \label{s:Poinc}

In this section we will prove Theorems \ref{t:end Poinc vac a} and \ref{t:end Poinc vac b}, along with their strengthenings,
given by Theorems \ref{t:end Poinc vac a'} and \ref{t:end Poinc vac b'}, respectively. 

\ssec{How to calculate those endomorphisms?} 

\sssec{}

Recall that the object $\on{Poinc}^{\on{Vac,glob}}_{G,!}$ is the !-direct image along the map
$$\fp:\Bun_{N,\rho(\omega_X)}\to \Bun_G$$
of $(\chi^{\on{glob}})^*(\on{exp})$, see \cite[Sect. 1.3.6]{GLC1}.

\medskip

We factor the above map $\fp$ as a composition
\begin{equation} \label{e:factor p}
\Bun_{N,\rho(\omega_X)}\overset{\sff}\to \Bun_{N,\rho(\omega_X)}/T\hookrightarrow \Bun^{(g-1)\cdot 2\rho}_B\to \Bun_G.
\end{equation} 

\sssec{}

Here is the crucial observation:

\medskip

Since $g\geq 2$, the coweight $(g-1)\cdot 2\rho$ belongs to $\Lambda_G^{++}$, and hence the map
$$\Bun^{(g-1)\cdot 2\rho}_B\to \Bun_G$$
is a locally closed embedding (see \cite[Theorem 7.4.3(1')]{DG}). Hence, the !-direct image with respect to it
is fully faithful.

\medskip

Hence, for the proofs of the theorems of in this section, we can perform the calculations on $\Bun^{(g-1)\cdot 2\rho}_B$. 
Here we remark that
because $Z_G\subset B$, the action of $\Bun_{Z_G}$ on $\Bun_G$ lifts to an action on $\Bun^{(g-1)\cdot 2\rho}_B$

\ssec{Proof of Theorems \ref{t:end Poinc vac a} and \ref{t:end Poinc vac a'}} \label{ss:proof of end Poinc}

\sssec{}

The map
\begin{equation} \label{e:Bun N mod T}
\Bun_{N,\rho(\omega_X)}/T\hookrightarrow \Bun^{(g-1)\cdot 2\rho}_B
\end{equation}
is a closed embedding as it comes via base-change from the closed embedding
$\on{pt}/T \overset{\rho(\omega_X)}{\rightarrow} \Bun_T$.

\medskip

Note that the action of 
$$B(Z_G)\subset \Bun_{Z_G}$$
on $\Bun^{(g-1)\cdot 2\rho}_B$ preserves the above locally closed embedding.

\medskip

Hence, in order to prove the theorems in this subsection, we can perform 
the calculations on $\Bun_{N,\rho(\omega_X)}/T$.

\sssec{}

Let $r$ be the semi-simple rank of $G$. For each vertex $i$ of the Dynkin diagram $I$ we have a canonically
defined map
$$\chi^{\on{glob},i}:\Bun_{N,\rho(\omega_X)} \to \BA^1.$$

The composite map
$$\Bun_{N,\rho(\omega_X)} \overset{\chi^{\on{glob},I}}\to (\BA^1)^r=\BA^r \overset{\on{sum}}\to \BA^1$$
is the map $\chi^{\on{glob}}$ (see \cite[Sect. 1.3.6]{GLC1}).

\medskip

We have a Cartesian square
\begin{equation} \label{e:res diag}
\CD
\Bun_{N,\rho(\omega_X)} @>{\chi^{\on{glob},I}}>> \BA^r \\
@V{\sff}VV @VV{\sff'}V \\
\Bun_{N,\rho(\omega_X)}/T @>{\chi^{\on{glob},I}/T}>> \BA^r/T,
\endCD
\end{equation} 
where 
$T$ acts on $\BA^r$ via
$$T\to T_{\on{ad}}\simeq (\BG_m)^r.$$

\sssec{}

Since 
$$Z_G=\on{ker}(T\to T_{\on{ad}}),$$
we have an action of $B(Z_G)$ on $\BA^r/T$, and the 
above map 
$$\chi^{\on{glob},I}/T:\Bun_{N,\rho(\omega_X)}/T\to \BA^r/T$$
is $B(Z_G)$-equivariant.

\medskip

Since the horizontal arrows in \eqref{e:res diag} are towers of torsors with respect to vector group-stacks\footnote{I.e., we take group-stacks 
locally of the form $\CE_1/\CE_0$, where $\CE_0$ and $\CE_1$ are vector spaces, viewed as group-schemes.} 
(and hence, the corresponding pullback functors are fully faithful), we obtain that 
in order to prove Theorems \ref{t:end Poinc vac a} and \ref{t:end Poinc vac a'},
it suffices to perform the corresponding calculation on $\BA^r/T$. 

\medskip

I.e., we have to show that for every $\alpha\in (Z_G)^\vee$, the map 
$$k\to \CEnd((\sff'_!\circ \on{sum}^*(\on{exp}))_\alpha)$$
is an isomorphism, where $(\sff'_!\circ \on{sum}^*(\on{exp}))_\alpha$ is the corresponding direct summand of
$\sff'_!\circ \on{sum}^*(\on{exp})$.

\medskip

Note also that
\begin{equation} \label{e:exp r}
\on{sum}^*(\on{exp})\simeq \on{exp}^{\boxtimes r}.
\end{equation} 

\sssec{}

Note that the map
$$\BA^r\to \BA^r/T$$
naturally factors as
$$\BA^r=\BA^r \times \on{pt} \to \BA^r\times \on{pt}/Z_G\simeq \BA^r/Z_G \overset{\sff''}\to \BA^r/T.$$

The direct summand $(\sff'_!\circ \on{sum}^*(\on{exp}))_\alpha$ can be explicitly described as
\begin{equation} \label{e:alpha dir sum}
\sff''_!( \on{sum}^*(\on{exp})\boxtimes \psi_\alpha),
\end{equation} 
where $\psi_\alpha$ is the (character) sheaf of $\on{pt}/Z_G$ corresponding to the character $\alpha$.

\sssec{}

First, we claim that \eqref{e:alpha dir sum} is the *-extension of its own restriction to 
$$(\BA^1-0)^r/T\hookrightarrow \BA^r/T.$$

To prove this, we have to show this for each coordinate hyperplane separately, and it is sufficient
to do so for the pullback of \eqref{e:alpha dir sum} to $\BA^r$; denote it by $\CG_\alpha$. 

\medskip

Thus, let us choose a vertex $i$ of the Dynkin diagram. Let
$$\sk_i:\BA^{r-1}\leftrightarrow \BA^r:\pi_i$$
be the inclusion and the projection into the corresponding hyperplane. We wish to show that $\sk_i^!(\CG_\alpha)=0$. 

\medskip

Choose a coweight of $T$ that
is a positive multiple of the fundamental coweight corresponding to $i$ (the latter might itself not be a
coweight of $T$). Note that the resulting action of $\BG_m$ \emph{contracts} $\BA^r$ onto the above
copy of $\BA^{r-1}$. 

\medskip

Our object is equivariant with respect to $T$, and hence also with respect to the above action of $\BG_m$. 
Hence, by the \emph{contraction principle},
$$\sk_i^!(\CG_\alpha)\simeq (\pi_i)_!(\CG_\alpha).$$

Now, $(\pi_i)_!(\CG_\alpha)$ is the pullback along $\BA^{r-1}\to \BA^{r-1}/T$ of 
$$\sff''_!((\pi_i)_!(\on{sum}^*(\on{exp}))\boxtimes \psi_\alpha),$$
where by a slight abuse of notation we denote by the same symbol $\sff''$ the map
$$\BA^{r-1}\times \on{pt}/Z_G \overset{\sff''}\to \BA^{r-1}/T.$$

However,
$$(\pi_i)_!(\on{sum}^*(\on{exp}))=0$$
(this follows, e.g., from \eqref{e:exp r}).

\sssec{}

Thus, it remains to show that the (derived) endomorphisms of the restriction of \eqref{e:alpha dir sum} to
the open substack $(\BA^1-0)^r/T$ are scalars. 

\medskip

Consider the corresponding map
$$\overset{\circ}\sff{}'':(\BA^1-0)^r\times \on{pt}/Z_G\to (\BA^1-0)^r/T\simeq \on{pt}/Z_G,$$
so that 
$$\sff''_!( \on{sum}^*(\on{exp})\boxtimes \psi_\alpha)|_{(\BA^1-0)^r/T}\simeq
(\overset{\circ}\sff{}'')_!(\on{sum}^*(\on{exp})|_{(\BA^1-0)^r}\boxtimes \psi_\alpha).$$

Note that the map $\overset{\circ}\sff{}''$ is equivariant with respect to $\on{pt}/Z_G$. Hence, the object 
\begin{equation} \label{e:alpha open}
(\overset{\circ}\sff{}'')_!(\on{sum}^*(\on{exp})|_{(\BA^1-0)^r}\boxtimes \psi_\alpha)
\end{equation}
is a tensor multiple of $\psi_\alpha$. 

\medskip

It remains to show that this multiple is by a one-dimensional vector space. I.e., we have to show that the fiber of
\eqref{e:alpha open} at the point $1\in \on{pt}/Z_G$ is 1-dimensional (in some cohomological degree). 

\sssec{}

It is easy to see that the map $\overset{\circ}\sff{}''$ equals the composition 
$$(\BA^1-0)^r\times \on{pt}/Z_G \simeq \BG_m^r\times  \on{pt}/Z_G \to \on{pt}/Z_G \times \on{pt}/Z_G \to \on{pt}/Z_G ,$$
where:

\begin{itemize}

\item The second arrow is induced by the map $\BG_m^r\simeq T_{\on{ad}}\to \on{pt}/Z_G$,
corresponding to the cover $T\to T_{\on{ad}}$;

\medskip

\item The third arrow is the multiplication map. 

\end{itemize}

From here we obtain that the fiber of \eqref{e:alpha open} at $1\in \on{pt}/Z_G$ identifies with
$$\on{C}^\cdot_c(\BG_m^r,\on{exp}^{\boxtimes r}|_{\BG_m^r}\otimes \Psi^{\otimes -1}_\alpha),$$
where $\Psi_\alpha$ is the Kummer sheaf on $\BG_m^r$ corresponding to $\psi_\alpha$. 

\medskip

The question readily reduces to the case when $r=1$. I.e., we are interested in 
\begin{equation} \label{e:Gauss}
\on{C}^\cdot_c(\BG_m,\on{exp}|_{\BG_m}\otimes \Psi),
\end{equation} 
where $\Psi$ is a Kummer sheaf. We wish to show that this cohomology is 1-dimensional. 
This is well-known, but we supply an argument for completeness. 

\sssec{}

We distinguish two scenarios.

\medskip

If $\Psi$ is trivial, the result follows from the fiber sequence
$$\on{C}^\cdot_c(\BG_m,\on{exp}|_{\BG_m})\to 
\on{C}^\cdot_c(\BA^1,\on{exp}) \to \on{C}^\cdot_c(\on{pt},\on{exp}|_0),$$
since the middle term vanishes.

\medskip

If $\Psi$ is non-trivial, \eqref{e:Gauss} is the fiber at the point $1$ in the \emph{dual} $\BA^1$ of the Fourier transform
of (the clean) extension of $\Psi$ to $\BA^1$. This Fourier transform is an irreducible perverse sheaf, which is
equivariant for the action of $\BG_m$ against (the inverse of) $\Psi$. Hence, this Fourier transform is a tensor
multiple of (the inverse of) $\Psi$.

\qed[Theorems \ref{t:end Poinc vac a} and \ref{t:end Poinc vac a'}]

\ssec{Proof of Theorems \ref{t:end Poinc vac b} and \ref{t:end Poinc vac b'}} \label{ss:proof of end Poinc transl}

\sssec{}

To prove Theorems \ref{t:end Poinc vac b} and \ref{t:end Poinc vac b'}, it suffices to show that the image of the
closed embedding \eqref{e:Bun N mod T} and its translate by means of a non-trivial $\CP_{Z_G}$ are disjoint.

\medskip

For that, it suffices to show that their images under the projection
$$\Bun_B\to \Bun_T$$
are disjoint. 

\sssec{}

For the latter it is sufficient to show that under the further projection
$$\Bun_T\to \Bun_T/B(T)$$
(where $\Bun_T/B(T)$ is the coarse moduli scheme), the above two images
correspond to two distinct points:
$$\rho(\omega_X) \text{ and } \CP_{Z_G}\cdot \rho(\omega_X).$$

\sssec{}

However, the latter follows from the fact that $\CP_{Z_G}$ is \emph{non-trivial} as a $T$-bundle. 

\qed[Theorems \ref{t:end Poinc vac b} and \ref{t:end Poinc vac b'}]

\section{Geometry of \texorpdfstring{$\Bun_G$}{geomBunG}} \label{s:LS}

The goal of this section is to prove Propositions \ref{p:pi 1 Bun}, \ref{p:compl 1}, \ref{p:compl 2} and \ref{p:conn on stable bundles}. 

\medskip

For the duration of this section, we will change the notation from $\cG$ to $G$, and we assume that it is
semi-simple. 

\ssec{Proof of \propref{p:pi 1 Bun}} \label{ss:calc pi 1 Bun}

\sssec{}

Let $G'$ be a reductive group equipped with a surjective map 
$$\phi:G'\to G,$$
such that\footnote{Such a data is known as a \emph{z-extension} in the sense of Kottwitz.}

\begin{itemize}

\item The kernel $T_0$ of $\phi$ is a (connected) torus (which automatically lies in the center of $G'$);

\item The derived group of $G'$ is simply-connected.

\end{itemize}

\sssec{}

Denote
$$T_1:=G'_{\on{ab}}.$$

We obtain an isogeny 
$$\phi':=T_0\to T_1,$$
and it is easy to see that we have a canonical isomorphism
\begin{equation} \label{e:pi 1 via tori}
\on{ker}(\phi')\simeq \pi_1(G):
\end{equation} 
indeed
$$\pi_1(G)\simeq \on{ker}(G'_{\on{der}}\to G)\simeq \on{ker}(T_0\to T_1).$$

\sssec{Example}

One can take $G'$ be the \emph{dual} group of
$$\cG\times \cT/Z_\cG,$$
where $\cT$ is the (abstract) Cartan of $\cG$.

\medskip

Then $$T_0=(\cT/Z_\cG)^\vee \simeq T_{\on{sc}},$$ where $T_{\on{sc}}$ is the (abstract) Cartan
of the simply-connected cover $G_{\on{sc}}$ of $G$. We also have $T_1=T$, so \eqref{e:pi 1 via tori} becomes
the isomorphism
$$\pi_1(G)\simeq \on{ker}(G_{sc}\to G)\simeq \on{ker}(T_{sc}\to T).$$

\sssec{Example}

Let $G=PGL_n$. In this case we can take $G':=GL_n$. 
We have 
$$T_0\simeq \BG_m \text{ and } T_0\simeq \BG_m,$$
and the map $\phi'$ is raising to the power $n$. 

\medskip

Then \eqref{e:pi 1 via tori} becomes the identification
$$\pi_1(PGL_n)\simeq \mu_n.$$

\sssec{}

Note that the map
\begin{equation} \label{e:to 2nd torus}
\Bun_{G'}\to \Bun_{T_1}
\end{equation} 
is smooth with fibers that are connected and simply-connected:

\medskip

Indeed, the fibers are isomorphic to $\Bun_{G_1}$,
where $G_1$ is a twisted form of $[G',G']$, the derived group of $G'$, and
the moduli stack of bundles for a simply-connected group is connected and 
simply-connected (see \cite[Corollary 3.4]{BMP}). 

\medskip

Hence, we obtain that the map \eqref{e:to 2nd torus}
induces an isomorphism of the $\tau_{\leq 1}$-truncations of \'etale homotopy types.

\sssec{}

The map $\phi$ induces an isomorphism
$$\Bun_{G'}/\Bun_{T_0}\simeq \Bun_G.$$

\medskip

Hence, we obtain that the map
$$\Bun_G\simeq \Bun_{G'}/\Bun_{T_0}\to \Bun_{T_1}/\Bun_{T_0}$$
induces an isomorphism of the $\tau_{\leq 1}$-truncations of \'etale homotopy types.

\sssec{}

Finally, we note that the isomorphism \eqref{e:pi 1 via tori} induces am identification 
$$\Bun_{T_1}/\Bun_{T_0} \simeq B^2(\on{C}^\cdot(X,\pi_1(G)))_{\on{et}},$$
and the resulting map
$$\Bun_G\to B^2(\on{C}^\cdot(X,\pi_1(G)))_{\on{et}}$$
is the same as \eqref{e:to gerbes}.

\qed[\propref{p:pi 1 Bun}]

\ssec{Proof of \propref{p:compl 1}} \label{ss:proof of compl 1}

\sssec{} At this point, we 
refer the reader to Appendix \ref{s:stable} for background
on stable bundles and some relevant notation.

\medskip 

Let $\Bun_G^{\on{unstbl}} \subset \Bun_G$ be the closed substack
of unstable\footnote{Here``unstable" means ``not stable", rather than ``not semi-stable".}  bundles. We need to show $\Bun_G^{\on{unstbl}}$ has 
codimension $\geq 2$ under our hypotheses on $X$.

\medskip 

By definition of stability for $G$-bundles, 
every point of $\Bun_G^{\on{unstbl}}$ is in the image of some map
$\Bun_P^{\lambda}$ for $P \subsetneq G$ a maximal parabolic
with Levi quotient $M$ and $\lambda \in \pi_{1,\on{alg}}(M)$
satisfying $\langle 2\rhoch_P,\lambda\rangle \geq 0$.

\medskip

Therefore, it suffices to show
$$\dim(\Bun_P^{\lambda}) \leq \dim(\Bun_G) - 2$$
for such $\lambda$.

\sssec{}

By Riemann-Roch, 
$$\dim(\Bun_G)=\dim(\fg)\cdot (g-1) = 
\dim(\fm)\cdot(g-1)+2\dim(\fn(P))\cdot (g-1)$$
and
$$\dim(\Bun^\lambda_P)=\dim(\fm)\cdot (g-1)+\dim(\fn(P))\cdot (g-1) -\langle 2\rhoch_P,\lambda\rangle,$$
where $2\rhoch_P$ is as in Appendix \ref{s:stable}. 

\medskip 

Therefore, we have to show 
$$2-\langle 2\rhoch_P,\lambda\rangle \leq \dim(\fn(P))\cdot (g-1).$$

By assumption on $\lambda$, the left hand side is at most $2$.
As $P$ is a proper parabolic and $g \geq 2$, this inequality obviously 
holds outside the exceptional
case where $\dim(\fn(P)) = (g-1) = 1$, which only happens if
$g = 2$ and $G_{\on{ad}}$ contains a $PGL_2$ factor.

\qed[\propref{p:compl 1}]

\begin{rem}
Note that the assertion of \propref{p:compl 1} is false for $G=SL_2$ and $g=2$: in this case the dimension
of the semi-stable but unstable locus is $2$, which is $>$ than 
$$1=3-2=\dim(\Bun_G)-2.$$
\end{rem} 

\ssec{Proof of \propref{p:compl 2}} \label{ss:proof of compl 2}

It is enough to show that for every maximal parabolic subgroup $P \subsetneq G$, we have
\begin{equation} \label{e:LS ineq}
\dim(\LS_P)\leq \dim(\LS_G)-2=\dim(\fg)\cdot (2g-2)-2.
\end{equation} 

\sssec{}

Consider the stack $\LS_M$. It is quasi-smooth of virtual dimension\footnote{See, e.g., \cite[Proposition 10.4.5]{AG}.} 
$$\dim(\fm)\cdot (2g-2),$$
and if $g\geq 2$, by \thmref{t:LS CM}, its underlying classical
stack is a locally complete intersection of dimension
$$\dim(\fm)\cdot (2g-2)+\dim(\fz_M),$$
where $\fz_M:=\on{Lie}(Z_M)$.

\medskip

Indeed, this follows by considering the fibration
$\LS_M \to \LS_{M/[M,M]}$, applying \thmref{t:LS CM} for
the derived group $[M,M]$, and noting that
$\LS_{\BG_m}$ has dimension one more than its virtual
dimension by explicit analysis.

\sssec{}

Consider the map
\begin{equation} \label{e:q map}
\sfq:\LS_P\to \LS_M.
\end{equation}

It is quasi-smooth of virtual relative dimension 
$$\dim(\fn(P))\cdot (2g-2).$$

\begin{lem} \hfill \label{l:LS P estim}

\smallskip

\noindent{\em(a)} Each fiber of the map $\sfq$ has
dimension $\leq \dim(\fn(P))\cdot (2g-1)$.

\smallskip

\noindent{\em(b)} There exists a dense open substack of $\LS_M$
over which $\sfq$ is smooth.

\end{lem}

Let us assume this lemma for a moment and proceed with the proof of \eqref{e:LS ineq}. 

\medskip 

It follows from point (b) of the lemma that the generic fiber of $\sfq$ has dimension
$\dim(\fn(P))\cdot (2g-2)$, so the substack of $\LS_M$
over which $\sfq$ has fibers of larger dimension 
has codimension at least one. We obtain:

\begin{cor}\label{c:LS P estim}

$\dim(\LS_P) \leq \dim(\LS_M) + \dim(\fn(P))\cdot (2g-1) - 1$.

\end{cor}

\sssec{}

From \corref{c:LS P estim}, we obtain:
$$\dim(\LS_P)\leq \dim(\fm)\cdot (2g-2)+\dim(\fz_M) + \dim(\fn(P))\cdot (2g-1)-1.$$

Thus it remains to show that, under the assumptions of \propref{p:compl 2}, 
\begin{multline} \label{e:LS ineq 1}
\dim(\fm)\cdot (2g-2)+\dim(\fz_M) + \dim(\fn(P))\cdot (2g-1)-1\leq \dim(\fg)\cdot (2g-2)-2 \\ 
=\dim(\fm) \cdot (2g-2) + 2\dim(\fn(P)) \cdot (2g-2) - 2.
\end{multline} 

This is equivalent to 
\begin{equation} \label{e:LS ineq 2}
\dim(\fz_M)+1 \leq 
\dim(\fn(P))\cdot(2g-3).
\end{equation} 

\sssec{}

We now use the assumption that $G$ is semi-simple and that the corank of $P$ is one, so that 
$\dim(\fz_M)=1$, i.e., \eqref{e:LS ineq 2} becomes
\begin{equation} \label{e:LS ineq 3}
2 \leq \dim(\fn(P))\cdot(2g-3).
\end{equation} 

This holds automatically if $g\geq 3$. If $g=2$, the above inequality can only be violated if $\dim(\fn(P))=1$,
but this only happens if the Dynkin diagram of $G$ has an $A_1$ factor.

\qed[\propref{p:compl 2}]

\begin{rem}
Note that the assertion of \propref{p:compl 2} is false for $G=SL_2$ and $g=2$: in this case the dimension
of the reducible locus is $5$, which is greater than 
$$4=6-2=\dim(\LS_G)-2.$$
\end{rem} 

\sssec{Proof of \lemref{l:LS P estim}(b)}

It suffices to show that for every $\sigma_M\in \LS_M$, there exists a point $\sigma'_M$ that lies in the same irreducible
component, over which the fiber of \eqref{e:q map} is smooth.

\medskip

Note that for $\sigma_M\in \LS_M$ and 
$$\sigma_P\in \sfq^{-1}(\sigma_M)\simeq \LS_{N(P)_{\sigma_M}},$$
the obstruction to the smoothness of the fiber is
$$H^2(X,\fn(P)_{\sigma_P}).$$

The latter is non-zero if for some subquotient $V$ of $\fn(P)$ as a $M$-representation, 
the local system $V_{\sigma_M}$ admits a trivial quotient. 

\medskip

Let $Z_M^0$ denote the neutral connected component 
of $Z_M$ and consider its action on $\fn(P)$. 
It acts on every $V$ as above by a \emph{non-trivial}
character. Hence, for a generic point $\sigma_Z\in \LS_{Z_M}$ and 
$$\sigma'_M:=\sigma_Z\otimes \sigma_M,$$
the local system $V_{\sigma'_M}$ will not have trivial quotients. 

\medskip

Since $\LS_{Z_M^0}$ is irreducible, its action on $\LS_M$ preserves irreducible components, i.e.,
$\sigma'$ lies in the same irreducible component as $\sigma$.

\qed[\lemref{l:LS P estim}(b)]

\sssec{Proof of \lemref{l:LS P estim}(a)}

We will use the following assertion:

\begin{lem}
Let $Y$ be a quasi-smooth scheme of virtual dimension $d$. Suppose that $m$ is an integer such that
for all field-valued
points $y\in Y$ we have
$$\dim(H^{-1}(T^*_y(Y)))\leq m.$$
Then
$\dim(Y)\leq d+m$.
\end{lem}

\begin{proof}

It is enough to show that for every field-valued point $y\in Y$, the dimension of the classical
cotangent space to $Y$ at $y$ is $\leq d+m$. However, the classical
cotangent space is just $H^0(T^*_y(Y))$. We have
$$\dim(H^0(T^*_y(Y)))=\dim(H^0(T^*_y(Y))-\dim(H^1(T^*_y(Y)))+\dim(H^1(T^*_y(Y))),$$
where $\dim(H^0(T^*_y(Y))-\dim(H^1(T^*_y(Y)))$ is the virtual dimension of $Y$.

\end{proof} 

\begin{cor}
Let $\CY$ be a quasi-smooth algebraic stack of virtual dimension\footnote{Recall that for a quasi-smooth algebraic stack $\CY$, its dimension/virtual dimension are defined
as follows: for a smooth cover $Y\to \CY$ with $Y$ a scheme, the dimension/virtual dimension of $\CY$ equals that of $Y$ minus the dimension of the fibers.}
 $d$. Suppose that $m$ is an integer such that
for all field-valued
points $y\in \CY$ we have
$$\dim(H^{-1}(T^*_y(\CY)))\leq m.$$
Then
$\dim(\CY)\leq d+m$.
\end{cor}

We apply the corollary to the fibers of the map
\eqref{e:q map}, i.e., to the stacks
$$\LS_{N(P)_{\sigma_M}}, \quad \sigma_M\in \LS_M.$$

It remains to show that
$$\dim(H^{-1}(T^*_{\sigma_P}(\LS_{N(P)_{\sigma_M}}))\leq \dim(\fn(P)), \quad \sigma_P\in \LS_{N(P)_\sigma}.$$

We have:
$$T^*_{\sigma_P}(\LS_{N(P)_{\sigma_M}})\simeq \on{C}^\cdot(X,\fn(P)_{\sigma_P}[1])^\vee,$$
so 
$$H^{-1}(T^*_{\sigma_P}(\LS_{N(P)_{\sigma_M}}))\simeq H^2(X,\fn(P)_{\sigma_P})^\vee,$$
which identifies with 
$$H^0\left(X,(\fn(P))^\vee_{\sigma_P}\right)$$
by Verdier duality.

\medskip

We clearly have
$$\dim(H^0(X,(\fn(P))^\vee_{\sigma_P})\leq \dim(\fn(P)).$$

\qed[\lemref{l:LS P estim}(a)]

\ssec{Proof of \propref{p:conn on stable bundles}} \label{ss:conn on stable bundles}

\sssec{}

The fact that the non-empty fibers of the map 
$$\LS_G\to \Bun_G$$
are affine spaces is completely general: 

\medskip

For a given $\CP_G\in \Bun_G$, the fiber in question is a torsor for the (derived) vector space
\begin{equation} \label{e:conn on bun}
\Gamma(X,\fg_{\CP_G}\otimes \omega_X)
\end{equation} 
(see \cite[Corollary 10.5.5]{AG}\footnote{At the classical level, it is clear that connections on a given $G$-bundle on a curve
form a torsor over the space of (Ad-twisted) $\cg$-valued 1-forms. However, in derived algebraic geometry, this requires some care, hence the reference.}).

\medskip

\noindent {\it Warning}: In the above formula $\omega_X$ stands for the canonical line bundle on $X$, and \emph{not} 
the dualizing sheaf of $X$, which is the $[1]$ shift of that. This deviates from the conventions adopted
in this series, according to which for a prestack $\CY$, we denote by $\omega_\CY$ the dualizing
sheaf on $\CY$. So, the curve $X$ itself is the only exception for this convention.  

\sssec{}

Let us show that the map in question is smooth over the stable locus. This is equivalent to the fibers
being smooth as derived schemes. 

\medskip

By the above torsor property, it suffices to show that if $\CP_G\in \Bun_G$ is stable, then the derived vector
space \eqref{e:conn on bun} is classical, i.e., that
$$H^1(X,\fg_{\CP_G}\otimes \omega_X)=0.$$

By Serre duality (and using the Killing form on $\fg$), this is equivalent to
$$H^0(X,\fg_{\CP_G})=0.$$

I.e., we need to show that stable bundles do not admit infinitesimal automorphisms. This is standard;
we supply a proof for completeness. 

\sssec{}

Suppose the contrary. Let $A$ be an infinitesimal automorphism of $\CP_G$. First, we show that $A$ is nilpotent.

\medskip

Consider the characteristic polynomial of the
$\CO_X$-valued Higgs field $A$, i.e., the map
$$X\to \ft/\!/W=:\fa$$
coming from $A$.

\medskip

This map is necessarily constant; denote its image by $a$. 

\medskip

Let $t\in \ft$ be a semi-simple element that maps under $\ft\to \fa$ to $a$. 
In this case, $\CP_G$ admits a reduction to $Z_G(t)$, which is a Levi subgroup. 

\medskip

If $a$ were not nilpotent, we would have $t\neq 0$, and $Z_G(t)$ is a \emph{proper}
Levi subgroup. However, the existence of such a reduction contradicts the assumption that $\CP_G$ is stable.

\sssec{}

Thus, $A$ is (non-zero) nilpotent. The Jacobson-Morosov theory supplies a (decreasing) filtration on the vector
bundle $\fg_{\CP_G}$ so that $$(\fg_{\CP_G})^{\geq 1}\subset \fg_{\CP_G}$$
is the unipotent radical of a parabolic reduction canonically associated to $A$, and 
the operator $(\on{ad}_A)^n$ defines a map
\begin{equation} \label{e:nilp map}
\on{gr}^{-n}(\fg_{\CP_G})\to \on{gr}^n(\fg_{\CP_G}),
\end{equation} 
which is an isomorphism at the generic point of $X$.

\medskip

Since $\CP_G$ was assumed stable, $\deg((\fg_{\CP_G})^{\geq 1})<0$. Hence, for some $n\geq 1$, we have
$\deg(\on{gr}^n(\fg_{\CP_G}))<0$. However,
$$\deg(\on{gr}^{-n}(\fg_{\CP_G}))=-\deg(\on{gr}^n(\fg_{\CP_G})),$$
and this contradicts the existence of \eqref{e:nilp map}. 

\sssec{}

It remains show that every stable $G$-bundle \emph{admits} a connection. For a general $\CP_G\in \Bun_G$,
the obstruction to having a connection is given by its \emph{Atiyah class}, which is an element of
$$H^1(X,\fg_{\CP_G}\otimes \omega_X).$$

However, we have just proved that this group is zero
for stable $\CP_G$.

\qed[\propref{p:conn on stable bundles}]

\begin{rem}
The above argument can be refined to prove the following criterion (originally due to A.~Weil) for a $G$-bundle 
$\CP_G$ to admit a connection: 

\medskip

This happens if and only if, for every reduction of $\CP_G$ to a \emph{Levi subgroup} $M$, this reduction, viewed as
an $M$-bundle, has degree $0$. See \cite{AB} for more
details.

\end{rem}

\section{2-Fourier-Mukai transform of the automorphic category} \label{s:gerbes}

The goal of this section is to prove Theorems \ref{t:Hecke Z 0} and \ref{t:Hecke Z 1}. We will do
so by considering a more general picture that involves twisting the constant group-scheme with fiber $G$ by $Z_G$-gerbes.

\ssec{The notion of 2-Fourier-Mukai transform}

Recall that the usual Fourier-Mukai transform is a functor between categories of quasi-coherent sheaves on a 
pair of prestacks.

\medskip

In this subsection we introduce the notion of \emph{2-Fourier-Mukai transform}, which is a (2-)functor between
2-categories of \emph{sheaves of categories} on a pair of prestacks.

\sssec{}

Let $\CY_1$ and $\CY_2$ be a pair of prestacks equipped with a map
\begin{equation} \label{e:FM pairing}
\CY_1\times \CY_2\to \on{pt}/\BG_m,
\end{equation}
i.e., a line bundle, denoted $\CL_{1,2}$ on $\CY_1\times \CY_2$.

\medskip

Assume that the functor $(p_2)_*:\QCoh(\CY_1\times \CY_2)\to \QCoh(\CY_2)$ is continuous
(this happens, e.g., when $\CY_1$ is quasi-compact with an affine diagonal). 

\medskip

Consider the functor
\begin{equation} \label{e:FM functor}
\on{FM}_{\CY_1\to \CY_2}:\QCoh(\CY_1)\to \QCoh(\CY_2), \quad 
\CF\mapsto (p_2)_*(\CL_{1,2}\otimes p_1^*(\CF)).
\end{equation}

We shall say that the map \eqref{e:FM pairing} is of Fourier-Mukai type if the functor \eqref{e:FM functor}
is an equivalence.

\sssec{} \label{sss:2FM setting}

Let us now be given a map 
\begin{equation} \label{e:2-FM pairing}
\CY_1\times \CY_2\to \on{Ge}_{\BG_m}(\on{pt}),
\end{equation}
where $\on{Ge}_{\BG_m}(\on{pt}) = B^2(\BG_m)_{\on{et}}$ is the (2-algebraic) stack classifying $\BG_m$-gerbes. Let
$\CG_{1,2}$ denote the corresponding gerbe on $\CY_1\times \CY_2$.

\medskip

Recall the notion of a \emph{sheaf of categories}, see \cite[Sect. 1.1]{Ga3}. Consider the 2-functor
\begin{equation} \label{e:2-FM functor}
\on{2-FM}_{\CY_1\to \CY_2}:\on{ShvCat}(\CY_1)\to \on{ShvCat}(\CY_2), \quad
\ul\bC\mapsto (p_2)_*((p_1^*(\ul\bC)_{\CG_{1,2}}),
\end{equation}
where:

\begin{itemize}

\item For a morphism $f:\CY'\to \CY''$ between prestacks, $f^*$ denotes the pullback functor
$$\on{ShvCat}(\CY'')\to \on{ShvCat}(\CY'),$$
see \cite[Sect. 3.1.2]{Ga3} (in {\it loc. cit.} it is denoted ${\mathbf{cores}}_f$); 

\medskip

\item For a morphism $f_*:\CY'\to \CY''$ between prestacks, $f_*$ denotes the pushforward functor
$$\on{ShvCat}(\CY')\to \on{ShvCat}(\CY''),$$
see \cite[Sect. 3.1.3]{Ga3} (in {\it loc. cit.} it is denoted ${\mathbf{coind}}_f$); 

\medskip

\item For a prestack $\CY$, and $\bC\in \on{ShvCat}(\CY)$ and a $\BG_m$-gerbe $\CG$ on $\CY$,
we denote by $\bC_\CG$ the twist of $\bC$ by $\CG$, see \cite[Sect. 1.7.2]{GLys}. 

\end{itemize} 

\medskip

We shall say that the map \eqref{e:2-FM pairing} is of 2-Fourier-Mukai type if the functor \eqref{e:2-FM functor}
is an equivalence.

\medskip

\noindent Note that the notion of being of 2-Fourier-Mukai type is a priori
asymmetric. 

\sssec{Example} \label{sss:Ge finite}

Let $\Gamma$ be a finite abelian group, and let $\Gamma^\vee$ be its Cartier dual. Take
$$\CY_1=B^2(\Gamma)_{\on{et}}=:\on{Ge}_{\Gamma}(\on{pt})$$
and 
$$\CY_2=\Gamma^\vee.$$

Then $\on{ShvCat}(\CY_1)$ is the 2-category of DG categories acted on by $\on{pt}/\Gamma$.
In other words, these are categories $\bC$ equipped with an action of $\Gamma$ on $\on{Id}_\bC$.
Decomposing with respect to the characters of $\Gamma$, we obtain that a datum of such $\bC$ is 
equivalent to the datum of a category graded by $\Gamma^\vee$
\begin{equation} \label{e:2-FM functor finite group}
\bC\mapsto \{\bC_\chi,\, \chi\in \Gamma^\vee\}.
\end{equation} 

\medskip

Evaluation defines a map
\begin{equation} \label{e:2-FM pairing finite group}
\on{Ge}_{\Gamma}(\on{pt})\times \Gamma^\vee\to \on{Ge}_{\BG_m}(\on{pt}).
\end{equation}

We claim that \eqref{e:2-FM pairing finite group} is of 2-Fourier-Mukai type. 

\medskip

Indeed, unwinding the definitions, we obtain that the functor $\on{2-FM}_{\on{Ge}_{\Gamma}\to \Gamma^\vee}$ 
is given exactly by \eqref{e:2-FM functor finite group}, and hence is an equivalence. 

\sssec{} \label{sss:Ge finite rev}

Swapping the factors in \eqref{e:2-FM pairing finite group} we obtain a pairing
\begin{equation} \label{e:2-FM pairing finite group reverse}
\Gamma^\vee\times \on{Ge}_{\Gamma}(\on{pt})\to \on{Ge}_{\BG_m}(\on{pt}),
\end{equation}
and it is easy to see that it is also of 2-Fourier-Mukai type. 

\medskip

Indeed, the corresponding functor $\on{2-FM}_{\Gamma^\vee\to \on{Ge}_{\Gamma}(\on{pt})}$ is
the inverse of $\on{2-FM}_{\on{Ge}_{\Gamma}(\on{pt})\to \Gamma^\vee}$ up to the inversion on $\Gamma$. 

\begin{rem}\label{r:not 1 aff}
The central players in the paper \cite{Ga3} are prestacks that are \emph{1-affine}, 
i.e., those for each the functor
of \emph{enhanced} global sections
\begin{equation} \label{e:glob sect shv categ}
\on{ShvCat}(\CY)\overset{\bGamma^{\on{enh}}(\CY,-)}\longrightarrow \QCoh(\CY)\mmod
\end{equation} 
is an equivalence.

\medskip

Note that the prestack $\on{Ge}_{\Gamma}(\on{pt})$ is \emph{not} 1-affine. Namely 
$\QCoh(\on{Ge}_{\Gamma}(\on{pt}))\simeq \Vect$, and the functor \eqref{e:glob sect shv categ} sends
$\bC$ as above to $\bC_0$, i.e., the fiber of $\on{2-FM}_{\on{Ge}_{\Gamma}(\on{pt})\to \Gamma^\vee}(\bC)$
at the point $0\in \Gamma^\vee$. 

\end{rem}

\sssec{Example} \label{sss:tors finite}

For $\Gamma$ as above, take 
$$\CY_1:=\on{pt}/\Gamma \text{ and } \CY_2:=\on{pt}/\Gamma^\vee.$$

Cup product defines a map
\begin{equation} \label{e:2-FM pairing finite group sym}
\on{pt}/\Gamma \times \on{pt}/\Gamma^\vee \to \on{Ge}_{\BG_m}(\on{pt}).
\end{equation}

We claim that \eqref{e:2-FM pairing finite group sym} is of 2-Fourier-Mukai type. 

\medskip

Note that $\on{ShvCat}(\on{pt}/\Gamma)$ (resp., $\on{ShvCat}(\on{pt}/\Gamma^\vee)$) 
identifies with the 2-category of DG categories equipped with an action of $\QCoh(\Gamma)$ (resp., $\QCoh(\Gamma^\vee)$),
viewed as a monoidal category with respect to \emph{convolution}. Note also that $\on{pt}/\Gamma$ is 1-affine, and
$$\QCoh(\on{pt}/\Gamma)\simeq \Rep(\Gamma).$$

Unwinding the definitions, we obtain that the functor $\on{2-FM}_{\on{pt}/\Gamma\to \on{pt}/\Gamma^\vee}$ 
identifies with 
$$\on{ShvCat}(\on{pt}/\Gamma) \overset{\bGamma^{\on{enh}}(\on{pt}/\Gamma,-)}\longrightarrow
\Rep(\Gamma)\mmod\simeq \QCoh(\Gamma^\vee)\mmod,$$
where we identify 
$$\Rep(\Gamma) \simeq \QCoh(\Gamma^\vee)$$
by Fourier transform. 

%

\sssec{} \label{l:FM adj}

Let us return to the general setting of \secref{sss:2FM setting}. Let us make the following assumptions
(essentially, on the geometric properties of $\CY_1)$:

\begin{itemize}

\item For $\ul\bD\in \ShvCat(\CY_1\times \CY_2)$, the functor
$$(p_2)^*\circ (p_2)_*(\ul\bD)\to \bD$$
admits a left adjoint;

\item For $\ul\bC_2\in \ShvCat(\CY_2)$, the functor
$$\ul\bC_2\to (p_2)_*\circ (p_2)^*(\ul\bC_2)$$
admits a left adjoint.

\end{itemize}

\medskip

Note that this case, the above left adjoints provide a unit and counit, making
$(p_2)^*$ the \emph{right} adjoint of $(p_2)_*$ (i.e., the usual 
$((p_2)^*,(p_2)_*)$-adjunction is \emph{ambidexterous}.

\medskip

In particular, the new $((p_2)_*,(p_2)^*)$-adjunction gives rise to an adjunction\footnote{In the formula below, $\on{2-FM}_{\CY_1\to \CY_2}$
is defined using the initial gerbe $\CG_{1,2}$, and $\on{2-FM}_{\CY_2\to \CY_1}$ is defined using its inverse.} 
$$(\on{2-FM}_{\CY_1\to \CY_2},\on{2-FM}_{\CY_2\to \CY_1}).$$

\medskip

In particular, we obtain:

\begin{lem} \label{l:invol}
Under the above circumstances, if $\on{2-FM}_{\CY_1\to \CY_2}$ is an equivalence, then so is
$\on{2-FM}_{\CY_2\to \CY_1}$, and 
$$\on{2-FM}_{\CY_2\to \CY_1}\circ \on{2-FM}_{\CY_1\to \CY_2}\simeq \on{Id}.$$
\end{lem} 

\ssec{Some compatibilities} \label{ss:compat}

\sssec{}

In this subsection, we will assume that $\CY_1:=\CH_1$ and $\CY_2:=\CH_2$ are commutative group-prestacks, and the map 
\eqref{e:2-FM pairing} is bilinear. 

\medskip

For a point $h_1\in \CH_1$ (resp., $h_2\in \CH_2)$ let $\CG_{h_1}$ (resp., $\CG_{h_2}$) denote the
corresponding $\BG_m$-gerbe on $\CH_2$ 

\medskip

Let $\ul\bC_1$ be an object of $\ShvCat(\CH_1)$ and set $\ul\bC_2:=\on{2-FM}_{\CH_1\to \CH_2}(\bC_2)$.
Denote also
$$\bC_1:=\bGamma(\CH_1,\ul\bC_1) \text{ and } \bC_2:=\bGamma(\CH_2,\ul\bC_2).$$

\sssec{}

Note that for $h_2\in \CH_2$ we have:
\begin{equation} \label{e:fiber FM gen}
\ul\bC_2|_{h_2}\simeq \bGamma(\CH_1,(\ul\bC_1)_{\CG_{h_2}}),
\end{equation} 
where:

\begin{itemize}

\item $(-)|_{y_2}$ denotes the fiber of a given sheaf of categories at $y_2$;

\item $(-)_{\CG}$ denotes the twist of a given sheaf of categories by a $\BG_m$-gerbe.

\end{itemize}. 

As particular case of \eqref{e:fiber FM gen}, we have:
\begin{equation} \label{e:fiber FM gen zero}
\ul\bC_2|_{\one_{\CH_2}}\simeq \bC_1.
\end{equation} 

\sssec{}

Looping the pairing \eqref{e:2-FM pairing} along $\CH_2$ we obtain a pairing
$$\CH_1\times \CA_2\to \on{pt}/\BG_m.,$$
where 
$$\CA:=\on{Aut}_{\CH_2}(\one_{\CH_2}).$$

\medskip

In particular, a point $a\in \CA$ gives rise to a line
bundle, to be denoted $\CL_a$ on $\CH_1$.

\medskip

Unwinding, we obtain:

\begin{lem} \label{l:Hecke 1 abs gen}
Under the identification of \eqref{e:fiber FM gen zero}, the action of $a$ on $\bC_2|_{\zero_{\CH_2}}$
corresponds to the endofunctor of $\bC_1$, given by tensoring wity $\CL_a$. 
\end{lem} 

\sssec{}

Consider now the group
$$\Omega:=\on{Aut}_{\CA}(\one_\CA).$$

Note that the group $\Omega$ acts on $\CO_{\CH_1}\in \QCoh(\CH_1)$. 

\medskip

Let $\omega$ be an idempotent in the category $\Rep(\Omega)$. Note if $\bd$ is an object of a DG category,
equipped with an action of $\Omega$, we can attach to it a direct summand
$\bd_\omega \subset \bd$.  

\medskip

In particular, to $\omega$
there corresponds a connected component $(\CH_1)_{\omega}$ of $\CH_1$. 

\medskip

Unwinding, we obtain:

\begin{lem} \label{l:Hecke 0 abs gen}
Under the identification of \eqref{e:fiber FM gen zero}, the direct summand 
$$\bC_{1,\omega}:=\Gamma((\CH_1)_\omega,\ul\bC_1)\subset \Gamma(\CH_1,\ul\bC_1)=\bC_1$$
corresponds to the direct summand
$$(\ul\bC_2|_{\one_{\CH_2}})_\omega \subset \ul\bC_2|_{\one_{\CH_2}}.$$
 \end{lem} 

\sssec{}

Let us suppose that the assumptions from \secref{l:FM adj} hold. Unwinding, we obtain:

\begin{lem}  \label{l:fiber and global sects gen}
For $\ul\bC_1\in \ShvCat(\CH_1)$, the following diagram commutes:
$$
\CD
\bGamma(\CH_2,\on{2-FM}_{\CH_1\to \CH_2}(\ul\bC_1)) @>{\on{ev}|_{h_2}}>>  \on{2-FM}_{\CH_2\to \CH_1}(\ul\bC_1)|_{h_2} \\
@A{\text{\eqref{e:fiber FM gen zero}}}A{\sim}A \\
\on{2-FM}_{\CH_2\to \CH_1}\circ \on{2-FM}_{\CH_1\to \CH_2}(\ul\bC_1))|_{\one_{\CH_1}}  & & @A{\sim}A{\text{\eqref{e:fiber FM gen}}}A \\
@AAA  \\
\ul\bC_1|_{\one_{\CH_1}} @>{(\on{ev}|_{\one_{\CH_1}})^L}>> \bGamma(\CH_1,(\ul\bC_1)_{\CG_{h_2}}),
\endCD
$$
where:

\begin{itemize}

\item $\on{ev}_{h_2}$ denotes the natural evaluation functor
$$\bGamma(\CH_2,\ul\bC_2)\to \ul\bC_2|_{h_2}$$
for $\ul\bC_2\in  \ShvCat(\CH_2)$;

\medskip

\item $(\on{ev}|_{\one_{\CH_1}})^L$ is the left adjoint of the similarly defined evaluation functor on $\CH_1$, where we note that
$$(\ul\bC_1)_{\CG_{h_2}}|_{\one_{\CH_1}}\simeq \ul\bC_1|_{\one_{\CH_1}},$$
since $\CG_{h_2}$ is trivialized at $\one_{\CH_1}$;

\medskip

\item The lower left vertical arrow is the unit of the $(\on{2-FM}_{\CH_1\to \CH_2},\on{2-FM}_{\CH_2\to \CH_1})$-adjunction.

\end{itemize}

\end{lem}

\ssec{The 2-Fourier-Mukai transform and Verdier duality} \label{ss:FM and Verdier}

In this subsection we consider a particular pair of prestacks that are 2-Fourier-Mukai dual to each
other. 

\medskip

Both sides have to do with gerbes for a finite abelian group on a smooth and complete curve $X$. 

\sssec{}

Let $\Gamma$ be a finite abelian group as above. Take
$$\CY_1:=\on{Ge}_{\Gamma}(X) \text{ and } \CY_2:=\on{Ge}_{\Gamma^\vee(1)}(X),$$
where $(-)(1)$ denotes the Tate twist, so that 
$$\Gamma^\vee \simeq \Hom(\Gamma,\BZ/n\BZ)(1)$$
for $n\cdot \Gamma = 0$.

\medskip

Verdier duality defines a pairing
$$B^2(\on{C}^\cdot(X,\Gamma))\times B^2(\on{C}^\cdot(X,\Gamma^\vee(1)))\to B^2(\mu_\infty).$$

Composing with $\mu_\infty\to \BG_m$ and applying \'etale sheafification, we obtain a pairing
\begin{equation} \label{e:Verdier pairing Gamma}
\on{Ge}_{\Gamma}(X)\times \on{Ge}_{\Gamma^\vee(1)}(X)\to \on{Ge}_{\BG_m}(\on{pt}).
\end{equation} 

\sssec{}

We claim:

\begin{thm} \label{t:Verdier FM} 
The pairing \eqref{e:Verdier pairing Gamma} is of 2-Fourier-Mukai type.
\end{thm}

\sssec{Proof of \thmref{t:Verdier FM}}

Recall that we can think of $\on{Ge}_{\Gamma}(X)$ as
$$B^2(\on{C}^\cdot(X,\Gamma))_{\on{et}},$$
and similarly for $\on{Ge}_{\Gamma^\vee(1)}(X)$.

\medskip

Choose a pair of points $x_1,x_2\in X$; denote $U_i=X-x_i$. Restriction to $x_1$ 
gives rise to an isomorphism
$$\on{C}^\cdot(X,\Gamma) \simeq \on{C}^\cdot(X;x_1,\Gamma) \times \Gamma,$$
where we note that
$$\on{C}^\cdot(X;x_1,\Gamma) \simeq \on{C}_c^\cdot(U_1,\Gamma).$$ 

The inclusion of $x_2$ gives rise to an isomorphism
$$\on{C}^\cdot(X;x_1,\Gamma) \simeq \on{C}^\cdot(U_2;x_1,\Gamma) \times  \Omega^2(\Gamma(-1)),$$
where $\Omega$ is the functor of loops on spectra, and where we note that
$$\Omega^2(\Gamma(-1)) \simeq \Omega^2(H^2_{\on{et}}(X,\Gamma))\simeq \on{C}^\cdot(X;U_2,\Gamma)$$
and
$$\on{C}^\cdot(U_2;x_1,\Gamma) \simeq \Omega(H^1_{\on{et}}(X,\Gamma)).$$

\medskip

Altogether we obtain an identification
$$\on{C}^\cdot(X,\Gamma)\simeq \Gamma \times \Omega(H^1_{\on{et}}(X,\Gamma)) \times \Omega^2(\Gamma(-1)),$$
and hence 
$$\on{Ge}_{\Gamma}(X)\simeq \on{Ge}_\Gamma(\on{pt}) \times \on{pt}/H^1_{\on{et}}(X,\Gamma) \times H^2_{\on{et}}(X,\Gamma).$$

\medskip

Similarly, we obtain 
$$\on{Ge}_{\Gamma^\vee(1)}(X)\simeq \on{Ge}_{\Gamma^\vee(1)}(\on{pt}) 
\times \on{pt}/H^1_{\on{et}}(X,\Gamma^\vee(1)) \times H^2_{\on{et}}(X,\Gamma^\vee(1)).$$

\medskip

Under this identification, the pairing \eqref{e:Verdier pairing Gamma} splits as a product of:

\begin{itemize}

\item The pairing \eqref{e:2-FM pairing finite group}, where we identify
$H^2(X,\Gamma^\vee(1))\simeq \Gamma^\vee$;

\medskip

\item The pairing \eqref{e:2-FM pairing finite group} with the two sides swapped, where we identify $H^2_{\on{et}}(X,\Gamma)\simeq (\Gamma^\vee(1))^\vee$;

\medskip

\item The pairing \eqref{e:2-FM pairing finite group sym}, where we identify 
$H_{\on{et}}^1(X,\Gamma^\vee(1))\simeq H_{\on{et}}^1(X,\Gamma)^\vee$. 

\end{itemize}

Now the assertion of the theorem follows by combining the examples from 
Sects. \ref{sss:Ge finite}, \ref{sss:Ge finite rev} and \ref{sss:tors finite}. 

\qed[\thmref{t:Verdier FM}]

\sssec{}

It is easy to see that the prestack $\on{Ge}_{\Gamma}(X)$ satisfies the assumptions of \secref{l:FM adj}. Hence, 
combining \thmref{t:Verdier FM} and \lemref{l:invol}, we obtain: 

\begin{cor} \label{c:Verdier FM invol}
The composition 
$$\on{2-FM}_{\on{Ge}_{\Gamma^\vee(1)}(X)\to \on{Ge}_{\Gamma}(X)}\circ 
\on{2-FM}_{\on{Ge}_{\Gamma}(X)\to \on{Ge}_{\Gamma^\vee(1)}(X)}$$
is the involution of $\on{ShvCat}(\on{Ge}_{\Gamma}(X))$ coming 
from the inversion on $\Gamma$,
$$(\Gamma^\vee(1))^\vee(1)\simeq \Gamma.$$
\end{cor}

\ssec{Some further compatibilities}

In this subsection we will specialize the discussion in \secref{ss:compat} to the special case of
$$\CH_1:=\on{Ge}_{\Gamma}(X) \text{ and } \CH_2:=\on{Ge}_{\Gamma^\vee(1)}(X).$$

This is done primarily in order to have a convenient reference in subsequent
subsections. The material here should be skipped and returned to when necessary.

\sssec{} \label{sss:fibers of 2-FM}

For a $\Gamma^\vee(1)$-gerbe $\fG_{\Gamma^\vee(1)}$ on $X$, let $\CG_{\fG_{\Gamma^\vee(1)}}$ be the $\BG_m$-gerbe 
on $\on{Ge}_{\Gamma}(X)$, corresponding to the restriction of the map \eqref{e:Verdier pairing Gamma} along
$$\on{Ge}_{\Gamma}(X)\times \{\fG_{\Gamma^\vee(1)}\}\to \on{Ge}_{\Gamma}(X)\times \on{Ge}_{\Gamma^\vee(1)}(X).$$

We obtain that for an object
$$\ul\bC_{\Gamma}\in \on{ShvCat}(\on{Ge}_{\Gamma}(X))$$
and the corresponding object
$$\ul\bC_{\Gamma^\vee(1)}:=\on{2-FM}_{\on{Ge}_{\Gamma}(X)\to \on{Ge}_{\Gamma^\vee(1)}(X)}(\ul\bC_{\Gamma})\in 
\on{ShvCat}(\on{Ge}_{\Gamma^\vee(1)}(X)),$$
we have
\begin{equation} \label{e:fib glob sect 1}
\ul\bC_{\Gamma^\vee(1)}|_{\fG_{\Gamma^\vee(1)}}\simeq \bGamma\left(\on{Ge}_{\Gamma}(X),(\ul\bC_{\Gamma})_{\CG_{\fG_{\Gamma^\vee(1)}}}\right),
\end{equation}
where:

\begin{itemize}

\item $(-)|_{\fG_{\Gamma^\vee(1)}}$ denotes the fiber of a given sheaf of categories at the point $\fG_{\Gamma^\vee(1)}\in \on{Ge}_{\Gamma^\vee(1)}(X)$; 

\item $(-)_\CG$ denotes the twist of a given sheaf of categories over some prestack by a $\BG_m$-gerbe $\CG$ on that prestack.

\end{itemize}

\sssec{}

By \corref{c:Verdier FM invol}, we also obtain that for a $\Gamma$-gerbe 
$\fG_\Gamma$ on $X$, and the corresponding $\BG_m$-gerbe $\CG_{\fG_\Gamma}$ on $\on{Ge}_{\Gamma^\vee(1)}(X)$, we have
\begin{equation} \label{e:fib glob sect 2}
\ul\bC_{\Gamma}|_{\fG^{-1}_\Gamma}\simeq \bGamma\left(\on{Ge}_{\Gamma^\vee(1)}(X),(\ul\bC_{\Gamma^\vee(1)})_{\CG_{\fG_\Gamma}}\right).
\end{equation}

\sssec{}

Let now
$$0\to \Gamma_1\to \Gamma\to \Gamma_2\to 0$$
be a short exact sequence of finite abelian groups, and let
$$0\to \Gamma_2^\vee\to \Gamma^\vee\to \Gamma_1^\vee\to 0$$
be the dual short exact sequence. 

\medskip

Fix $\fG_{\Gamma_1^\vee(1)}\in \on{Ge}_{\Gamma_1^\vee(1)}(X)$, and let 
$\CG_{\fG_{\Gamma_1^\vee(1)}}$ be the corresponding $\BG_m$-gerbe on $\on{Ge}_{\Gamma_1}(X)$. 

\medskip

Generalizing \eqref{e:fib glob sect 1} and \eqref{e:fib glob sect 2}, we have:

\begin{lem} \label{l:sgrp}
There is a canonical equivalence
$$\bGamma\left(\on{Ge}_{\Gamma_1}(X), (\ul\bC_\Gamma|_{\on{Ge}_{\Gamma_1}(X)})_{\CG_{\fG_{\Gamma_1^\vee(1)}}}\right)
\simeq \bGamma\biggl(\on{Ge}_{\Gamma^\vee(1)}(X)\underset{\on{Ge}_{\Gamma_1^\vee(1)}(X)}\times \{\fG_{\Gamma_1^\vee(1)}\},\ul\bC_{\Gamma^\vee(1)}\biggr).$$
\end{lem}

\sssec{}

For $\ul\bC_{\Gamma}$ as above, denote
$$\bC_\Gamma:=\bGamma(\on{Ge}_{\Gamma}(X),\ul\bC_{\Gamma}) \text{ and }
\bC_{\Gamma^\vee(1)}:=\bGamma(\on{Ge}_{\Gamma^\vee(1)}(X),\ul\bC_{\Gamma^\vee(1)}).$$

Let $\fG^0_\Gamma$ (resp., $\fG^0_{\Gamma^\vee(1)}$) denote the trivial $\Gamma$-gerbe 
(resp., $\Gamma^\vee(1)$)-gerbe on $X$. As a particular case of \eqref{e:fib glob sect 1},
we obtain an equivalence
\begin{equation} \label{e:fib glob sect neutral}
\ul\bC_{\Gamma^\vee(1)}|_{\fG^0_{\Gamma^\vee(1)}}\simeq \bC_\Gamma,
\end{equation}
and as a particular case of \eqref{e:fib glob sect 2} we obtain an equivalence
\begin{equation} \label{e:glob fib sect neutral}
\ul\bC_{\Gamma}|_{\fG^0_{\Gamma}}\simeq \bC_{\Gamma^\vee(1)}.
\end{equation}

\sssec{}

Note that $\bC_\Gamma$ is a category acted on by $\QCoh(\on{Ge}_{\Gamma}(X))$. For 
$$\alpha \in \Gamma(-1)\simeq H^2_{\on{et}}(X,\Gamma)\simeq \pi_0(\on{Ge}_{\Gamma}(X))),$$
consider the corresponding idempotent
$$\CO_{\on{Ge}_{\Gamma}(X),\alpha}\in \QCoh(\on{Ge}_{\Gamma}(X))$$
as acting on $\bC_\Gamma$. 

\medskip

Note also that $\Gamma^\vee(1)$ acts by automorphisms of the identity functor on $\fG^0_{\Gamma^\vee(1)}$, and hence
also by the automorphisms of the identity functor of $\ul\bC_{\Gamma^\vee(1)}|_{\fG^0_{\Gamma^\vee(1)}}$. For
$$\alpha\in \Gamma(-1)\simeq (\Gamma^\vee(1))^\vee$$ 
let $\sP_\alpha$ denote the corresponding idempotent on $\ul\bC_{\Gamma^\vee(1)}|_{\fG^0_{\Gamma^\vee(1)}}$.

\medskip

The following is a particular case of \eqref{l:Hecke 0 abs gen}: 

\begin{lem} \label{l:Hecke 0 abs}
Under the identification \eqref{e:fib glob sect neutral}, the action of $\CO_{\on{Ge}_{\Gamma}(X),\alpha}$ on $\bC_\Gamma$
corresponds to the action of
$\sP_\alpha$ on $\ul\bC_{\Gamma^\vee(1)}|_{\fG^0_{\Gamma^\vee(1)}}$. 
\end{lem} 

\sssec{}

As in \secref{sss:loop gerbes}, a point 
$$\CP_{\Gamma^\vee(1)}\in \Bun_{\Gamma^\vee(1)}$$
gives rise to a line bundle denoted $\CL_{\CP_{\Gamma^\vee(1)}}$ on $\on{Ge}_{\Gamma}(X)$. 

\medskip

In particular, we consider the endofunctor 
$$\CL_{\CP_{\Gamma^\vee(1)}}\otimes (-)$$
of  $\bC_\Gamma$.

\medskip

We can view $\CP_{\Gamma^\vee(1)}$ itself as an automorphism of $\fG^0_{\Gamma^\vee(1)}$. And as such, it induces an 
autoequivalence of $\ul\bC_{\Gamma^\vee(1)}|_{\fG^0_{\Gamma^\vee(1)}}$.

\medskip

The followings is a particular case of \lemref{l:Hecke 1 abs gen}:

\begin{lem} \label{l:Hecke 1 abs}
Under the identification \eqref{e:fib glob sect neutral}, the action of $\CL^{\otimes -1}_{\CP_{\Gamma^\vee(1)}}$ on $\bC_\Gamma$
corresponds to the action of $\CP_{\Gamma^\vee(1)}$ on $\ul\bC_{\Gamma^\vee(1)}|_{\fG^0_{\Gamma^\vee(1)}}$.
\end{lem} 

\sssec{}

For $\fG_{\Gamma^\vee(1)}\in \on{Ge}_{\Gamma^\vee(1)}(X)$, denote by  
$$\on{ev}|_{\fG_{\Gamma^\vee(1)}}:\bC_{\Gamma^\vee(1)}\to (\ul\bC_{\Gamma^\vee(1)})|_{\fG_{\Gamma^\vee(1)}}$$
the canonical evaluation functor
$$\on{ev}|_{\fG_{\Gamma^\vee(1)}}:\bC_{\Gamma^\vee(1)}\to (\ul\bC_{\Gamma^\vee(1)})|_{\fG_{\Gamma^\vee(1)}}.$$

\medskip

More generally, for a $\BG_m$-gerbe $\CG$ on $\on{Ge}_{\Gamma^\vee(1)}(X)$, we have the evaluation functor
$$\on{ev}_\CG|_{\fG_{\Gamma^\vee(1)}}:\bGamma\left(\on{Ge}_{\Gamma^\vee(1)}(X),(\ul\bC_{\Gamma^\vee(1)})_\CG\right)\to
(\ul\bC_{\Gamma^\vee(1)})_\CG|_{\fG_{\Gamma^\vee(1)}}.$$

\sssec{}

For a given $\fG_{\Gamma^\vee(1)}\in \on{Ge}_{\Gamma^\vee(1)}(X)$, let 
\begin{equation} \label{e:insert 1}
(\on{ev}_{\CG_{\fG_{\Gamma^\vee(1)}}}|_{\fG^0_{\Gamma}})^L:\ul\bC_{\Gamma}|_{\fG^0_{\Gamma}}\to 
\bGamma\left(\on{Ge}_{\Gamma}(X),(\ul\bC_{\Gamma})_{\CG_{\fG_{\Gamma^\vee(1)}}}\right)
\end{equation}
be the left adjoint of the evaluation functor
$$\bGamma\left(\on{Ge}_{\Gamma}(X),(\ul\bC_{\Gamma})_{\CG_{\fG_{\Gamma^\vee(1)}}}\right)
\overset{\on{ev}_{\CG_{\fG_{\Gamma^\vee(1)}}}|_{\fG^0_{\Gamma}}}\longrightarrow 
(\ul\bC_{\Gamma})_{\CG_{\fG_{\Gamma^\vee(1)}}}|_{\fG^0_{\Gamma}}\simeq \ul\bC_{\Gamma}|_{\fG^0_{\Gamma}}.$$

\sssec{}

As a particular case of \lemref{l:fiber and global sects gen}, we obtain: 

\begin{lem} \label{l:fiber and global sects}
The following diagram commutes:
$$
\CD
\bC_{\Gamma^\vee(1)} @>{\on{ev}|_{\fG_{\Gamma^\vee(1)}}}>> (\ul\bC_{\Gamma^\vee(1)})|_{\fG_{\Gamma^\vee(1)}} \\
@A{\text{\eqref{e:glob fib sect neutral}}}A{\sim}A @A{\sim}A{\text{\eqref{e:fib glob sect 1}}}A \\
\ul\bC_{\Gamma}|_{\fG^0_{\Gamma}} @>{\text{\eqref{e:insert 1}}}>> \bGamma(\on{Ge}_{\Gamma}(X),(\ul\bC_{\Gamma})_{\CG_{\fG_{\Gamma^\vee(1)}}}).
\endCD
$$
\end{lem}

\ssec{Example: the usual Fourier-Mukai transform}

This subsection is included in order to illustrate how the 2-Fourier-Mukai transform works; this material 
is not needed in the rest of the paper.

\sssec{}

Let 
$$1\to \Gamma\to T_1\to T\to 1$$
be an isogeny of tori. Consider the dual isogeny 
$$1\to \Gamma^\vee(1)\to T^\vee\to T^\vee_1\to 1.$$

Consider the corresponding maps
$$\sfp_{\Gamma}:\Bun_T\to \on{Ge}_{\Gamma}(X) \text{ and }
\sfp_{\Gamma^\vee(1)}:\Bun_{T^\vee_1}\to \on{Ge}_{\Gamma^\vee(1)}(X).$$

\sssec{}

Consider the unit sheaf of categories 
$$\ul\QCoh(\Bun_T)$$ 
over $\Bun_T$, and let
$$\ul\QCoh^\Gamma(\Bun_T):=(\sfp_{\Gamma})_*(\ul\QCoh(\Bun_T))$$
be its direct image along $\sfp_{\Gamma}$, viewed as a sheaf of categories over $\on{Ge}_{\Gamma}(X)$.

\medskip 

Similarly, consider
$$\ul\QCoh^{\Gamma^\vee(1)}(\Bun_{T^\vee_1}):=(\sfp_{\Gamma^\vee(1)})_*(\ul\QCoh(\Bun_{T^\vee_1}))$$
as a sheaf of categories over $\on{Ge}_{\Gamma^\vee(1)}(X)$.

\sssec{}

We claim: 

\begin{lem}
With respect to the equivalence $\on{2-FM}_{\on{Ge}_{\Gamma}(X)\to \on{Ge}_{\Gamma^\vee(1)}(X)}$, the objects
$$\ul\QCoh^\Gamma(\Bun_T) \text{ and } \ul\QCoh^{\Gamma^\vee(1)}(\Bun_{T^\vee_1})$$
correspond to one another.
\end{lem}

\begin{proof}[Sketch of proof]

The proof follows from the usual properties of the usual Fourier-Mukai equivalences
$$\QCoh(\Bun_T)\overset{\on{FM}}\simeq \QCoh(\Bun_{T^\vee})  \text{ and }
\QCoh(\Bun_{T_1})\overset{\on{FM}}\simeq \QCoh(\Bun_{T_1^\vee}),$$
combined with the following observation:

\medskip

The composition
$$\Bun_T \times \Bun_{T^\vee_1}\to \on{Ge}_{\Gamma}(X)\times \on{Ge}_{\Gamma^\vee(1)}(X) \overset{\text{\eqref{e:Verdier pairing Gamma}}}\longrightarrow
\on{Ge}_{\BG_m}(\on{pt})$$
is canonically trivial. Furthermore, the resulting map
$$\Bun_{T_1} \times \Bun_{T^\vee_1}\to \Bun_T \times \Bun_{T^\vee_1} \to \on{pt}/\BG_m$$
equals the Weil pairing.

\end{proof}

\ssec{Two sheaves associated with automorphic category}

Let $G$ be a semi-simple group and consider the category
$$\Dmod_{\frac{1}{2}}(\Bun_G).$$

We will upgrade it to two objects
$$\ul{\Dmod}^{Z_G}_{\frac{1}{2}}(\Bun_{G_{\on{ad}}})\in \on{ShvCat}(\on{Ge}_{Z_G}(X))
\text{ and }  \ul{\Dmod}^{\pi_1(\cG)}_{\frac{1}{2}}(\Bun_G)\in \on{ShvCat}(\on{Ge}_{\pi_1(\cG)}(X)).$$

We will state \thmref{t:FM L} that says that the above two objects correspond to one another under
the 2-Fourier-Mukai transform. The nature of the two constructions will 
then immediately yield Theorems \ref{t:Hecke Z 0} and \ref{t:Hecke Z 1}. 

\sssec{}

The short exact sequence of groups
$$1\to Z_G\to G \to G_{\on{ad}}\to 1$$
gives rise to a map
\begin{equation} 
\sfp_{Z_G}:\Bun_{G_{\on{ad}}}\to \on{Ge}_{Z_G}(X).
\end{equation}

Consider the induced map
$$(\sfp_{Z_G})_\dr:(\Bun_{G_{\on{ad}}})_\dr\to (\on{Ge}_{Z_G}(X))_\dr.$$

\sssec{}

Note, however, that (by nil-invariance of \'etale cohomology) for a finite abelian group $\Gamma$,  the map of prestacks 
$$\on{Ge}_{\Gamma}(X) \to (\on{Ge}_{\Gamma}(X))_{\dR}$$
is an isomorphism. 

\medskip

Hence, we can regard $(\sfp_{Z_G})_\dr$ as a map
\begin{equation} \label{e:to pi Ge}
(\Bun_{G_{\on{ad}}})_\dr\to \on{Ge}_{Z_G}(X).
\end{equation}

\sssec{}

The category $\Dmod_{\frac{1}{2}}(\Bun_{G_{\on{ad}}})$ is (tautologically) the category of global sections of a sheaf
of categories, denoted
$$\ul{\Dmod}^{G_{\on{ad}}}_{\frac{1}{2}}(\Bun_{G_{\on{ad}}})$$
over $(\Bun_{G_{\on{ad}}})_\dr$. 

\medskip

Set
$$\ul{\Dmod}^{Z_G}_{\frac{1}{2}}(\Bun_{G_{\on{ad}}}):=((\sfp_{Z_G})_\dr)_*(\ul{\Dmod}^{G_{\on{ad}}}_{\frac{1}{2}}(\Bun_{G_{\on{ad}}}))\in \on{ShvCat}(\on{Ge}_{Z_G}(X)).$$

\sssec{}

Tautologically, we have
\begin{equation} \label{e:glob sect BunG sheaf}
\bGamma\left(\on{Ge}_{Z_G}(X),\ul{\Dmod}^{Z_G}_{\frac{1}{2}}(\Bun_{G_{\on{ad}}})\right)\simeq
\Dmod_{\frac{1}{2}}(\Bun_{G_{\on{ad}}}).
\end{equation} 

In addition, 
\begin{equation} \label{e:fiber BunG sheaf}
\left(\ul{\Dmod}^{Z_G}_{\frac{1}{2}}(\Bun_{G_{\on{ad}}})\right)|_{\fG^0_{Z_G}}\simeq \Dmod_{\frac{1}{2}}(\Bun_G),
\end{equation} 
see the notations in \secref{sss:fibers of 2-FM}. 

%
%
%
%

\medskip

More generally, for a given $\fG_{Z_G}\in \on{Ge}_{Z_G}(X)$, we have
\begin{equation} \label{e:fiber BunG sheaf twisted}
\left(\ul{\Dmod}^{Z_G}_{\frac{1}{2}}(\Bun_{G_{\on{ad}}})\right)|_{\fG_{Z_G}}\simeq \Dmod_{\frac{1}{2}}(\Bun_{G,\fG_{Z_G}}),
\end{equation} 
where 
$$\Bun_{G,\fG_{Z_G}}:=\Bun_{G_{\on{ad}}}\underset{\on{Ge}_{Z_G}(X)}\times \{\fG_{Z_G}\}.$$

\begin{rem} \label{r:twisted groups}

We can think $\Bun_{G,\fG_{Z_G}}$ as follows: pick a $G_{\on{ad}}$-torsor $\CP_{G_{\on{ad}}}$ that maps to
$\fG_{Z_G}$, and let 
$$G_{\CP_{G_{\on{ad}}}}$$
be the corresponding (non-pure) inner form of the constant group-scheme with fiber $G$.

\medskip

Then
$$\Bun_{G,\fG_{Z_G}}\simeq \Bun_{G_{\CP_{G_{\on{ad}}}}}.$$

(Note that different choices for $\CP_{G_{\on{ad}}}$ differ by $G$-torsors, and hence the corresponding moduli spaces
$\Bun_{G_{\CP_{G_{\on{ad}}}}}$ are a priori canonically isomorphic.) 

\end{rem}

\sssec{}

We now consider $\Dmod_{\frac{1}{2}}(\Bun_G)$ as equipped with the spectral action of $\QCoh(\LS_\cG)$. Since the stack $\LS_\cG$ 
is 1-affine (say, by \cite[Theorem 2.2.4]{Ga3}), we can canonically attach to  $\Dmod_{\frac{1}{2}}(\Bun_G)$ an object
\begin{equation} \label{e:1-aff upgrade}
\ul{\Dmod}^{\cG}_{\frac{1}{2}}(\Bun_G)\in \on{ShvCat}(\LS_\cG), 
\end{equation} 

\medskip

The short exact sequence of groups
$$1\to \pi_1(\cG)\to \cG_{\on{sc}}\to \cG\to 1$$
give rise to a map 
\begin{equation} \label{e:LS to gerbe}
\sfp_{\pi_1(\cG)}:\LS_\cG\to \on{Ge}_{\pi_1(\cG)}(X),
\end{equation} 
where we note that $\cG_{\on{sc}}$ is the Langlands dual of $G_{\on{ad}}$.

\medskip

Denote
$$\ul{\Dmod}^{\pi_1(\cG)}_{\frac{1}{2}}(\Bun_G):=(\sfp_{\pi_1(\cG)})_*(\ul{\Dmod}^{\cG}_{\frac{1}{2}}(\Bun_G))\in 
\on{ShvCat}(\on{Ge}_{\pi_1(\cG)}(X)).$$

\medskip

Note that tautologically,
\begin{equation} \label{e:glob sect BunG sheaf dual}
\bGamma\left(\on{Ge}_{\pi_1(\cG)}(X),\ul{\Dmod}^{\pi_1(\cG)}_{\frac{1}{2}}(\Bun_G)\right) \simeq
\Dmod_{\frac{1}{2}}(\Bun_G).
\end{equation} 

\sssec{}

We will prove:

\begin{thm} \label{t:FM L}
Under the identification $\pi_1(\cG)\simeq (Z_G)^\vee(1)$, we have
$$\on{2-FM}_{\on{Ge}_{Z_G}(X)\to \on{Ge}_{\pi_1(\cG)}(X)}\left(\ul{\Dmod}^{Z_G}_{\frac{1}{2}}(\Bun_{G_{\on{ad}}})\right)\simeq
\ul{\Dmod}^{\pi_1(\cG)}_{\frac{1}{2}}(\Bun_G),$$
up to the inversion involution\footnote{The inversion involution has to do with our normalization of the geometric Satake equivalence.}  
on $Z_G$.
\end{thm} 

Combining \thmref{t:FM L} with \corref{c:Verdier FM invol}, we obtain: 

\begin{cor} \label{c:FM L}
Under the identification $Z_G\simeq (\pi_1(\cG)^\vee)(1)$, we have:
$$\on{2-FM}_{\on{Ge}_{\pi_1(\cG)}(X)\to \on{Ge}_{Z_G}(X)}\left(\ul{\Dmod}^{\pi_1(\cG)}_{\frac{1}{2}}(\Bun_G)\right)
\simeq \ul{\Dmod}^{Z_G}_{\frac{1}{2}}(\Bun_{G_{\on{ad}}}).$$
\end{cor}

\sssec{} \label{sss:proof central Hecke}

Let us show how \corref{c:FM L} implies Theorems \ref{t:Hecke Z 0} and  \ref{t:Hecke Z 1}. 

\medskip

Indeed, the two theorems follow immediately from Lemma \ref{l:Hecke 0 abs} and \ref{l:Hecke 1 abs}, respectively. 

\qed[Theorems \ref{t:Hecke Z 0} and  \ref{t:Hecke Z 1}]

\ssec{Proof of \thmref{t:FM L}}

\sssec{}

We start by constructing a functor 
\begin{equation} \label{e:FM L 1}
\ul{\Dmod}^{\pi_1(\cG)}_{\frac{1}{2}}(\Bun_G)\to \on{2-FM}_{\on{Ge}_{Z_G}(X)\to \on{Ge}_{\pi_1(\cG)}(X)}\left(\ul{\Dmod}^{Z_G}_{\frac{1}{2}}(\Bun_{G_{\on{ad}}})\right),
\end{equation}
up to the inversion involution.

\medskip

Since $\on{Ge}_{\pi_1(\cG)}(X)$ is algebro-geometrically discrete, the datum of \eqref{e:FM L 1} consists of the data of functors 
\begin{equation} \label{e:FM L 1 fib}
\ul{\Dmod}^{\pi_1(\cG)}_{\frac{1}{2}}(\Bun_G)|_{\fG^{\otimes -1}_{\pi_1(\cG)}}
\to \on{2-FM}_{\on{Ge}_{Z_G}(X)\to \on{Ge}_{\pi_1(\cG)}(X)}\left(\ul{\Dmod}^{Z_G}_{\frac{1}{2}}(\Bun_{G_{\on{ad}}})\right)|_{\fG_{\pi_1(\cG)}}
\end{equation}
that depend functorially on $\fG_{\pi_1(\cG)}\in \on{Ge}_{\pi_1(\cG)}(X)$. 

\medskip

Applying \eqref{e:fib glob sect 1}, the datum of \eqref{e:FM L 1 fib} is equivalent to that of a functor
\begin{equation} \label{e:FM L 1 fib glob}
\ul{\Dmod}^{\pi_1(\cG)}_{\frac{1}{2}}(\Bun_G)|_{\fG^{\otimes -1}_{\pi_1(\cG)}}\to 
\bGamma\left(\on{Ge}_{Z_G}(X),\ul{\Dmod}^{Z_G}_{\frac{1}{2}}(\Bun_{G_{\on{ad}}})_{\CG_{\fG_{\pi_1(\cG)}}}\right).
\end{equation}

We rewrite the right-hand side in \eqref{e:FM L 1 fib glob} as
\begin{equation} \label{e:FM L 1 fib glob RHS}
\Dmod_{\frac{1}{2}}(\Bun_{G_{\on{ad}}})_{\CG_{\fG_{\pi_1(\cG)}}},
\end{equation} 
where by a slight abuse of notation we regard $\CG_{\fG_{\pi_1(\cG)}}$ as a $\BG_m$-gerbe on $(\Bun_{G_{\on{ad}}})_\dr$. 

\sssec{}

We rewrite the left-hand side in \eqref{e:FM L 1 fib glob} as
\begin{equation} \label{e:FM L 1 fib glob LHS}
\Dmod_{\frac{1}{2}}(\Bun_G)\underset{\QCoh(\LS_\cG)}\otimes \QCoh(\LS_{\cG_{\on{sc}},\fG^{\otimes -1}_{\pi_1(\cG)}}),
\end{equation}
where  
$$\LS_{\cG_{\on{sc}},\fG^{\otimes -1}_{\pi_1(\cG)}}:=
\LS_\cG\underset{\on{Ge}_{\pi_1(\cG)}(X)}\times \{\fG^{\otimes -1}_{\pi_1(\cG)}\}.$$ 

The notation $\LS_{\cG_{\on{sc}},\fG^{\otimes -1}_{\pi_1(\cG)}}$ expresses the fact that this stack is a twisted form of $\LS_{\cG_{\on{sc}}}$.
Namely, for $\fG_{\pi_1(\cG)}=\fG^0_{\pi_1(\cG)}$ we have
$$\LS_{\cG_{\on{sc}},\fG^0_{\pi_1(\cG)}}=\LS_{\cG_{\on{sc}}}.$$

\sssec{}

Let $\Rep(\cG_{\on{sc}})_{\Ran,\fG^{\otimes -1}_{\pi_1(\cG)}}$ be the $\fG^{\otimes -1}_{\pi_1(\cG)}$-twist of $\Rep(\cG_{\on{sc}})$, i.e., this is 
the factorization category that associates to a point $\ul{x}\in \Ran$ the category
$$\Rep(\cG_{\on{sc}})_{\ul{x},\fG^{\otimes -1}_{\pi_1(\cG)}}:=\QCoh(\LS^\reg_{\cG_{\on{sc}},\fG^{\otimes -1}_{\pi_1(\cG)},\ul{x}}),$$
where:

\begin{itemize}

\item $\LS^\reg_{\cG_{\on{sc}},\fG^{\otimes -1}_{\pi_1(\cG)},\ul{x}}:=
\LS^\reg_{\cG,\ul{x}}\underset{\on{Ge}_{\pi_1(\cG)}(\cD_{\ul{x}})}\times \{\fG^{\otimes -1}_{\pi_1(\cG)}|_{\cD_{\ul{x}}}\}$;

\medskip

\item $\on{Ge}_{\pi_1(\cG)}(\cD_{\ul{x}})$ denotes the space of $\pi_1(\cG)$-gerbes on the formal disc $\cD_{\ul{x}}$ around $\ul{x}$.

\end{itemize} 

\sssec{}

Recall now that we have the (symmetric) monoidal localization functors
$$\Loc^{\on{spec}}_\cG:\Rep(\cG)_\Ran\to \QCoh(\LS_\cG) \text{ and }
\Loc^{\on{spec}}_{\cG_{\on{sc}}}:\Rep(\cG_{\on{sc}})_\Ran\to \QCoh(\LS_{\cG_{\on{sc}}}).$$

We have the corresponding twisted version:
$$\Loc^{\on{spec}}_{\cG_{\on{sc}},\fG^{\otimes -1}_{\pi_1(\cG)}}:\Rep(\cG_{\on{sc}})_{\Ran,\fG^{\otimes -1}_{\pi_1(\cG)}}\to \QCoh(\LS_{\cG_{\on{sc}},\fG^{\otimes -1}_{\pi_1(\cG)}}),$$
and as in the case of $\Loc^{\on{spec}}_{\cG_{\on{sc}}}$, one shows that the functor $\Loc^{\on{spec}}_{\cG_{\on{sc}},\fG^{\otimes -1}_{\pi_1(\cG)}}$
is a \emph{localization}, i.e., its right adjoint is fully faithful. 
See \cite[Prop. C.1.7]{GLC4} for a general result encompassing
these statements.

\sssec{}

Let $\phi$ denote the tautological map $G\to G_{\on{ad}}$, and also the map
$$\Bun_G\to \Bun_{G_{\on{ad}}}.$$

Note that the pullback along $\phi$ of $\CG_{\fG_{\pi_1(\cG)}}$, viewed as a gerbe on $(\Bun_{G_{\on{ad}}})_\dr$
canonically trivializes. Hence, we obtain a functor
\begin{equation} \label{e:from G to Gad twisted}
\phi_!:\Dmod_{\frac{1}{2}}(\Bun_G) \to \Dmod_{\frac{1}{2}}(\Bun_{G_{\on{ad}}})_{\CG_{\fG_{\pi_1(\cG)}}}.
\end{equation} 

\sssec{}

We will construct the following gerbe-twisted version of the Hecke action: 

\begin{prop}  \label{p:key local} 
There is a canonically defined action of the monoidal category 
$\Rep(\cG_{\on{sc}})_{\Ran,\fG^{\otimes -1}_{\pi_1(\cG)}}$ on $\Dmod_{\frac{1}{2}}(\Bun_{G_{\on{ad}}})_{\CG_{\fG_{\pi_1(\cG)}}}$.
Moreover, the following properties hold: 

\smallskip

\noindent{\em(a)} 
The functor $\phi_!$ of \eqref{e:from G to Gad twisted} is equivariant with respect to the $\Rep(\cG)_\Ran$-action on $\Dmod_{\frac{1}{2}}(\Bun_G)$ via
the tautological functor 
$$\Rep(\cG)_\Ran\to \Rep(\cG_{\on{sc}})_{\Ran,\fG^{\otimes -1}_{\pi_1(\cG)}}.$$

\smallskip

\noindent{\em(b)} The resulting $\Rep(\cG_{\on{sc}})_{\Ran,\fG^{\otimes -1}_{\pi_1(\cG)}}$-action on $\Dmod_{\frac{1}{2}}(\Bun_{G_{\on{ad}}})_{\CG_{\fG_{\pi_1(\cG)}}}$
factors canonically via the localization functor
$$\Loc^{\on{spec}}_{\cG_{\on{sc}},\fG^{\otimes -1}_{\pi_1(\cG)}}:\Rep(\cG_{\on{sc}})_{\Ran,\fG^{\otimes -1}_{\pi_1(\cG)}}\to 
\QCoh(\LS_{\cG_{\on{sc}},\fG^{\otimes -1}_{\pi_1(\cG)}}).$$

\end{prop} 

The proposition will be proved in Sects. \ref{ss:proof key local a} and \ref{ss:proof key local b}. 

\sssec{}

We now return to the sought-for functor \eqref{e:FM L 1 fib glob}, thought of as a functor
\begin{equation} \label{e:FM L LHS and RHS}
\Dmod_{\frac{1}{2}}(\Bun_G)\underset{\QCoh(\LS_\cG)}\otimes \QCoh(\LS_{\cG_{\on{sc}},\fG^{\otimes -1}_{\pi_1(\cG)}})\to
(\Bun_{G_{\on{ad}}})_{\CG_{\fG_{\pi_1(\cG)}}}.
\end{equation}

Its construction follows from \propref{p:key local} by considering the adjoint pair
$$\QCoh(\LS_\cG)\mmod \rightleftarrows \QCoh(\LS_{\cG_{\on{sc}},\fG^{\otimes -1}_{\pi_1(\cG)}})\mmod,$$
where the right adjoint is the forgeftul functor, and the left adjoint is the tensoring up functor
with respect to
$$\QCoh(\LS_\cG)\to \QCoh(\LS_{\cG_{\on{sc}},\fG^{\otimes -1}_{\pi_1(\cG)}}).$$

\sssec{}

The next assertion results from \lemref{l:fiber and global sects}: 

\begin{lem} \label{l:L FM glob sect}
The functor
\begin{multline} \label{e:FM L glob}
\bGamma\left(\on{Ge}_{\pi_1(\cG)}(X),\ul{\Dmod}^{\pi_1(\cG)}_{\frac{1}{2}}(\Bun_G)\right)\to \\
\to \bGamma\left(\on{Ge}_{\pi_1(\cG)}(X),\on{2-FM}_{\on{Ge}_{Z_G}(X)\to \on{Ge}_{\pi_1(\cG)}(X)}\left(\ul{\Dmod}^{Z_G}_{\frac{1}{2}}(\Bun_{G_{\on{ad}}})\right)\right)
\end{multline}
induced by \eqref{e:FM L 1}, identifies via
$$\bGamma\left(\on{Ge}_{\pi_1(\cG)}(X),\ul{\Dmod}^{\pi_1(\cG)}_{\frac{1}{2}}(\Bun_G)\right) \simeq 
\Dmod_{\frac{1}{2}}(\Bun_G)$$
and
\begin{multline}
\bGamma\left(\on{Ge}_{\pi_1(\cG)}(X),\on{2-FM}_{\on{Ge}_{Z_G}(X)\to \on{Ge}_{\pi_1(\cG)}(X)}\left(\ul{\Dmod}^{Z_G}_{\frac{1}{2}}(\Bun_{G_{\on{ad}}})\right)\right)
\overset{\text{\eqref{e:fib glob sect 2}}}\simeq \\
\simeq \left(\ul{\Dmod}^{Z_G}_{\frac{1}{2}}(\Bun_{G_{\on{ad}}})\right)|_{\fG^0_{Z_G}}\simeq \Dmod_{\frac{1}{2}}(\Bun_G)
\end{multline}
with the identity endofunctor on $\Dmod_{\frac{1}{2}}(\Bun_G)$.
\end{lem}

\sssec{}

We are now ready to prove that the map \eqref{e:FM L 1} is an equivalence. 

\medskip

Note that by point (b) of \propref{p:key local}, for every $\fG_{\pi_1(\cG)}\in \on{Ge}_{\pi_1(\cG)}(X)$, 
the category 
$$\on{2-FM}_{\on{Ge}_{Z_G}(X)\to \on{Ge}_{\pi_1(\cG)}(X)}\left(\ul{\Dmod}^{Z_G}_{\frac{1}{2}}(\Bun_{G_{\on{ad}}})\right)|_{\fG^{\otimes -1}_{\pi_1(\cG)}}\simeq
\Dmod_{\frac{1}{2}}(\Bun_{G_{\on{ad}}})_{\CG_{\fG_{\pi_1(\cG)}}}$$
is a module 
over $\QCoh(\LS_{\cG_{\on{sc}},\fG^{\otimes -1}_{\pi_1(\cG)}})$, while $\LS_{\cG_{\on{sc}},\fG^{\otimes -1}_{\pi_1(\cG)}}$ is 1-affine. Hence, the object
$$\on{2-FM}_{\on{Ge}_{Z_G}(X)\to \on{Ge}_{\pi_1(\cG)}(X)}\left(\ul{\Dmod}^{Z_G}_{\frac{1}{2}}(\Bun_{G_{\on{ad}}})\right)\in 
\on{ShvCat}(\on{Ge}_{\pi_1(\cG)}(X))$$
is canonically of the form 
$$(\sfp_{\pi_1(\cG)})_*\left(\ul{\Dmod}_{\frac{1}{2}}(\Bun_{G_{\on{ad}}})_{\CG_{\on{univ}}}\right),$$
for an object
$$\ul{\Dmod}_{\frac{1}{2}}(\Bun_{G_{\on{ad}}})_{\CG_{\on{univ}}}\in \on{ShvCat}(\LS_\cG).$$

Moreover, by construction, the map \eqref{e:FM L 1} comes from a map
\begin{equation} \label{e:FM L 1 upstairs}
\ul{\Dmod}^\cG_{\frac{1}{2}}(\Bun_G)\to \ul{\Dmod}_{\frac{1}{2}}(\Bun_{G_{\on{ad}}})_{\CG_{\on{univ}}}
\end{equation} 
in $\on{ShvCat}(\LS_\cG)$, where $\ul{\Dmod}^\cG_{\frac{1}{2}}(\Bun_G)$ is as in \eqref{e:1-aff upgrade}. 

\medskip

It is sufficient to show that the map \eqref{e:FM L 1 upstairs} is an equivalence. However, the stack $\LS_\cG$
is 1-affine, and hence, the functor
$$\bGamma(\LS_\cG,-): \on{ShvCat}(\LS_\cG)\to \DGCat$$
is conservative. 

\medskip

Hence, it is sufficient to show that the resulting functor
$$\bGamma(\LS_\cG,\ul{\Dmod}^\cG_{\frac{1}{2}}(\Bun_G))\to  \bGamma(\LS_\cG,\ul{\Dmod}_{\frac{1}{2}}(\Bun_{G_{\on{ad}}})_{\CG_{\on{univ}}})$$
is an equivalence. 

\medskip

However, the latter functor identifies with the functor \eqref{e:FM L glob}, and hence is an equivalence by \lemref{l:L FM glob sect}. 

\qed[\thmref{t:FM L}]

%
%
%
%
%
%

\ssec{Proof of \propref{p:key local}(a)} \label{ss:proof key local a} 

\sssec{}

Let us return to the setting of \secref{ss:FM and Verdier}. Let $\ul{x}$ be a point of $\Ran$. Consider the spaces
$$\on{Ge}_\Gamma(X)_{\ul{x}}:=\on{Fib}(\on{Ge}_\Gamma(X)\to \on{Ge}_\Gamma(X-\ul{x}))$$
and 
$$\on{Ge}_\Gamma(\cD_{\ul{x}})_{\ul{x}}:=
\on{Fib}(\on{Ge}_\Gamma(\cD_{\ul{x}})\to \on{Ge}_\Gamma(\ocD_{\ul{x}})),$$
where 
$$\ocD_{\ul{x}}:=\cD_{\ul{x}}-\ul{x}.$$

\medskip

Restriction along $\cD_{\ul{x}}\to X$ defines an \emph{isomorphism} 
$$\on{Ge}_\Gamma(X)_{\ul{x}} \to \on{Ge}_\Gamma(\cD_{\ul{x}})_{\ul{x}}.$$

\begin{rem} 

Here is an explicit description of the spaces 
$\on{Ge}_\Gamma(\cD_{\ul{x}})$,  $\on{Ge}_\Gamma(\cD^\times_{\ul{x}})$ and 
$\on{Ge}_\Gamma(\cD_{\ul{x}})_{\ul{x}}$:

\medskip

Write $\Gamma$ as the kernel of a homomorphism of two tori $T_0\to T_1$. Then
$$\on{Ge}_\Gamma(\cD_{\ul{x}})=B^2(\on{ker}(\fL^+(T_0)_{\ul{x}}\to \fL^+(T_1)_{\ul{x}}))_{\on{et}}.$$

When $\ul{x}$ is a singleton, the above space is just $\on{Ge}_\Gamma(\on{pt})$. 

\medskip

Further,
$$\on{Ge}_\Gamma(\cD^\times_{\ul{x}})\simeq B^2(\on{Fib}(\fL(T_0)_{\ul{x}}\to \fL(T_1)_{\ul{x}}))_{\on{et}}\simeq 
B^1(\on{coFib}(\fL(T_0)_{\ul{x}}\to \fL(T_1)_{\ul{x}}))_{\on{et}}.$$

Finally, 
$$\on{Ge}_\Gamma(\cD_{\ul{x}})_{\ul{x}}\simeq \Gr_{T_1,\ul{x}}/\Gr_{T_0,\ul{x}}.$$

In particular, we obtain that for $\ul{x}=x$ being a singleton, we have: 
\begin{equation} \label{e:Gamma gr}
\on{Ge}_\Gamma(\cD_{\ul{x}})_{\ul{x}} \simeq \Gamma(-1),
\end{equation}
where the right-hand side is viewed as an algebro-geometrically discrete finite set. 

\end{rem} 

\sssec{} \label{sss:local Verdier}

A local variant of \eqref{e:Verdier pairing Gamma} is a pairing
\begin{equation} \label{e:Verdier pairing Gamma loc}
\on{Ge}_\Gamma(\cD_{\ul{x}})_{\ul{x}} \times \on{Ge}_{\Gamma^\vee(1)}(\cD_{\ul{x}})\to \on{Ge}_{\BG_m}(\on{pt}).
\end{equation}

\medskip

In particular, to a $\Gamma^\vee(1)$-gerbe $\fG^{\on{loc}}_{\Gamma^\vee(1)}$ on $\cD_{\ul{x}}$ we can canonically 
associate a $\BG_m$-gerbe $\CG_{\fG^{\on{loc}}_{\Gamma^\vee(1)}}$ on $\on{Ge}_\Gamma(\cD_{\ul{x}})_{\ul{x}}$.

\begin{rem}
Although we do not need it in what follows, we remark that the 
pairing \eqref{e:Verdier pairing Gamma loc} is also of Fourier-Mukai type. When $\ul{x}=x$ is a singleton, this follows immediately from 
the identifications
$$\on{Ge}_\Gamma(\cD_{\ul{x}})_{\ul{x}} \simeq \Gamma(-1) \text{ and } \on{Ge}_{\Gamma^\vee(1)}(\cD_{\ul{x}})\simeq \Gamma^\vee(1).$$
\end{rem} 

\sssec{}

Note that we have the following commutative diagram of pairings
\begin{equation} \label{e:loc vs global pairings}
\CD
\on{Ge}_\Gamma(X)_{\ul{x}} \times \on{Ge}_{\Gamma^\vee(1)}(X) @>>> \on{Ge}_\Gamma(X) \times \on{Ge}_{\Gamma^\vee(1)}(X) \\
@A{\simeq}AA  \\
\on{Ge}_\Gamma(\cD_{\ul{x}})_{\ul{x}} \times \on{Ge}_{\Gamma^\vee(1)}(X)  & & @VV{\text{\eqref{e:Verdier pairing Gamma}}}V \\
@VVV \\
\on{Ge}_\Gamma(\cD_{\ul{x}})_{\ul{x}}   \times \on{Ge}_{\Gamma^\vee(1)}(\cD_{\ul{x}})  
@>{\text{\eqref{e:Verdier pairing Gamma loc}}}>> \on{Ge}_{\BG_m}(\on{pt}).
\endCD
\end{equation} 

In particular, for a $\Gamma^\vee(1)$-gerbe $\fG_{\Gamma^\vee(1)}$ on $X$ and
$$\fG^{\on{loc}}_{\Gamma^\vee(1)}:=\fG_{\Gamma^\vee(1)}|_{\cD_{\ul{x}}}$$
we have
$$\CG_{\fG^{\on{loc}}_{\Gamma^\vee(1)}}|_{\on{Ge}_\Gamma(X)_{\ul{x}}}\simeq \CG_{\fG_{\Gamma^\vee(1)}}|_{\on{Ge}_\Gamma(X)_{\ul{x}}}.$$

\sssec{} \label{sss:loc Hecke to gerbes}

Let 
$$\on{Hecke}^{\on{loc}}_{G_{\on{ad}},\ul{x}}:=\fL^+(G_{\on{ad}})_{\ul{x}}\backslash \fL(G_{\on{ad}})_{\ul{x}}/\fL^+(G_{\on{ad}})_{\ul{x}}$$
be the local Hecke stack for $G_{\on{ad}}$ at $x$. 

\medskip

We have a natural projection
$$\on{Hecke}^{\on{loc}}_{G_{\on{ad}},\ul{x}} \to \on{Ge}_{Z_G}(\cD_{\ul{x}})\underset{\on{Ge}_{Z_G}(\ocD_{\ul{x}})}\times \on{Ge}_{Z_G}(\cD_{\ul{x}}).$$

The commutative group structure on the space of gerbes gives rise to a map\footnote{For a map of commutative group-objects $\CA\to \CB$, there is a 
canonical isomorphism $\CA\underset{\CB}\times \CA\simeq \CA\times \on{Fib}(\CA\to \CB)$.}
$$\on{Ge}_{Z_G}(\cD_{\ul{x}})\underset{\on{Ge}_{Z_G}(\ocD_{\ul{x}})}\times \on{Ge}_{Z_G}(\cD_{\ul{x}})\to 
\on{Ge}_{Z_G}(\cD_{\ul{x}})_{\ul{x}}.$$

Composing we obtain a map
\begin{equation} \label{e:loc Hecke to gerbes}
\on{Hecke}^{\on{loc}}_{G_{\on{ad}},\ul{x}} \to \on{Ge}_{Z_G}(\cD_{\ul{x}})_{\ul{x}}.
\end{equation}

\begin{rem}
Note that the map \eqref{e:loc Hecke to gerbes} induces a bijection on the sets of connected components when $G$
is simply-connected. 
\end{rem} 

\sssec{}

Let $\CG^{\on{loc}}$ be a $\BG_m$-gerbe on $\on{Ge}_{Z_G}(\cD_{\ul{x}})_{\ul{x}}$. By a slight abuse of notation,
we will denote by
$$\Sph(G_{\on{ad}})_{\ul{x},\CG^{\on{loc}}}$$ 
the corresponding twisted version of the category
$$\Dmod_{\frac{1}{2}}(\on{Hecke}^{\on{loc}}_{G_{\on{ad}},\ul{x}}),$$
obtained by pulling back the gerbe $\CG^{\on{loc}}$ along \eqref{e:loc Hecke to gerbes}.

\medskip

Assume now that $\CG^{\on{loc}}$ is \emph{multiplicative} (with respect to the group structure on $\on{Ge}_{Z_G}(\cD_{\ul{x}})_{\ul{x}}$).
Then the category $\Sph(G_{\on{ad}})_{\ul{x},\CG^{\on{loc}}}$ acquires a natural monoidal structure. 

\sssec{}

Let $\fG^{\on{loc}}_{\pi_1(\cG)}$ be a $\pi_1(\cG)$-gerbe on $\cD_{\ul{x}}$. Note that since the pairing 
\eqref{e:Verdier pairing Gamma loc} is bilinear, the corresponding $\BG_m$-gerbe $\CG_{\fG^{\on{loc}}_{\pi_1(\cG)}}$ on 
$\on{Ge}_{Z_G}(\cD_{\ul{x}})_{\ul{x}}$ has a natural multiplicative structure.

\medskip

The following is a twisted version of the (naive) geometric Satake functor:

\begin{lem}
There exists a monoidal functor
$$\Sat^{\on{nv}}_{G_{\on{ad}},\fG^{\on{loc}}_{\pi_1(\cG)}}:
\Rep(\cG_{\on{sc}})_{\ul{x},\fG^{\otimes -1}_{\pi_1(\cG)}}\to \Sph(G_{\on{ad}})_{\ul{x},\CG_{\fG^{\on{loc}}_{\pi_1(\cG)}}}.$$
\end{lem} 

\begin{proof}

The pairing \eqref{e:Verdier pairing Gamma loc} induces a bijection between:

\begin{itemize}

\item The set of characters of the (finite) group $\Maps(\cD_{\ul{x}},\pi_1(\cG))$, which is a subgroup of 
$$\Maps(\cD_{\ul{x}},\pi_1(\cG_{\on{ad}}))\simeq \Maps(\cD_{\ul{x}},Z_{\cG_{\on{sc}}})\simeq Z_{(\cG_{\on{sc}})_{\ul{x}}};$$


\item The set $\pi_0(\on{Ge}_{Z_G}(\cD_{\ul{x}})_{\ul{x}})$, which is a quotient of the set
$$\pi_0(\on{Hecke}^{\on{loc}}_{G_{\on{ad}},\ul{x}})\simeq \pi_0(\on{Ge}_{Z_{G_{\on{sc}}}}(\cD_{\ul{x}})_{\ul{x}}).$$

\end{itemize}

The assertion of the lemma follows from the fact that under the usual (naive)
geometric Satake functor
$$\Sat^{\on{nv}}_{G_{\on{ad}}}:\Rep(\cG_{\on{sc}})_{\ul{x}}\to \Sph(G_{\on{ad}})_{\ul{x}},$$
the decomposition of $\Rep(\cG_{\on{sc}})_{\ul{x}}$ according to central characters corresponds to
the decomposition of $\Sph(G_{\on{ad}})_{\ul{x}}$ along the connected components of 
$\on{Hecke}^{\on{loc}}_{G_{\on{ad}},\ul{x}}$,
according to support. Indeed, this observation implies that 
$\Sat^{\on{nv}}_{G_{\on{ad}}}:\Rep(\cG_{\on{sc}})_{\ul{x}}\to \Sph(G_{\on{ad}})_{\ul{x}}$ is equivariant for the
action of $B(Z_{(\cG_{\on{sc}})_{\ul{x}}})$ on both sides, and therefore
we can twist $\Sat^{\on{nv}}_{G_{\on{ad}}}$ by 
$Z_{(\cG_{\on{sc}})_{\ul{x}}}$-gerbes.

\end{proof}

\sssec{}

The constructions in \secref{sss:loc Hecke to gerbes} have immediate counterparts for the global Hecke stack
$$\on{Hecke}^{\on{glob}}_{G_{\on{ad}},\ul{x}}:=
\Bun_{G_{\on{ad}}}\underset{\Bun_{G_{\on{ad}}}(X-\ul{x})}\times \Bun_{G_{\on{ad}}}.$$

In particular, a multiplicative $\BG_m$-gerbe $\CG^{\on{loc}}$ on $\on{Ge}_{Z_G}(X)_{\ul{x}}$ gives rise to a monoidal category
\begin{equation} \label{e:twisted glob Hecke}
\Dmod_{\frac{1}{2}}(\on{Hecke}^{\on{glob}}_{G_{\on{ad}},\ul{x}})_{\CG^{\on{loc}}}.
\end{equation}

Note that pullback defines a monoidal functor
$$\Sph(G_{\on{ad}})_{\ul{x},\CG^{\on{loc}}}\to \Dmod_{\frac{1}{2}}(\on{Hecke}^{\on{glob}}_{G_{\on{ad}},\ul{x}})_{\CG^{\on{loc}}}.$$

\medskip

Assume now that $\CG^{\on{loc}}$ is obtained by restriction along
$$\on{Ge}_{Z_G}(X)_{\ul{x}}\to \on{Ge}_{Z_G}(X)$$
of a multiplicative $\BG_m$-gerbe $\CG$ on $\on{Ge}_{Z_G}(X)$. 

\medskip

Then we have a natural monoidal action of \eqref{e:twisted glob Hecke} on $\Dmod_{\frac{1}{2}}(\Bun_{G_{\on{ad}}})_\CG$.

\sssec{}

Combining the above ingredients, we obtain that for a $\pi_1(\cG)$-gerbe $\fG_{\pi_1(\cG)}$ on $X$, we have a monoidal action of 
$\Rep(\cG_{\on{sc}})_{\ul{x},\fG^{\otimes -1}_{\pi_1(\cG)}}$ on $\Dmod_{\frac{1}{2}}(\Bun_{G_{\on{ad}}})_{\CG_{\fG_{\pi_1(\cG)}}}$.

\medskip

This construction makes sense in families as $\ul{x}$ moves over the Ran space, thereby giving rise to the sought-for action of 
$\Rep(\cG_{\on{sc}})_{\Ran,\fG^{\otimes -1}_{\pi_1(\cG)}}$ on $\Dmod_{\frac{1}{2}}(\Bun_{G_{\on{ad}}})_{\CG_{\fG_{\pi_1(\cG)}}}$.

\medskip

The compatibility in point (a) of \propref{p:key local} follows by construction.

\qed[\propref{p:key local}(a)]

\ssec{Proof of \propref{p:key local}(b)} \label{ss:proof key local b} 

We will show how to adapt the proof of \cite[Theorem 4.5.2]{Ga4} to apply in the current gerbe-twisted situation. 

\sssec{}

Choose a point $x\in X$. Since $H^2_{\on{et}}(X-x,\pi_1(\cG))=0$, we can choose a trivialization of $\fG_{\pi_1(\cG)}$
over $X-x$. I.e., we can assume that $\fG_{\pi_1(\cG)}$ comes from an object
$$\fG^{\on{loc}}_{\pi_1(\cG)}\in \on{Ge}_{\pi_1(\cG)}(X)_x.$$

Recall that the space $\on{Ge}_{\pi_1(\cG)}(X)_x$ is canonically the discrete set $\pi_1(\cG)(-1)$ (see \eqref{e:Gamma gr}), 
which is in bijection with the set of characters of $Z_G$. Denote the element corresponding to our gerbe by $\chi$. 

\sssec{}

We claim now that the proof of the spectral decomposition theorem in \cite[Sect. 11.1]{Ga4}\footnote{The proof in 
{\it loc.cit.} relies on two ingredients that were not written down at the time, but are available now. One is
what is stated as \cite[Theorem 10.3.4]{Ga4}, for which the proof has been supplied in \cite[Sects. 15-16]{GLC2}. Another
is \cite[Proposition 11.1.3]{Ga4}, for which the proof has been supplies in \cite{FH} (in fact, as was observed by J.~Faergeman, 
this second ingredient is not even necessary).} applies in the current context.
Recall that in {\it loc. cit.} the proof was based on considering the localization functor
$$\Loc_{G_{\on{ad}}}:\KL(G_{\on{ad}})_{\crit,\Ran}\to \Dmod_{\frac{1}{2}}(\Bun_{G_{\on{ad}}}).$$

In the present twisted situation we will need to make the following modifications.

\sssec{}

We replace $\Ran$ by its relative version $\Ran_x$ that classifies finite subsets of $X$
that contain the point $x$. 

\medskip

For further discussion, in order to simplify the notation, we will work with a fixed element 
$\ul{x}\in \Ran_x$. Write $\ul{x}=\ul{x}'\sqcup \{x\}$.
 
\sssec{}

We replace the category 
$$\KL(G_{\on{ad}})_{\crit,\ul{x}}\simeq \KL(G_{\on{ad}})_{\crit,\ul{x}'}\otimes \KL(G_{\on{ad}})_{\crit,x}$$
by 
$$\KL(G_{\on{ad}})_{\crit,\ul{x},\chi}:=\KL(G_{\on{ad}})_{\crit,\ul{x}'}\otimes \KL(G_{\on{ad}})_{\crit,x,\chi},$$
where $\KL(G_{\on{ad}})_{\crit,x,\chi}$ is the full subcategory of $\KL(G)_{\crit,x}$, consisting of objects,
on which $$Z_G\subset \fL^+(G)_x$$ acts by the character $\chi$. 

\sssec{}

Note that by \eqref{e:loc vs global pairings}, the $\BG_m$-gerbe $\CG_{\fG_{\pi_1(\cG)}}$ on $\on{Ge}_{Z_G}(X)$ 
is obtained by pullback via the map
$$\on{Ge}_{Z_G}(X)\to \on{Ge}_{Z_G}(\cD_x)\simeq \on{Ge}_{Z_G}(\on{pt})\simeq B^2(Z_G)_{\on{et}}$$
from the $\BG_m$-gerbe on $B^2(Z_G)_{\on{et}}$ corresponding to the character $\chi$. 

\medskip

The pullback of $\CG_{\fG_{\pi_1(\cG)}}$ to $\Bun_{G_{\on{ad}}}$ trivializes over the cover
$$\Bun_{G_{\on{ad}}}\underset{\on{pt}/\fL^+(G_{\on{ad}})_x}\times \on{pt}/\fL^+(G)_x,$$
and corresponds to the multiplicative line bundle on $\on{pt}/Z_G$ given by $\chi$. 

\medskip

From here we obtain that we have a well-defined localization functor
$$\Loc_{G_{\on{ad}},\ul{x},\chi}:\KL(G_{\on{ad}})_{\crit,\ul{x},\chi}\to \Dmod_{\frac{1}{2}}(\Bun_{G_{\on{ad}}})_{\CG_{\fG_{\pi_1(\cG)}}}.$$

\sssec{}

Recall that the category $\KL(G_{\on{ad}})_{\crit,\ul{x}}$ is acted on by
$$\QCoh(\Op^{\on{mon-free}}_{\cG_{\on{sc}},\ul{x}}).$$

The key point in the proof of \cite[Theorem 4.5.2]{Ga4} is the fact (going back to \cite{BD} and proved in \cite[Sects. 15-16]{GLC2}) that the functor 
$$\Loc_{G_{\on{ad}},\ul{x}}:\KL(G_{\on{ad}})_{\crit,\ul{x}}\to \Dmod_{\frac{1}{2}}(\Bun_{G_{\on{ad}}})$$ factors as
\begin{multline*} 
\KL(G_{\on{ad}})_{\crit,\ul{x}}\to
\KL(G_{\on{ad}})_{\crit,\ul{x}}\underset{\QCoh(\Op^{\on{mon-free}}_{\cG_{\on{sc}},\ul{x}})}\otimes
\QCoh(\Op^{\on{mon-free}}_{\cG_{\on{sc}}}(X-x))\overset{\Loc^{\on{glob}}_{G_{\on{ad}},\ul{x}}}\longrightarrow \\
\to \Dmod_{\frac{1}{2}}(\Bun_{G_{\on{ad}}}),
\end{multline*}
where the functor $\Loc_{G_{\on{ad}},\ul{x}}^{\on{glob}}$ is $\QCoh(\LS_{\cG_{\on{sc}}})$-linear with respect to the tautological projection
$\Op^{\on{mon-free}}_{\cG_{\on{sc}}}(X-x)\to \LS_{\cG_{\on{sc}}}$.

\medskip

We will now explain the modification of this construction. 

\sssec{}

Consider the space $\Op^{\on{mon-free}}_{\cG,x}$. We claim that there is a canonically defined map
\begin{equation} \label{e:mon-free opers to Ge}
\Op^{\on{mon-free}}_{\cG,x} \to \on{Ge}_{\pi_1(\cG)}(\cD_x)_x.
\end{equation} 

To construct it, it suffices to show that the composition
\begin{equation} \label{e:Op to Ge}
\Op^{\on{mer}}_{\cG,x} \to \LS^{\on{mer}}_{\cG,x} \to \on{Ge}_{\pi_1(\cG)}(\ocD_x)
\end{equation} 
factors through the point of $\on{Ge}_{\pi_1(\cG)}(\ocD_x)$, corresponding to the trivial $\pi_1(\cG)$-gerbe,

\medskip

Note that the map \eqref{e:Op to Ge} factors as 
$$\Op^{\on{mer}}_{\cG,x} \to \LS^{\on{mer}}_{\cG,x} \to \Bun_\cG(\cD^\times_x)\to  \on{Ge}_{\pi_1(\cG)}(\ocD_x),$$
while the map
$$\Op^{\on{mer}}_{\cG,x} \to \LS^{\on{mer}}_{\cG,x} \to \Bun_\cG(\cD^\times_x)$$
in turn factors as 
$$\Op^{\on{mer}}_{\cG,x} \to \Bun_\cB(\cD^\times_x) \to \Bun_\cG(\cD^\times_x).$$

Hence,  \eqref{e:Op to Ge} factors as 
\begin{equation} \label{e:via torus}
\Op^{\on{mer}}_{\cG,x} \to \Bun_\cB(\cD^\times_x) \to \Bun_\cT(\cD^\times_x) \to \on{Ge}_{\pi_1(\cG)}(\ocD_x),
\end{equation} 
where we think of $\pi_1(\cG)$ as $\on{ker}(\cT\to \cT_{\on{sc}})$. 

\medskip

However, the map
$$\Op^{\on{mer}}_{\cG,x} \to \Bun_\cB(\cD^\times_x) \to \Bun_\cT(\cD^\times_x)$$
corresponds to the point $2\rhoch(\omega^{\otimes \frac{1}{2}})\in \Bun_\cT(\cD^\times_x)$,
and that point lifts to a point of $\Bun_{\cT_{\on{sc}}}(\cD^\times_x)$. This implies that the map
\eqref{e:via torus} factors via the trivial gerbe. 

\begin{rem}

We claim that $\Op^{\on{mon-free}}_{\cG,x}$ maps in fact to a twisted form of the affine Grassmannian
of the group $\cG$, so that \eqref{e:mon-free opers to Ge} factors via this map\footnote{This remark is inessential for the sequel
and the reader may choose to skip it.}.

\medskip 

Indeed, recall that for a fixed curve $X$, we can think of $\cG$-opers as connections of the standard form on a fixed
$\cG$-bundle $\CP^{\on{Op}}_\cG$, induced from a particular Borel bundle for a principal $SL_2$-triple, see \cite[Sect. 3.1.4]{GLC2}. 

\medskip

Consider the twisted affine Grassmannian $\Gr_{\cG,\CP^{\on{Op}}_\cG,x}$, 
i.e., the moduli space of pairs
$$(\CP_\cG,\alpha),$$
where $\CP_\cG$ is a $\cG$-bundle on $\cD_x$, and $\alpha$ is an isomorphism $\CP_\cG\simeq \CP^{\on{Op}}_\cG$
over $\cD^\times_x$. 

\medskip

Note that we can think of a point of $\Op^{\on{mon-free}}_{\cG,x}$ as a triple
$$(A,\CP_\cG,\alpha),$$
where:

\begin{itemize}

\item  $A$ is a connection of the standard oper form on $\CP^{\on{Op}}_\cG$ over $\cD^\times_x$;

\item $\CP_\cG$ is a $\cG$-bundle on $\cD_x$;

\item $\alpha$ is an isomorphism $\CP_\cG\simeq \CP^{\on{Op}}_\cG$ over $\cD^\times_x$, so that
the \emph{a priori meromorphic} connection on $\CP_\cG$, induced by $A$ via $\alpha$ is \emph{regular}.

\end{itemize}

The assignment 
$$(A,\CP_\cG,\alpha)\mapsto (\CP_\cG,\alpha)$$
is the sought-for map
$$\Op^{\on{mon-free}}_{\cG,x}\to \Gr_{\cG,\CP^{\on{Op}}_\cG,x}.$$

\end{rem}

\sssec{}

In the twisted situation we replace
$$\Op^{\on{mon-free}}_{\cG_{\on{sc}},\ul{x}}\simeq \Op^{\on{mon-free}}_{\cG_{\on{sc}},\ul{x}'}\times 
\Op^{\on{mon-free}}_{\cG_{\on{sc}},x}$$
by
$$\Op^{\on{mon-free}}_{\cG_{\on{sc}},\ul{x},\chi}:=
\Op^{\on{mon-free}}_{\cG_{\on{sc}},\ul{x}'}\times \Op^{\on{mon-free}}_{\cG_{\on{sc}},x,\chi},$$
where $\Op^{\on{mon-free}}_{\cG_{\on{sc}},x,\chi}$ is the preimage of the point $\chi$ under the projection \eqref{e:mon-free opers to Ge}. 

\medskip

Note that with respect to the $\QCoh(\Op^{\on{mon-free}}_{\cG,x})$-action on $\KL(G)_{\crit,x}$, we can identify the subcategory $\KL(G_{\on{ad}})_{\crit,x,\chi}$
with the direct summand that is supported over
$$\Op^{\on{mon-free}}_{\cG_{\on{sc}},x,\chi}\subset \Op^{\on{mon-free}}_{\cG,x}.$$

\sssec{}

Consider the space
$$\Op^{\on{mon-free}}_{\cG_{\on{sc}}}(X-x)_\chi:=
\Op_{\cG_{\on{sc}}}(X-x)\underset{\LS_{\cG_{\on{sc}}}(X-x)}\times \LS_{\cG_{\on{sc}},\fG_{\pi_1(\cG)}}.$$

Note that we have a naturally defined map
$$\Op^{\on{mon-free}}_{\cG_{\on{sc}}}(X-x)_\chi\to \Op^{\on{mon-free}}_{\cG_{\on{sc}},\ul{x},\chi}.$$

Now, by the same principle as in \cite[Theorem 10.3.4]{Ga4} (see \cite[Sects. 15-16]{GLC2} for a full argument), 
we obtain that the functor $\Loc_{G_{\on{ad}},\ul{x},\chi}$ factors as
\begin{multline*} 
\KL(G_{\on{ad}})_{\crit,\ul{x},\chi}\to
\KL(G_{\on{ad}})_{\crit,\ul{x},\chi}\underset{\QCoh(\Op^{\on{mon-free}}_{\cG_{\on{sc}},\ul{x},\chi})}\otimes
\QCoh(\Op^{\on{mon-free}}_{\cG_{\on{sc}}}(X-x)_\chi)\overset{\Loc^{\on{glob}}_{G_{\on{ad}},\ul{x},\chi}}\longrightarrow \\
\to \Dmod_{\frac{1}{2}}(\Bun_{G_{\on{ad}}})_{\CG_{\fG_{\pi_1(\cG)}}},
\end{multline*}
where the functor $\Loc^{\on{glob}}_{G_{\on{ad}},\ul{x},\chi}$ is linear with respect to
$\QCoh(\LS_{\cG_{\on{sc}},\fG_{\pi_1(\cG)}})$. 

\medskip

Now the argument parallel to that in \cite[Sect. 11.1]{Ga4} establishes the factorization of the action stated in \propref{p:key local}(b). 

\qed[\propref{p:key local}(b)]

\ssec{Geometric Langlands for non-pure inner forms}

\sssec{}

Note that, in view of Remark \ref{r:twisted groups}, from \corref{c:FM L} we obtain an expression of the twisted categories 
$\Dmod_{\frac{1}{2}}(\Bun_{G,\fG_{Z_G}})$ in terms of the usual $\Dmod_{\frac{1}{2}}(\Bun_{G})$ and the spectral action. 

\medskip

Namely, using \eqref{e:fib glob sect 1} and the fact that $\LS_\cG$ is 1-affine, we obtain:

\begin{cor} \label{c:twisted BunG}
For a $Z_G$-gerbe $\fG_{Z_G}$ on $X$, we have a canonical equivalence: 
$$\Dmod_{\frac{1}{2}}(\Bun_{G,\fG^{-1}_{Z_G}}) \simeq
\Dmod_{\frac{1}{2}}(\Bun_{G})\underset{\QCoh(\LS_\cG)}\otimes \QCoh(\LS_\cG)_{\CG_{\fG_{Z_G}}},$$
where $\QCoh(\LS_\cG)_{\CG_{\fG_{Z_G}}}$ is the twist of $\QCoh(\LS_\cG)$ by the pullback of the gerbe $\CG_{\fG_{Z_G}}$
on $\on{Ge}_{\pi_1(\cG)}(X)$ along the map $\sfp_{\pi_1(\cG)}$ of \eqref{e:LS to gerbe}. 
\end{cor} 

Combining with GLC for $G$, we obtain a form of GLC for non-pure inner twists: 

\begin{cor} \label{c:twisted GLC 1}
There is a canonical equivalence: 
$$\Dmod_{\frac{1}{2}}(\Bun_{G,\fG^{-1}_{Z_G}}) \simeq \IndCoh_\Nilp(\LS_\cG)
\underset{\QCoh(\LS_\cG)}\otimes \QCoh(\LS_\cG)_{\CG_{\fG_{Z_G}}}.$$
\end{cor} 

\sssec{}

Let $G_{\on{sc}}$ be the simply-connected cover of $G$; consider the short exact sequence
$$1\to \pi_1(G)\to G_{\on{sc}}\to G\to 1$$
and the resulting map
$$\sfp_{\pi_1(G)}:\Bun_G\to \on{Ge}_{\pi_1(G)}(X).$$

For a $\BG_m$-gerbe $\CG$ on $\on{Ge}_{\pi_1(G)}(X)$, let us denote by $\Dmod_{\frac{1}{2},\CG}(\Bun_G)$ 
the corresponding category of gerbe-twisted D-modules on $\Bun_G$. 

\sssec{}

On the dual side we have the short exact sequence
$$1\to Z_{\cG}\to \cG\to \cG_{\on{ad}}\to 1,$$
and a map
$$\sfp_{Z_\cG}:\LS_{\cG_{\on{ad}}}\to \on{Ge}_{Z_{\cG}}(X).$$

To a point $\fG_{Z_{\cG}}\in  \on{Ge}_{Z_\cG}(X)$ we can associate a 
gives rise to a (non-pure) inner twist of $\LS_\cG$: 
$$\LS_{\cG,\fG_{Z_{\cG}}}:=\LS_{\cG_{\on{ad}}}\underset{\on{Ge}_{Z_{\cG}}(X)}\times \{\fG_{Z_{\cG}}\}.$$

\sssec{}

Consider the short exact sequence
$$0\to \pi_1(G)\to Z_{G_{\on{sc}}}\to Z_G\to 0$$
and its dual
$$0\to \pi_1(\cG)\to \pi_1(\cG_{\on{ad}})\to Z_\cG\to 0.$$

Combining \thmref{t:FM L} and \lemref{l:sgrp} we obtain:

\begin{cor} \label{c:twisted LS}
For a $Z_\cG$-gerbe $\fG_{Z_{\cG}}$ on $X$, there is a canonical equivalence
$$\Dmod_{\frac{1}{2}}(\Bun_G)_{\CG_{\fG_{Z_{\cG}}}} \simeq \Dmod_{\frac{1}{2}}(\Bun_{G_{\on{sc}}}) \underset{\QCoh(\LS_{\cG_{\on{ad}}})}\otimes
\QCoh(\LS_{\cG,\fG_{Z_{\cG}}}).$$
\end{cor}

Combining with GLC for $G_{\on{sc}}$ we obtain:

\begin{cor}  \label{c:twisted GLC 2}
For $\fG_{Z_{\cG}}\in  \on{Ge}_{Z_\cG}(X)$ there is a canonical equivalence
$$\Dmod_{\frac{1}{2}}(\Bun_G)_{\CG_{\fG_{Z_{\cG}}}} \simeq \IndCoh_\Nilp(\LS_{\cG,\fG_{Z_{\cG}}}).$$
\end{cor}

\begin{rem}

We expect that equivalences parallel to Corollaries \ref{c:twisted GLC 1} and \ref{c:twisted GLC 2} 
also take place in the framework of Fargues-Scholze theory of \cite{FS}.

\end{rem}

\ssec{An arithmetic variant}

The contents of this subsection are not used elsewhere in the paper, and and be considered
as a collection of informal remarks. Our goal here is to explain that Corollaries 
\ref{c:twisted GLC 1} and \ref{c:twisted GLC 2} are not just category-theoretic, but have 
meaningful implications for automorphic functions.

\medskip

In this subsection we will appeal to the notations introduced in \cite[Sect. 24]{AGKRRV}. We will work
over the ground field $k=\ol\BF_p$, and we will assume that $Z_G$ and $Z_\cG$ have orders prime to $p$. 

\sssec{}

In this subsection we will assume that GLC holds (for constant group-schemes) in the context of $\ell$-adic sheaves
over the ground field $\ol\BF_p$, see \cite[Conjecture 21.2.7]{AGKRRV}:
\begin{equation} \label{e:restricted GLC}
\Shv_{\Nilp,\frac{1}{2}}(\Bun_G)\simeq \IndCoh_\Nilp(\LS_\cG^{\on{restr}}).
\end{equation} 

Then by the same principle as in \corref{c:twisted GLC 1}, from \eqref{e:restricted GLC} one can derive a 
GLC-type equivalence for non-pure inner forms of $G$:

\medskip

For a $Z_G$-gerbe $\fG_{Z_G}$ on $X$, we have 
\begin{equation} \label{e:restricted GLC 1}
\Shv_{\Nilp,\frac{1}{2}}(\Bun_{G,\fG^{-1}_{Z_G}}) \simeq \IndCoh_\Nilp(\LS^{\on{restr}}_\cG)
\underset{\QCoh(\LS^{\on{restr}}_\cG)}\otimes \QCoh(\LS^{\on{restr}}_\cG)_{\CG_{\fG_{Z_G}}},
\end{equation}
where we view $\CG_{\fG_{Z_G}}$ as naturally a $\mu_\infty(\ol\BQ_\ell)$-gerbe on $\on{Ge}_{\pi_1(\cG)}$, and we turn it into a $\BG_m$-gerbe 
via
$$\mu_\infty(\ol\BQ_\ell)\subset \BG_m.$$

\sssec{}

Similarly, for a $Z_\cG$-gerbe $\fG_{Z_{\cG}}$ on $X$, combining \corref{c:twisted GLC 2} and \eqref{e:restricted GLC}, we obtain:
\begin{equation} \label{e:restricted GLC 2}
\Shv_{\Nilp,\frac{1}{2}}(\Bun_G)_{\CG_{\fG_{Z_{\cG}}}} \simeq \IndCoh_\Nilp(\LS^{\on{restr}}_{\cG,\fG_{Z_{\cG}}}),
\end{equation}
where we view $\CG_{\fG_{Z_{\cG}}}$ as a $\mu_\infty(\ol\BQ_\ell)$-gerbe on $\on{Ge}_{\pi_1(G)}$.

\sssec{}

Assume now that $X$ and $G$ are defined over $\BF_q$. Let $\fG_{Z_G}$ be a $Z_G$-gerbe that is also
defined over $\BF_q$, and hence so is the stack $\Bun_{G,\fG^{-1}_{Z_G}}$. Thus, we an consider the 
corresponding space of automorphic functions
$$\on{Funct}_c(\Bun_{G,\fG^{-1}_{Z_G}}(\BF_q),\ol\BQ_\ell).$$

\medskip

The Frobenius-equivariant structure on $\fG_{Z_G}$ gives rise to a Frobenius-equivariant structure
on the $\BG_m$-gerbe $\sfp_{\pi_1(\cG)}^*(\CG_{\fG_{Z_G}})$ over $\LS^{\on{restr}}_\cG$. 
Hence, the restriction of $\sfp_{\pi_1(\cG)}^*(\CG_{\fG_{Z_G}})$ to
$$\LS_\cG^{\on{arithm}}:=(\LS^{\on{restr}}_\cG)^{\on{Frob}}$$
gives rise to a line bundle on $\LS_\cG^{\on{arithm}}$, to be denote $\CL_{\CG_{\fG_{Z_G}}}$. 

\medskip

As in \cite[Conjecture 24.8.6]{AGKRRV}, applying the categorical trace of Frobenius to the two sides of \eqref{e:restricted GLC 1}, 
we obtain an isomorphism of vector spaces
\begin{equation} \label{e:arithm 1}
\on{Funct}_c(\Bun_{G,\fG^{-1}_{Z_G}}(\BF_q),\ol\BQ_\ell) \simeq
\Gamma(\LS_\cG^{\on{arithm}},\omega_{\LS_\cG^{\on{arithm}}}\otimes \CL_{\CG_{\fG_{Z_G}}}).
\end{equation} 

Thus, \eqref{e:arithm 1} is an expression for the (spherical) automorphic category 
for a (non-pure) inner form of $G$.

\sssec{}

Let now $\fG_{Z_\cG}$ be a $Z_\cG$-gerbe on $X$ defined over $\BF_q$. Then the stack 
$\LS^{\on{restr}}_{\cG,\fG_{Z_{\cG}}}$ acquires an action of the Frobenius automorphism.
Denote
$$\LS^{\on{arithm}}_{\cG,\fG_{Z_{\cG}}}:=\left(\LS^{\on{restr}}_{\cG,\fG_{Z_{\cG}}}\right)^{\on{Frob}}.$$

The $\ol\BQ_\ell^\times$-gerbe $\sfp^*_{\pi_1(G)}(\CG_{\fG_{Z_{\cG}}})$ on
$\Bun_G$ gives rise to a $\ol\BQ_\ell^\times$-torsor over $\Bun_G(\BF_q)$, to be denoted $\CP_{\fG_{Z_{\cG}}}$.
Given this torsor, consider the corresponding space of twisted compactly-supported functions
$$\on{Funct}_c(\Bun_G(\BF_q))_{\CP_{\fG_{Z_{\cG}}}}:=
\on{Sect}_c(\Bun_G(\BF_q),\CP_{\fG_{Z_{\cG}}}\overset{\ol\BQ_\ell^\times}\times \ol\BQ_\ell),$$

\medskip

As in \cite[Conjecture 24.8.6]{AGKRRV}, applying the categorical trace of Frobenius to the two sides of \eqref{e:restricted GLC 2}, 
we obtain an isomorphism of vector spaces
\begin{equation} \label{e:arithm 2}
\on{Funct}_c(\Bun_G(\BF_q))_{\CP_{\fG_{Z_{\cG}}}}\simeq 
\Gamma(\LS_{\cG,\fG_{Z_{\cG}}}^{\on{arithm}},\omega_{\LS_\cG^{\on{arithm}}}).
\end{equation} 

Thus, \eqref{e:arithm 2} is an expression for the \emph{metaplectic} (spherical) automorphic category 
of $G$, which is given in terms of the (non-pure) inner twist of $\cG$. 

\appendix

\section{Review of (semi-)stability for \texorpdfstring{$G$}{stb}-bundles}\label{s:stable}

We briefly review the
basic notions from the theory of (semi-)stable $G$-bundles
following the original source \cite{Ra}.

\subsection{Definition of (semi-)stability}

In what follows, we only consider parabolic subgroups $P$ 
containing our fixed Borel. 

\subsubsection{Notation related to root data}

Recall that $\Lambda$ denotes the coweight lattice of $G$ and 
$\check{\Lambda}$ denotes the weight lattice. 
As is standard, $2\rhoch_G \in \check{\Lambda}$ denotes the sum of the positive
roots. 

\medskip 

Let $P$ be a parabolic subgroup of $G$ with Levi quotient $M$.
As the weight lattices of $M$ and $G$ coincide, 
we also have the weight $2\rhoch_M \in \check{\Lambda}$.
We set $2\rhoch_P := 2\rhoch_G -2\rhoch_M$, which is the sum of the roots
occurring in $\fn(P)$. 

\medskip 

We remark that $\langle 2\rhoch_P,\alpha_i\rangle = 0$ for 
every vertex $i$ in the Dynkin diagram $I_M$ of $M$, i.e., for each
simple coroot $\alpha_i$ whose $\mathfrak{sl}_2$ maps into $\fm$. Indeed,
we have $\langle 2\rhoch_G,\alpha_i\rangle = 
\langle 2\rhoch_M,\alpha_i\rangle = 2$ for such $\alpha_i$.

\medskip 

It follows that for a coweight $\lambda \in \Lambda$, the value of 
$\langle 2\rhoch_P,\lambda\rangle$ only depends on the class of 
$\lambda$ in $\Lambda/\on{Span}\{\alpha_i\}_{i \in I_M} =: 
\pi_{1,\on{alg}}(M)$.

\subsubsection{}

We have the following definition (cf. \cite{Ra}):

\begin{defin}

A $G$-bundle $\CP_G$ on $X$ is \emph{semi-stable} (resp. \emph{stable})
if for every maximal (proper) parabolic subgroup 
$P \subsetneq G$ and every reduction
$\CP_P$ of $\CP_G$ to $P$, we have
\[
\langle 2\rhoch_P,\deg(\CP_P)\rangle \leq 0 \,\,\,\, \text{(resp. $< 0$)}. 
\]

\noindent Here we remind that $\deg(\CP_P)$ is an element of 
$\pi_{1,\on{alg}}(M)$.

\end{defin}

We remark that the integer $\langle 2\rhoch_P,\deg(\CP_P)\rangle$ appearing
above is the degree of the vector bundle $\fn(P)_{\CP_P}$ on $X$.

\begin{example}

This definition is rigged to recover the usual one for 
$G = GL_n$. 

\medskip

Indeed, suppose $\CE$ has rank $n$ and $P$ is the maximal parabolic
whose reductions correspond to subbundles $\CE_0 \subset \CE$ of rank $m$.
Then a straightforward calculation yields
\[
\langle 2\rhoch_P,\deg(\CP_P)\rangle = 
\on{rank}(\CE)\cdot\deg(\CE_0) - \on{rank}(\CE_0)\cdot\deg(\CE).
\]

\end{example}

\subsection{A characterization of (semi-)stability}

\subsubsection{}

We have the following basic result.

\begin{prop} \label{p:semistab}

For a $G$-bundle $\CP_G$ on $X$, the following conditions are equivalent.

\medskip 

\noindent{\em(a)} $\CP_G$ is semi-stable (resp. stable).

\noindent{\em(b)} For every proper parabolic subgroup $P \subsetneq G$ (possibly
not of corank 1) and every
reduction $\CP_P$ of $\CP_G$ to $P$, we have
\[
\langle 2\rhoch_P,\deg(\CP_P)\rangle \leq 0 \,\,\,\, \text{(resp. $< 0$)}
\]

\noindent{\em(c)} For every reduction $\CP_B$ of $\CP_G$ to the
Borel, we have:
\begin{equation}\label{eq:ss borel}
\deg(\CP_B) = \sum_{i \in I_G} n_i\alpha_i + \varepsilon, 
\,\,\, n_i \in \BQ^{\leq 0}, 
\,\, \varepsilon \in \BQ \cdot \Lambda_{Z_G}
\text{(resp. $n_i \in \BQ^{<0}$)}.
\end{equation}
Here $\Lambda_{Z_G}$ is the set of coweights mapping
into the center of $G$.

\end{prop}

\sssec{Proof of \propref{p:semistab}}

The key point is to observe
\begin{equation}\label{eq:rho_p}
\begin{cases}
\langle 2\rhoch_P,\varepsilon\rangle = 0, & \\
\langle 2\rhoch_P,\alpha_i\rangle = 0 & \text{ if $i \in I_M$} \\
\langle 2\rhoch_P,\alpha_i\rangle \geq 2 & \text{ if $i \not\in I_M$} \\
\end{cases}
\end{equation}

\noindent where the last expression follows as
$\langle 2\rhoch_G,\alpha_i\rangle = 2$ and 
$2\rhoch_M$ is a sum of roots $\check{\alpha}_j$ with 
$j \in I_M$ (so $j \neq i$). Then (b) tautologically implies
(a), (c) implies (b) by \eqref{eq:rho_p}, and
(a) implies (c) again by noting that for $P_i$ the 
maximal parabolic corresponding to $i \in I_G$, 
we have $\langle 2\rhoch_{P_i},\deg(\CP_B)\rangle = 
n_i \langle 2\rhoch_P,\alpha_i\rangle \in n_i \cdot \BZ^{>0}$
by \eqref{eq:rho_p}.

\qed[\propref{p:semistab}]

\sssec{}

In the above, the condition (c) immediately matches the notion of
semi-stability used in \cite{DG} (see \emph{loc. cit}. Lemma 7.3.2),
and it matches the notion of stability implicitly suggested in \emph{loc. cit}.

\end{document}